\documentclass[a4paper,12pt]{article}

\usepackage{amsmath}
\usepackage{ntheorem}
\usepackage{amssymb}
\usepackage{latexsym}
\usepackage{url,mathrsfs}
\usepackage[all]{xy}\UseComputerModernTips
\usepackage{graphicx}

\theoremstyle{plain}
\newtheorem{proposition}{Proposition}[section]
\newtheorem{theorem}[proposition]{Theorem}
\newtheorem{lemma}[proposition]{Lemma}
\newtheorem{corollary}[proposition]{Corollary}

\theoremstyle{plain}
\theorembodyfont{\normalfont\rm}
\newtheorem{definition}[proposition]{Definition}
\newtheorem{remark}[proposition]{Remark}
\newtheorem{example}[proposition]{Example}

\theoremstyle{nonumberplain}
\theoremheaderfont{\normalfont\itshape}
\theoremseparator{.}
\theorembodyfont{\normalfont}
\newtheorem{proof}{Proof}

\newcommand{\qed}{\hfill $\Box$}

\makeatletter

\@addtoreset{equation}{section}
\makeatother

\newcommand{\mer}{{\rm mero}}
\newcommand{\merc}{{\rm mero,c}}
\newcommand{\LK}{\mathrm{lk}}
\newcommand{\MCS}{\mathcal{S}}

\newcommand{\GR}{\mathrm{Gr}}

\newcommand{\ZZ}{{\mathbb Z}}

\newcommand{\RR}{{\mathbb R}}
\newcommand{\CC}{{\mathbb C}}
\newcommand{\PP}{{\mathbb P}}
\newcommand{\QQ}{{\mathbb Q}}
\newcommand{\LL}{{\mathbb L}}

\renewcommand{\d}{{\rm dim}}

\newcommand{\Vol}{{\rm Vol}}

\newcommand{\e}{\varepsilon}
\renewcommand{\SS}{{\mathcal S}}
\newcommand{\Spec}{{\rm Spec}}
\newcommand{\Con}{{\rm Con}}

\newcommand{\id}{{\rm id}}
\newcommand{\supp}{{\rm supp}}

\newcommand{\Int}{{\rm Int}}

\newcommand{\Var}{{\rm Var}}
\newcommand{\HSm}{{\rm HS}^{\rm mon}}

\newcommand{\relint}{{\rm rel.int}}

\newcommand{\Cone}{{\rm Cone}}

\newcommand{\dist}{{\rm dist}}
\newcommand{\height}{{\rm ht}}
\newcommand{\Gr}{{\rm Gr}}

\newcommand{\Db}{{\bf D}^{b}}
\newcommand{\Dbc}{{\bf D}_{c}^{b}}

\newcommand{\KK}{{\rm K}}

\newcommand{\F}{{\cal F}}
\newcommand{\G}{{\cal G}}

\renewcommand{\L}{{\mathcal L}}
\newcommand{\M}{{\mathcal M}}
\renewcommand{\O}{{\cal O}}

\newcommand{\CF}{{\rm CF}}
\newcommand{\J}{{S}}

\renewcommand{\sp}{{\rm sp}}

\newcommand{\tl}[1]{\widetilde{#1}}

\newcommand{\simto}{\overset{\sim}{\longrightarrow}}

\newcommand{\dsum}{\displaystyle \sum}

\newcommand{\dprod}{\displaystyle \prod}

\renewcommand{\(}{\left(}
\renewcommand{\)}{\right)}

\newcommand{\longhookrightarrow}{\DOTSB\lhook\joinrel\longrightarrow}

\def\to{\mathchoice{\longrightarrow}{\rightarrow}{\rightarrow}{\rightarrow}}

\def\cf{cf.\kern.3em}
\def\eg{e.g.\kern.3em}

\title{Geometric Monodromies, Mixed Hodge Numbers 
of Motivic Milnor Fibers and 
Newton Polyhedra 
}

\author{Kiyoshi TAKEUCHI 
\footnote{Mathematical Institute, Tohoku University, 
Aramaki Aza-Aoba 6-3, Aobaku, Sendai, 980-8578, Japan. 
E-mail: takemicro@nifty.com} 
}

\date{}

\begin{document}

\maketitle

\begin{abstract}
We introduce the theory of local and global 
monodromies of polynomials in cohomology groups 
in various geometric situations, focusing  
on its relations with toric geometry and 
motivic Milnor fibers, and moreover in the modern languages 
of nearby and vanishing cycle functors. 
Equivariant mixed Hodge numbers 
of motivic Milnor fibers will be described in 
terms of Newton polyhedra of polynomials. 
\end{abstract}

\maketitle

\section{Introduction}\label{section 1}

Since Milnor's discovery of Milnor fibers of 
complex hypersufaces in \cite{Milnor}, many 
mathematicians in different fields, such as 
singularity theory, algebraic geometry, topology, 
complex analysis, number theory and D-modules, 
studied them from various points of view. 
Moreover, after the two fundamental papers 
\cite{Broughton} and \cite{S-T-1}, globalizing 
the situation, they started to study also the singularities and 
the monodromies at infinity of polynomial maps. 
The results thus obtained are surprisingly similar 
to the ones obtained in the local case. 
Here let us call these monodromies 
in the both local and global cases ``geometric 
monodromies" for short, in order to distinguish 
them from the ones of general constructible 
sheaves. The aim of this note is to introduce the theory of 
such geometric monodromies, laying special emphasis 
on its relations with toric geometry and 
motivic Milnor fibers, and moreover in the modern languages 
of nearby and vanishing cycle functors. 

To this end, we try to explain all the basic definitions and 
results on nearby and vanishing 
cycle functors in a comprehensible way and 
show their geometric applications at the same time 
as many as possible. By considering also the 
superstructures of mixed Hodge modules 
of nearby cycle complexes, 
in \cite{D-L-1} Denef and Loeser introduced 
motivic Milnor fibers, which can be regarded  
as the motivic reincarnations of the classical 
Milnor fibers. This motivic point of view is 
very important in the recent progress of 
the theory of geometric monodromies. In the 
last half of this note, 
we introduce the theory of 
motivic Milnor fibers at infinity developed in Matsui-Takeuchi 
\cite{M-T-3} and Raibaut \cite{Raibaut}, 
\cite{Raibaut-1.5}, \cite{Raibaut-1.8}. 
Combining it with some results in toric geometry 
such as the Bernstein-Khovanskii-Kouchnirenko 
theorem in \cite{Khovanskii} and the results in 
\cite[Section 2]{M-T-3}, in the both local and 
global cases, one can now describe the 
equivariant mixed Hodge numbers 
of motivic Milnor fibers 
in terms of Newton polyhedra of polynomials. 
In this way, the readers will see 
how the combinatorial 
expressions of the Jordan normal forms of various 
geometric monodromies are obtained.

The organization of this note is as follows. In Section \ref{section 2} 
we first recall the classical definitions of 
Milnor fibers and Milnor monodromies and explain their relations with 
the nearby and vanishing cycle functors introduced by 
Deligne \cite{Deligne}. Then we show how the monodromy zeta 
functions associated to nearby cycle complexes of constructible 
sheaves are calculated via resolutions of singularities. 
In Section \ref{section 3} we combine this method with 
some notions and results in toric geometry, such as 
toric modifications and the Bernstein-Khovanskii-Kouchnirenko 
theorem. Specifically, we thus prove 
the celebrated theorems of Kouchnirenko, Varchenko and Oka 
which allow us to express the monodromy zeta 
functions explicitly 
in terms of Newton polyhedra associated to 
polynomials. Here as in \cite{M-T-new2}, we 
prove them by using the theory of 
nearby cycle functors explained in Section \ref{section 2}. 
By the clarity of its expression and its applicability 
to generic situations, 
Varchenko's theorem is extremely useful in 
singularity theory, especially in the study of 
the monodromy conjecture proposed in 
Denef-Loeser \cite{D-L}. For the applications of 
Varchenko's theorem in this direction, 
see e.g. \cite{E-L-T}, \cite{L-P-S}, 
\cite{L-V}. In Section \ref{section 4} 
we take the same strategy to introduce some basic results on 
singularities and monodromies at infinity of polynomial maps. 
In particular, here we focus on 
N{\'e}methi-Zaharia's theorem in \cite{N-Z} and 
Libgober-Sperber's one in \cite{L-S}. Note that Libgober-Sperber's 
theorem can be considered as a global counterpart 
of Varchenko's one. In Sections \ref{section 5} 
and \ref{section 6}, we prepare several notions and 
results which will be used in subsequent sections. 
Then in Section \ref{section 7}, we explain the theory of 
motivic Milnor fibers at infinity developed in Matsui-Takeuchi 
\cite{M-T-3} and Raibaut \cite{Raibaut}, 
\cite{Raibaut-1.5}, \cite{Raibaut-1.8} independently.  
Combining it with the results in Sections \ref{section 5} and \ref{section 6}, 
then we see how the combinatorial 
expressions in \cite{M-T-3} of the Jordan normal forms of 
monodromies at infinity are obtained. In the course of their proofs, 
we obtain also precise descriptions of the 
equivariant mixed Hodge numbers of motivic 
Milnor fibers at infinity. 
As we will see in Section \ref{section 8}, similar results 
hold true also for (local) Milnor monodromies and 
(local) motivic Milnor fibers. For the further 
applications of these methods, see \cite{C-R}, 
\cite{E-T}, \cite{N-T}, \cite{Saito} and 
\cite{T-T}. 
In Section \ref{section 9} we introduce the theory of 
Katz-Stapledon \cite{Ka-St-1}, \cite{Ka-St-2}  and 
Stapledon \cite{Stapledon} which impoved 
the results in \cite{E-T}, \cite{M-T-3}, \cite{M-T-1}, 
\cite{T-T}. Remarkably, by using also 
the theories of intersection cohomology, tropical geometry, 
mirror symmetry and combinatorics, 
in \cite{Stapledon} Stapledon obtained 
the full combinatorial expressions 
of the Jordan normal forms of various geometic 
monodromies. Finally in Section \ref{section 10}, 
we show how these methods can be applied partially 
to Milnor fibers and Milnor monodromies of 
meromorphic functions. For this purpose, we 
explain some basic properties of 
the meromorphic nearby and 
vanishing cycle functors introduced in 
\cite{N-T} and \cite{Raibaut-2}.

\bigskip
\noindent{\bf Acknowledgement}: The author 
is very grateful to Professor Pepe Seade and his colleagues
for suggesting him to write this expository paper 
in Volume VI of the Handbook of Geometry and Topology of Singularities. 
He also thanks the referee for many valuable suggestions 
and comments which 
improved this paper.

\section{Milnor fibers and nearby cycle sheaves}\label{section 2}

In this section, we recall the classical definitions of 
Milnor fibers and Milnor monodromies and explain their relations with 
the nearby and vanishing cycle functors introduced by 
Deligne \cite{Deligne}. 
In this note, we essentially follow the terminology of \cite{Dimca}, \cite{H-T-T}, 
\cite{K-book}, \cite{K-S} and \cite{Sch}. 
For example, for a topological space $X$ we denote by $\Db(X)$ the derived category whose 
objects are bounded complexes of sheaves of $\CC_X$-modules on $X$. For a complex 
analytic space $X$, let $\Dbc(X)$ be the full subcategory of $\Db(X)$ consisting 
of constructible complexes of sheaves. In what follows, let 
$X$ be a complex analytic space and $f\colon X \longrightarrow \CC$ 
a non-constant holomorphic function on it. The problem being local, 
we may assume that $X$ is a complex analytic subset of 
a complex manifold $M$ and for a non-constant holomorphic 
function $\widetilde{f}\colon M \longrightarrow \CC$ on it  
we have $\widetilde{f}\big|_{X} \equiv f$. Set 
$X_0:=f^{-1}(0)\subsetneqq X$. First of all, the following fundamental result was proved by 
Milnor \cite{Milnor} and L\^e \cite{Le}. See also Massey 
\cite[Proposition 1.3]{M}. 

\begin{theorem}\label{thm.7.2.001}
Let $X\subset\CC^N$ be a complex analytic subset 
of $\CC^N$ containing the origin $0\in\CC^N$ and 
$f\colon X \longrightarrow \CC$ a non-constant holomorphic 
function on it satisfying the condition $f(0)=0$. 
Then for $0<\eta\ll\varepsilon\ll 1$ the 
restriction 
\begin{equation}
X\cap B(0;\varepsilon)\cap f^{-1}(B(0;\eta)^*) \longrightarrow B(0;\eta)^*\subset\CC^*
\end{equation}
of $f$ is a fiber bundle, where we set $B(0;\eta)^*:= 
\{z\in\CC \mid 0<|z|<\eta \} \subset \CC^*$ and $B(0;\varepsilon):= 
\{x=({x}_{1},\dots,x_{N})\in\CC^N \mid ||x||=\sqrt{\sum_{i=1}^{N} |x_i|^2}<
\varepsilon \} \subset \CC^N$. 
\end{theorem}

\begin{definition}\label{def.7.2.002}
In the situation of Theorem \ref{thm.7.2.001}, 
for a point $z\in B(0;\eta)^*$ we call 
\begin{equation}
F_0 := X \cap B(0;\varepsilon) \cap f^{-1}(z) 
\end{equation}
the Milnor-L\^e fiber (or simply the Milnor fiber) 
of $f$ at the origin 
$0\in X_0=f^{-1}(0)\subset X$. 
\end{definition}

In the situation of Theorem \ref{thm.7.2.001}, 
for any $j \in \ZZ$ the $j$-th higher direct image 
sheaf of the constant sheaf on 
$X\cap B(0;\varepsilon)\cap f^{-1}(B(0;\eta)^*)$ by $f$ 
is a local system on the punctured disk 
$B(0;\eta)^* \subset \CC^*$. Then 
by turning around the origin $0 \in \CC$ in 
$B(0;\eta)^*$ we obtain 
automorphisms of their stalks, i.e. 
$\CC$-linear isomorphisms 
\begin{equation}
\Phi_{j,0} \colon H^j(F_0; \CC ) \simto H^j(F_0; \CC ) \qquad (j\in\ZZ). 
\end{equation}
We call them the Milnor-L\^e monodromies (or simply 
the Milnor monodromies) of $f$ at the origin 
$0\in X_0=f^{-1}(0)\subset X$. It is clear 
that the Milnor fiber and the Milnor monodromies 
are defined similarly 
also at any point $x \in X_0=f^{-1}(0)\subset X$ of 
$X_0$. We denote them $F_x$ and 
$\Phi_{j,x}$ ($j\in\ZZ$) respectively. If 
$X= \CC^n$ and the complex hypersurface $X_0=f^{-1}(0)\subset X$ 
of it has an isolated singular point at the 
origin $0\in X_0=f^{-1}(0)\subset X$, we have 
the following beautiful theorem due to Milnor \cite{Milnor} 

\begin{theorem}\label{thm.7.2.003}
{\rm (Milnor \cite{Milnor})} Let 
$f\colon X=\CC^n \longrightarrow \CC \ (n \ge 2)$ 
be a non-constant holomorphic function on 
$X= \CC^n$ satisfying the condition $f(0)=0$ and 
assume that the complex hypersurface $X_0=f^{-1}(0)\subset X$ 
has an isolated singular point at the 
origin $0\in X_0=f^{-1}(0)\subset X$. 
Then the Milnor fiber $F_0$ of $f$ at the origin 
is homotopy equivalent to the bouquet of some 
$(n-1)$-dimensional spheres $S^{n-1}$: 
\begin{equation}
F_0  \sim  S^{n-1}\vee\cdots\vee S^{n-1}.
\end{equation}
In particular, for its cohomology groups 
$H^j(F_0; \CC )$ ($j \in \ZZ$) we have 
\begin{equation}
H^j(F_0; \CC ) \simeq 
\begin{cases}
\ \CC & (j=0) \\
& \\
\ \CC^{\mu} & (j=n-1) \\
 & \\
\ 0 & (j \neq 0,n-1), 
\end{cases}
\end{equation}
where $\mu>0$ is the number of spheres $S^{n-1}$ in 
the above bouquet decomposition. 
\end{theorem}

We call the number $\mu >0$ in Theorem \ref{thm.7.2.003} 
the Milnor number of $f$ (or of $X_0=f^{-1}(0)$) 
at the origin $0\in X_0=f^{-1}(0)\subset X$. 
Note that the Milnor fiber $F_0 \subset \CC^n$ in 
Theorem \ref{thm.7.2.003} is an  
$(n-1)$-dimensional complex manifold i.e. a 
$(2n-2)$-dimensional real analytic manifold. 
Theorem \ref{thm.7.2.003} asserts that it 
has the homotopy type of the real 
$(n-1)$-dimensional topological space 
$S^{n-1}\vee\cdots\vee S^{n-1}$. 
If $X_0=f^{-1}(0) \subset X= \CC^n$ for $f: X= \CC^n \longrightarrow \CC$ 
does not have an isolated singular point at the 
origin $0\in X_0=f^{-1}(0)$, 
it is in general very hard to determine  
the Betti numbers of the Milnor fiber $F_0$.  
Nevertheless, making use of Morse theory 
L\^e~\cite{Le-1} and Massey~\cite{Mas2} obtained 
the handle decompositions of $F_0$. Namely 
we can know at least the homotopy type of it. 
From now, we shall explain how 
Milnor fibers and their monodromies can be effectively 
studied by the more sophisticated machineries 
of constructible sheaves. 
Let $i\colon X_0 = f^{-1}(0) \hookrightarrow X$
and $j\colon X\setminus X_0 \hookrightarrow X$ be 
the inclusion maps. For $\CC^* := \CC \setminus\{0\}$ 
we regard the exponential map 
\begin{equation}
\exp : \widetilde{\CC^*} := \CC  \longrightarrow \CC^*, 
\qquad (t \longmapsto \exp (t)) 
\end{equation}
as its universal covering. Then for 
the fiber product $\widetilde{X\setminus X_0}$ of 
$\exp : \widetilde{\CC^*} = \CC  \longrightarrow \CC^*$ and 
$f|_{X\setminus X_0}\colon X\setminus X_0 \longrightarrow \CC^*$
we obtain the following commutative diagram. 

\begin{equation}\begin{xy}\xymatrix@C=30pt{
& & \,\widetilde{X\setminus X_0}\, \ar[ld]_-{j\circ\pi} \ar[d]_-{\pi} \ar[r] 
\ar@{}[dr]|{\square} & \,\widetilde{\CC^*} \ar[d]^-{\mathrm{exp}}
 & \hspace{-48pt} {\,\simeq\,\CC} & \\
X_0=f^{-1}(0)\,\,\, \ar@{^{(}->}@<-0.3ex>[r]_-{i} & \,X\, & \,X\setminus X_0\, 
\ar@{_{(}->}@<+0.3ex>[l]^-{j} \ar[r]_-{f|_{X\setminus X_0}} & \,\CC^*. &
}\end{xy}\end{equation}
Here $\Box$ stands for the Cartesian diagram. 

\begin{definition}\label{def.7.1.001}
{\rm (Deligne~\cite{Deligne})} 
For $\F \in \Db (X)$ we define its nearby cycle 
complex $\psi_f( \F )\in \Db (X_0)$ along $f$ by 
\begin{equation}
\psi_f( \F ) = i^{-1} {\rm R} (j\circ\pi)_*(j\circ\pi)^{-1} \F
\ \in \Db (X_0).
\end{equation}
Moreover, we define the vanishing cycle 
complex $\phi_f( \F )\in \Db (X_0)$ of $\F \in \Db (X)$ 
along $f$ so that it fits into the distinguished triangle 
\begin{equation}
i^{-1} \F \longrightarrow \psi_f( \F ) \longrightarrow 
\phi_f( \F ) \overset{+1}{\longrightarrow}, 
\end{equation}
where the morphism $i^{-1} \F  \longrightarrow \psi_f( \F )$ 
is induced by the one 
${\rm id} \longrightarrow {\rm R} (j\circ\pi)_*(j\circ\pi)^{-1}$ 
of functors (for the more precise definition, see 
Kashiwara-Schapira \cite[page 351]{K-S}). 
\end{definition}
We thus obtain two functors 
\begin{equation}
\psi_f( \cdot ), \ \phi_f( \cdot ) \colon \Db (X) \longrightarrow \Db (X_0). 
\end{equation}
We call them the nearby and vanishing cycle functors 
respectively. Note that for the 
universal covering $ \widetilde{\CC^*} = \CC$ of 
$\CC^*$ we have a deck transformation 
\begin{equation}
\widetilde{\CC^*} = \CC \simto 
\widetilde{\CC^*} = \CC \qquad 
(z \longmapsto z-2\pi\sqrt{-1}). 
\end{equation}
Then by our construction of the covering space 
$\pi\colon\widetilde{X\setminus X_0} \longrightarrow  X\setminus X_0$ 
of $X\setminus X_0$, 
it induces an automorphism 
\begin{equation}
T\colon\widetilde{X\setminus X_0} \simto \widetilde{X\setminus X_0}
\end{equation}
of $\widetilde{X\setminus X_0}$. By $\pi \circ T= \pi$ 
there exists a morphism of functors 
\begin{equation}
{\rm R} \pi_* \pi^{-1}  \longrightarrow {\rm R} \pi_* ( {\rm R} T_* 
T^{-1} ) \pi^{-1}= {\rm R} \pi_* \pi^{-1}. 
\end{equation}
Then we obtain isomorphisms 
\begin{equation}
\Psi_f( \F ) \colon \psi_f( \F ) \simto \psi_f( \F ), \qquad 
\Phi_f( \F ) \colon \phi_f( \F ) \simto \phi_f( \F ). 
\end{equation} 
We call them the monodromy automorphisms. 
Now we recall the following basic result. 

\begin{proposition}\label{prp.7.1.002}
{\rm (see Kashiwara-Schapira \cite[Exercise VIII.15]{K-S} and 
Dimca \cite[Proposition 4.2.11]{Dimca})}  
Let $\rho \colon Y  \longrightarrow X$ be a 
proper morphism of complex analytic spaces 
and $f\colon X  \longrightarrow \CC$ a non-constant 
holomorphic function on $X$. We set 
$g:=f\circ\rho\colon Y  \longrightarrow \CC$ and 
\begin{equation}
X_0=f^{-1}(0)\subset X,\quad Y_0=g^{-1}(0)=\rho^{-1}(X_0)\subset Y. 
\end{equation}
Let $\rho|_{Y_0}\colon Y_0 \longrightarrow X_0$ be the 
restriction of $\rho$ to $Y_0\subset Y$, Then for 
$\G \in \Db (Y)$ there exist isomorphisms 
\begin{equation}
\psi_f( {\rm R} \rho_* \G ) \simeq {\rm R} (\rho|_{Y_0})_*\psi_g( \G ), \qquad 
\phi_f( {\rm R} \rho_* \G ) \simeq {\rm R} (\rho|_{Y_0})_*\phi_g( \G ). 
\end{equation} 
\end{proposition}

By the following basic results we can study 
Milnor fibers and their monodromies via nearby and 
vanishing cycle functors. 

\begin{theorem}\label{thm.7.2.004}
{\rm (see e.g. Dimca \cite[Proposition 4.2.2]{Dimca})} 
For any point 
$x \in X_0=f^{-1}(0)\subset X$ of 
$X_0$ there exist isomorphisms 
\begin{equation}
H^j(F_x; \CC ) \simeq H^j \psi_f( \CC_X )_x,
\quad \widetilde{H}^j(F_x; \CC ) \simeq 
H^j \phi_f(  \CC_X )_x \qquad (j\in\ZZ), 
\end{equation}
where $\widetilde{H}^j(F_x; \CC )$ stand for 
the reduced cohomology groups of $F_x$. 
Moreover these isomorphisms are compatible 
with the automorphisms of both sides 
induced by the monodromies. 
\end{theorem}

\begin{proof}
The problem being local, we may assume that 
$X \subset \CC^N$ and the point $x$ is 
the origin $0\in\CC^N$ of $\CC^N$. Set 
$Y:=X\times\CC_t$ and let 
\begin{equation}
g:X \hookrightarrow Y=X\times\CC_t \qquad 
(x \longmapsto (x,\,f(x))) 
\end{equation}
be the graph embedding of $X$ by $f$. By the 
projection $Y=X\times\CC_t \longrightarrow X$ 
we identify the graph $\Gamma_{f} :=g(X)\subset Y$ 
of $f$ with $X$. Then for 
$0<\eta\ll\varepsilon\ll 1$ 
\begin{equation}
\Gamma_{f} \cap (B(0;\varepsilon)\times\{ t_0\})\qquad (0<|t_0|<\eta)
\end{equation}
is identified with the Milnor fiber $F_0$ of $f$ at 
the origin $0\in X_0=f^{-1}(0)$. 
On the other hand, by applying Proposition \ref{prp.7.1.002} 
to the constructible sheaves 
$\CC_X \in \Dbc (X)$ and 
$g_*\CC_X =\CC_{\Gamma_{f}}\in \Dbc(Y) $ and 
the function $t\colon Y= X\times \CC \longrightarrow \CC$, 
we obtain an isomorphism 
\begin{equation}
\psi_t(\CC_{ \Gamma_{f}})_{(0,0)} \simeq \psi_{t\circ g}(\CC_X)_0 = 
\psi_f(\CC_X)_{0}.
\end{equation}
Hence it suffices to show that for 
$0<\eta\ll\varepsilon\ll 1$ we have isomorphisms 
\begin{equation}\label{eqn.7.2.005}
H^j\psi_t( \CC_{\Gamma_{f}} )_{(0,0)} \simeq 
H^j( B(0 ,\varepsilon)\times\{ t_0 \} ; \CC_{\Gamma_{f}} )
\qquad (j\in\ZZ, 0<|t_0|<\eta). 
\end{equation}
Note also that by Theorem \ref{thm.7.2.001} for 
$0<\eta\ll\varepsilon\ll 1$ there exist isomorphisms 
\begin{equation}
H^j ( \overline{B(0 ,\varepsilon)}
 \times\{ t_0 \} ; \CC_{\Gamma_{f}}) \simeq 
H^j( B(0 ,\varepsilon)\times\{ t_0 \} ; \CC_{\Gamma_{f}})
\qquad (j\in\ZZ, 0<|t_0|<\eta). 
\end{equation}
Restricting the exponential map 
$\exp \colon\widetilde{\CC^*} = \CC \longrightarrow \CC^*$ 
to the punctured disk 
$B(0;\eta)^*\subset\CC^*$ we obtain a covering maps  
\begin{equation}
\pi_0\colon\widetilde{B(0;\eta)^*}:=\{ z\in\CC \mid {\rm Re}\,z<\log \eta \} 
\longrightarrow B(0;\eta)^* 
\end{equation}
and 
\begin{equation}
\pi={\rm id_{\overline{B(0;\varepsilon)}}} \times\pi_0\colon 
\overline{B(0;\varepsilon)} \times\widetilde{B(0;\eta)^*} 
\longrightarrow \overline{B(0;\varepsilon)} \times B(0;\eta)^*.
\end{equation}
Then by the definition of the 
nearby cycle sheaf $\psi_t( \CC_{\Gamma_{f}})$ 
 for $0<\eta\ll\varepsilon\ll 1$ we obtain isomorphisms 
\begin{equation}
H^j\psi_t( \CC_{ \Gamma_{f}})_{(0,0)} \simeq H^j(
\overline{B(0;\varepsilon)} 
\times\widetilde{B(0;\eta)^*}; \pi^{-1} \CC_{ \Gamma_{f}}) \qquad (j\in\ZZ). 
\end{equation}
Let $p \colon \overline{B(0;\varepsilon)} \times B(0;\eta)^* \longrightarrow 
B(0;\eta)^*$ and 
$q\colon \overline{B(0;\varepsilon)} \times\widetilde{B(0;\eta)^*} 
\longrightarrow \widetilde{B(0;\eta)^*}$ be the projections. 
As they are proper, we obtain an isomorphism 
\begin{equation}
{\rm R} q_* (\pi^{-1} \CC_{ \Gamma_{f}}) \simeq 
\pi_0^{-1} {\rm R} p_* (\CC_{ \Gamma_{f}}). 
\end{equation}
Moreover by Theorem \ref{thm.7.2.001} the cohomology sheaves 
\begin{equation}
H^j {\rm R} q_* (\pi^{-1} \CC_{ \Gamma_{f}}) \simeq 
H^j \pi_0^{-1} {\rm R} p_* (\CC_{ \Gamma_{f}})
\qquad (j\in\ZZ)
\end{equation}
on $\widetilde{B(0;\eta)^*}$ are locally constant. Since 
$\widetilde{B(0;\eta)^*}$ is homeomorphic to 
the complex plane $\CC$, they are globally constant. 
Take a point $z_0\in\widetilde{B(0;\eta)^*}$ and set 
$t_0:= \pi_0(z_0)= \exp (z_0) \in B(0;\eta)^*$. Then we obtain 
isomorphisms 
\begin{align}
& {\rm R} \Gamma ( \widetilde{B(0;\eta)^*}; {\rm R} q_*
(\pi^{-1} \CC_{ \Gamma_{f}})) \simeq 
 {\rm R} q_*(\pi^{-1} \CC_{ \Gamma_{f}})_{z_0} \\
& \simeq {\rm R} p_*( \CC_{ \Gamma_{f}})_{t_0}
\simeq  {\rm R} \Gamma (  \overline{B(0;\varepsilon)} \times\{ t_0\} ; 
\CC_{ \Gamma_{f}} ). 
\end{align}
We thus obtain the isomorphisms in \eqref{eqn.7.2.005}. 
Following the above argument more carefully, 
we can also show the assertion on the monodromy of 
$\psi_f( \CC_X)$. The remaining assertion 
for $\phi_f( \CC_X)$ follows from the 
cone theorem proved in Milnor \cite[Theorem 2.10]{Milnor} 
and Burghelea-Verona \cite[Lemma 3.2]{B-V}. 
This completes the proof. 
\qed
\end{proof}

Similarly, we can prove the following 
more general result. 

\begin{theorem}\label{thm.7.2.006}
In the situation of Theorem \ref{thm.7.2.004}, 
for any constructible sheaf $\F \in \Dbc (X)$ on $X$ and 
any point 
$x \in X_0=f^{-1}(0)\subset X$ of $X_0$ there exist isomorphisms 
\begin{equation}
H^j(F_x; \F ) \simeq H^j \psi_f( \F )_x 
 \qquad (j\in\ZZ).
\end{equation}
\end{theorem}

Since for a point $x \in X_0=f^{-1}(0)$ the Milnor fiber $F_x$ of $f$ 
at $x$ is a subset of $X \setminus X_0$, we obtain the 
following very simple consequence of Theorem \ref{thm.7.2.006}. 

\begin{corollary}\label{new-corol}
In the situation of Theorem \ref{thm.7.2.006}, 
let $j: X \setminus X_0 \hookrightarrow X$ 
be the inclusion map. Then for any 
constructible sheaf $\F \in \Dbc (X)$ on $X$ and 
any point 
$x \in X_0=f^{-1}(0)\subset X$ of $X_0$ there exist isomorphisms 
\begin{align}
\psi_f( j_!j^{-1} \F )_x & = \psi_f( \F_{X \setminus X_0} )_x 
\simeq 
\psi_f( \F )_x,  
\\ 
\psi_f( {\rm R}j_*j^{-1} \F )_x & = \psi_f( 
{\rm R} \Gamma_{X \setminus X_0} \F )_x 
\simeq 
\psi_f( \F )_x,  
\end{align}
\end{corollary}

\begin{lemma}\label{lem:2-ac-1}
For $1 \leq k<l \leq n$ 
set $\Omega :=( \CC^*)^l \times \CC^{n-l} \subset \CC^n$ and 
let $j: \Omega \hookrightarrow \CC^n$
be the inclusion map and $h: \CC^n \longrightarrow \CC$ 
the function on $\CC^n$ defined by $h(z)=z_1^{m_1}z_2^{m_2} 
\cdots z_k^{m_k}$ 
($m_i \in \ZZ_{>0}$) for $z=(z_1,z_2, \ldots, z_n) \in \CC^n$. 
Then for the constructible sheaf $j_!j^{-1} \CC_{\CC^n} 
\simeq j_! \CC_{\Omega} \in 
\Dbc ( \CC^n)$ we have $\psi_h( j_! \CC_{\Omega} )_0 \simeq 0$. 
\end{lemma}

\begin{proof}
Set $D:= \CC^n \setminus \Omega 
= \{ z \in \CC^n \ | \ z_1z_2 \cdots z_l=0 \}$ and 
let $i:D \hookrightarrow \CC^n$ be the inclusion map. 
Then there exists a distinguished triangle 
\begin{equation}
\psi_h(  j_! \CC_{\Omega} )_0 \longrightarrow 
\psi_h( \CC_{\CC^n} )_0 \longrightarrow \psi_h( i_! \CC_{D} )_0
 \overset{+1}{\longrightarrow}
\end{equation}
in $\Dbc ( \{ 0 \} )$. 
Note that by the condition $l>k$ the Milnor fiber $F_0$ of 
$h$ at the origin $0 \in h^{-1}(0) \subset \CC^n$ is 
homotopic to $F_0 \cap D$. Then by Theorem \ref{thm.7.2.004} 
we obtain an isomorphism 
\begin{equation}
\psi_h( \CC_{\CC^n} )_0 \simto \psi_h( i_! \CC_{D} )_0. 
\end{equation}
From this the assertion immediately follows. 
\qed
\end{proof}

Now we shall introduce some basic properties of the 
two functors 
\begin{equation}
\psi_f(  \cdot  ), \ \phi_f(  \cdot  ) \colon \Db (X) \longrightarrow \Db (X_0).
\end{equation}

\begin{theorem}\label{thm.7.1.004}
{\rm (see Kashiwara-Schapira \cite[Proposition 8.6.3]{K-S}, 
Dimca \cite[page 103]{Dimca} and 
Sch\"urmann \cite[Theorem 4.0.2 and Lemma 4.2.1]{Sch})} 
Assume that $\F \in \Db(X)$ is constructible. Then 
$\psi_f( \F ), \phi_f( \F ) \in  \Db(X_0)$ 
are also constructible. Namely there exist two  
functors 
\begin{equation}
\psi_f(  \cdot  ),\ 
\phi_f(  \cdot  ) \colon \Dbc (X) \longrightarrow \Dbc (X_0).
\end{equation}
\end{theorem}
Shifting the functors $\psi_f( \cdot ), \phi_f( \cdot )$ by 
$-1 \in \ZZ$ we set 
\begin{equation}
^p\!\psi_f( \cdot ) := \psi_f( \cdot )[-1], \quad 
^p\!\phi_f( \cdot ) := \phi_f( \cdot )[-1].
\end{equation}
Then we have the following deep result 
proved first by Kashiwara in \cite{K-2} 
using D-modules. See also 
Brylinski \cite[Theorem 1.2 and Corollary 1.7]{Bry} and 
Mebkhout-Sabbah 
\cite[Chapter III, Section 4]{Meb} for 
the details. Later, purely geometric  
proofs to it were given by Goresky-MacPherson 
\cite{GM}, Kashiwara-Schapira \cite[Corollary 10.3.11]{K-S} and 
Sch\"urmann \cite[Chapter 6]{Sch}. 

\begin{theorem}\label{thm.7.2.007}
In the situation of Theorem \ref{thm.7.2.004}, 
assume that $\F \in \Dbc (X)$ is a perverse 
sheaf on $X$. Then 
$^p\!\psi_f( \F ), \ ^p\!\phi_f( \F )\in \Dbc (X_0)$ are 
also perverse sheaves on $X_0=f^{-1}(0)$. 
\end{theorem}

By Theorem \ref{thm.7.2.007}, even if the complex hypersurface 
$X_0=f^{-1}(0) \subset X= \CC^n$ for $f: X= \CC^n \longrightarrow \CC$ 
does not have an isolated singular point at the 
origin $0\in X_0=f^{-1}(0)$, we can prove various 
vanishing theorems for the cohomology groups $H^j(F_0; \CC)$ 
of Milnor fiber $F_0$ (see e.g. Dimca \cite[Section 6.1]{Dimca} 
for the details). Similarly, by using D-modules we obtain 
also the following refinement of Theorem \ref{thm.7.2.007}. 

\begin{theorem}\label{thm.7.2.0078}
In the situation of \ref{thm.7.2.004}, 
assume that $\F \in \Dbc (X)$ is a perverse 
sheaf on $X$. Then in the abelian category ${\rm Perv}( \CC_{X_0})$ of 
perverse sheaves on $X_0$ we have the direct sum 
decompositions: 
\begin{equation}
^p\!\psi_f( \F ) \simeq \bigoplus_{\lambda \in \CC} \ ^p\!\psi_{f, \lambda} ( \F ), 
\qquad 
^p\!\phi_f( \F ) \simeq \bigoplus_{\lambda \in \CC} \ ^p\!\phi_{f, \lambda}( \F ), 
\end{equation} 
where the right hand sides are finite sums and for $\lambda \in \CC$ 
taking a large enough integer $m \gg 0$ we set 
\begin{align}
 ^p\!\psi_{f, \lambda}( \F ):= & 
{\rm Ker} \Bigl[ \bigl( \Psi_f( \F )- \lambda \cdot {\rm id} \bigr)^m :  
\  ^p\!\psi_f( \F ) \longrightarrow \ ^p\!\psi_f( \F ) \Bigr] 
\subset \  ^p\!\psi_f( \F ), 
\\ 
 ^p\!\phi_{f, \lambda}( \F ):= & 
{\rm Ker} \Bigl[ \bigl( \Phi_f( \F )- \lambda \cdot {\rm id} \bigr)^m : 
 \ ^p\!\phi_f( \F ) \longrightarrow \ ^p\!\phi_f( \F ) \Bigr] 
\subset \  ^p\!\phi_f( \F ). 
\end{align} 
\end{theorem}

\begin{corollary}\label{cor.7.2.0078}
In the situation of Theorem \ref{thm.7.2.004}, assume 
moreover that $X= \CC^n$. Then for any point $x \in 
X_0=f^{-1}(0)$ of $X_0$, $\lambda \in \CC$ and 
$j \in \ZZ$ 
there exist isomorphisms 
\begin{align}
H^j(F_x; \CC )_{\lambda} \simeq & H^{j-n+1}
\ ^p\!\psi_{f, \lambda}( \CC_X [n] )_x, 
\\ 
\widetilde{H}^j(F_x; \CC )_{\lambda} \simeq & H^{j-n+1}
\ ^p\!\phi_{f, \lambda}( \CC_X [n] )_x, 
\end{align} 
where $H^j(F_x; \CC )_{\lambda} \subset H^j(F_x; \CC )$ and 
$\widetilde{H}^j(F_x; \CC )_{\lambda} \subset \widetilde{H}^j(F_x; \CC )$ 
are the generalized eigenspaces of the Milnor 
monodromies for the eigenvalue $\lambda$. 
\end{corollary} 

As was first observed by A'Campo \cite{A'Campo}, 
Milnor fibers and their monodromies can be precisely studied 
by resolutions of singularities of 
$X_0=f^{-1}(0) \subset X$. Let us explain his brilliant idea 
in the more general framework 
of constructible sheaves. For this purpose, 
we define monodromy zeta functions of $f$ as follows. 
Let $\CC(t)^*=\CC(t) \setminus \{0\}$ be the 
multiplicative group of the function field $\CC(t)$ of 
one variable $t$. 

\begin{definition}\label{dfn:2-79}
For a point $x \in X_0=f^{-1}(0)$ of $X_0$ we define 
the monodromy zeta function $\zeta_{f,x}(t) \in \CC(t)^*$ 
of $f$ at $x$ by 
\begin{equation}
\zeta_{f,x}(t):=\prod_{j=0}^{\infty} \ 
 \Bigl\{ \det(\id -t \cdot \Phi_{j,x}) \Bigr\}^{(-1)^j} 
\quad \in \CC(t)^*, 
\end{equation}
where the $\CC$-linear maps 
$\Phi_{j,x} \colon H^j(F_x; \CC ) \simto H^j(F_x; \CC )$ are 
the Milnor monodromies at $x$. 
\end{definition}

The following lemma was obtained by A'Campo \cite{A'Campo} 
(see also Oka \cite[Chapter I, Example (3.7)]{Oka} for a 
precise explanation). 

\begin{lemma}\label{lem:2-ac-21}
For $1 \leq k \leq n$ let $h: \CC^n \longrightarrow \CC$ be  
the function on $\CC^n$ defined by $h(z)=z_1^{m_1}z_2^{m_2} 
\cdots z_k^{m_k}$ 
($m_i \in \ZZ_{>0}$) for $z=(z_1,z_2, \ldots, z_n) \in \CC^n$. 
Then we have 
\begin{equation}
\zeta_{h,0}(t)=
\begin{cases}
\ 1-t^{m_1} & (k=1) \\
& \\
\ 1 & (k>1) \\
\end{cases}
\end{equation} 
\end{lemma}

Note that if the complex hypersurface 
$X_0=f^{-1}(0) \subset X= \CC^n$ for $f: X= \CC^n \longrightarrow \CC$ 
has an isolated singular point at the 
origin $0\in X_0=f^{-1}(0)$ by Theorem \ref{thm.7.2.003}  
we have 
\begin{equation}
\zeta_{f,x}(t)=(1-t) \cdot 
\Bigl\{ \det(\id -t \cdot \Phi_{n-1,x}) \Bigr\}^{(-1)^{n-1}}. 
\end{equation}
Hence in this particular case, we can know the characteristic polynomial 
$\det(t \cdot \id - \Phi_{n-1,x}) \in \CC [t]$ of 
the ``principal" Milnor monodromy 
\begin{equation}
\Phi_{n-1,x} \colon H^{n-1}(F_x; \CC ) \simto H^{n-1}(F_x; \CC )
\end{equation}
from the monodromy zeta function $\zeta_{f,x}(t) \in \CC(t)^*$. 
This classical notion of monodromy zeta functions can be generalized as follows.

\begin{definition}\label{dfn:2-8}
For a constructible sheaf $\F \in \Dbc(X)$ on $X$ and 
a point $x \in X_0=f^{-1}(0)$ of $X_0$ we define 
the monodromy zeta function $\zeta_{f,x}(\F) \in \CC(t)^*$ of 
$\F$ along $f$ at $x$ by 
\begin{equation}
\zeta_{f,x}(\F)(t):=\prod_{j \in \ZZ} \ 
\Bigl\{ \det\left(\id -t\Phi(\F)_{j,x}\right) \Bigr\}^{(-1)^j} 
\quad \in \CC(t)^*, 
\end{equation}
where the $\CC$-linear maps 
$\Phi(\F)_{j,x} \colon H^j(\psi_f(\F))_x \simto H^j(\psi_f(\F))_x$ are induced by 
$\Phi(\F)$.
\end{definition}
As in A'Campo \cite{A'Campo} (see also Oka \cite[Chapter I]{Oka}) 
we can easily prove the following very useful result. 

\begin{lemma}\label{lemma:2-8}
Let $\F^{\prime} \longrightarrow \F \longrightarrow \F^{\prime \prime}
 \overset{+1}{\longrightarrow}$ be a distinguished triangle 
in $\Dbc (X)$. Then for any point $x \in X_0=f^{-1}(0)$ of $X_0$ we have 
\begin{equation}
\zeta_{f,x}( \F )(t)= \zeta_{f,x}( \F^{\prime} )(t) 
\cdot \zeta_{f,x}( \F^{\prime \prime} )(t). 
\end{equation}
\end{lemma} 

For an abelian group $G$ and a complex analytic space $Z$,  
we shall say that a $G$-valued function $\varphi :Z \longrightarrow G$ on $Z$ is constructible if 
there exists a stratification $Z=\bigsqcup_{\alpha}Z_{\alpha}$ of $Z$ 
 such that $\varphi |_{Z_{\alpha}}$ is 
constant for any $\alpha$. We denote by $\CF_G(Z)$ 
the abelian group of $G$-valued constructible functions on $Z$. 
In this note, we consider $\CF_G(Z)$ only for $G= \CC(t)^*$. 
Then for the complex analytic space $X$ and the non-constant 
holomorphic function $f :X \longrightarrow \CC$ on it, 
by Theorem \ref{thm.7.2.0078}   
we can easily see that the $\CC(t)^*$-valued function 
$\zeta_{f}(\F) :X_0 \longrightarrow \CC(t)^*$ on $X_0=f^{-1}(0)$ 
defined by 
\begin{equation}
\zeta_{f}(\F)(x):=\zeta_{f,x}(\F)(t) \quad \in \CC(t)^* \qquad 
(x \in X_0=f^{-1}(0))
\end{equation}
is constructible. 
For a $G$-valued constructible function $\varphi \colon Z \longrightarrow G$ on 
a complex analytic space $Z$, by taking 
a stratification $Z=\bigsqcup_{\alpha}Z_{\alpha}$ of $Z$ such that $\varphi |_{Z_{\alpha}}$ is 
constant for any $\alpha$, we set 
\begin{equation}
\int_Z \varphi :=\sum_{\alpha}\chi(Z_{\alpha}) \cdot 
\varphi (z_{\alpha}) \quad \in G, 
\end{equation}
where $\chi ( \cdot )$ stands for the 
topological Euler characteristic and $z_{\alpha}$ is a reference point in 
the stratum $Z_{\alpha}$. 
By the following lemma $\int_Z \varphi \in G$ does not depend on the choice of 
the stratification $Z=\bigsqcup_{\alpha} Z_{\alpha}$ of $Z$. 
We call it the topological (or Euler) integral of 
$\varphi$ over $Z$. 

\begin{lemma}\label{lem:2-str-1}
Let $Y$ be a complex analytic space and 
$Y=\bigsqcup_{\alpha}Y_{\alpha}$ a stratification of $Y$. 
Then we have 
\begin{equation}
\chi_{{\rm c}}(Y) = \dsum_{\alpha} \chi_{{\rm c}}(Y_{\alpha}), 
\end{equation}
where $\chi_{{\rm c}}( \cdot )$ stands for the Euler 
characteristic with compact supports. Moreover, 
for any $\alpha$ we have 
$\chi_{{\rm c}}(Y_{\alpha})=\chi (Y_{\alpha})$. 
\end{lemma}

\begin{proof}
Let $n:= {\rm dim }Y$ and for $0 \leq i \leq n$ set 
\begin{equation}
Y_i:= \bigsqcup_{\alpha : \ {\rm dim }Y_{\alpha} \leq i} Y_{\alpha} 
\quad \subset Y. 
\end{equation}
Then $Y_i$ are closed subsets of $Y$ and $Y_n=Y$. 
Applying the functor ${\rm R} \Gamma_{{\rm c}} (Y ;  \ \cdot \ )$ 
to the exact sequence 
\begin{equation}
0 \longrightarrow  \CC_{Y \setminus Y_{n-1}}  \longrightarrow  \CC_{Y} 
 \longrightarrow   \CC_{Y_{n-1}} \longrightarrow  0, 
\end{equation}
we obtain an equality 
\begin{equation}
\chi_{{\rm c}}(Y) = \chi_{{\rm c}}(Y_{n-1}) + 
\chi_{{\rm c}}(Y \setminus Y_{n-1}). 
\end{equation}
Note also that we have 
\begin{equation}
Y \setminus Y_{n-1}= 
\bigsqcup_{\alpha : \ {\rm dim }Y_{\alpha}=n} Y_{\alpha}  
\end{equation}
and hence 
\begin{equation}
\chi_{{\rm c}}(Y \setminus Y_{n-1}) = 
\dsum_{\alpha : \ {\rm dim }Y_{\alpha}=n} \chi_{{\rm c}}(Y_{\alpha}).  
\end{equation}
Repeating this argument, we finally obtain the first assertion. 
Since for any $\alpha$ the stratum $Y_{\alpha}$ is a 
complex manifold, by the (generalized) 
Poincar\'e duality isomorphisms 
\begin{equation}
\bigl[ H^j_{{\rm c}} (Y_{\alpha}; \CC ) 
\bigr]^* \simeq 
H^{2 {\rm dim }Y_{\alpha}-j}  (Y_{\alpha}; \CC ) 
\qquad (0 \leq j \leq 2 {\rm dim }Y_{\alpha}) 
\end{equation}
we obtain also the second assertion. 
\qed 
\end{proof} 

More generally, 
for any morphism $\rho \colon Z \longrightarrow W$ of complex analytic spaces 
and any $G$-valued constructible function 
$\varphi \in \CF_G(Z)$ on $Z$, we define the push-forward $\int_{\rho} \varphi \in \CF_G(W)$ of 
$\varphi$ by 
\begin{equation}
\Bigl( \int_{\rho} \varphi \Bigr) (w) :=\int_{\rho^{-1}(w)} \varphi \qquad 
(w \in W). 
\end{equation}
We thus obtain a homomorphism 
\begin{equation}
\int_{\rho} : \CF_G(Z) \longrightarrow \CF_G(W)
\end{equation}
of abelian groups. By Theorem \ref{thm.7.2.0078} and 
Proposition \ref{prp.7.1.002}, we can easily prove the 
following very useful result.   

\begin{theorem}\label{prp:2-99}
{\rm (Dimca \cite[Proposition 4.2.11]{Dimca} 
and Sch\"urmann \cite[Chapter 2]{Sch})}  
In the situation of Proposition \ref{prp.7.1.002}, 
assume that $\G\in \Db (Y)$ is constructible. 
Then we have 
\begin{equation}
\int_{ \rho|_{Y_0}} \zeta_g(\G) =\zeta_f({\rm R} \rho_*\G)
\end{equation}
in $\CF_{\CC(t)^*}(X_0)$, where $\int_{\rho|_{Y_0}}\colon \CF_{\CC(t)^*}(Y_0) 
\longrightarrow \CF_{\CC(t)^*}(X_0)$ is the push-forward of 
$\CC(t)^*$-valued constructible functions by $\rho|_{Y_0} \colon Y_0 \longrightarrow X_0$.
\end{theorem}
For the precise proof of this theorem, 
see for example, \cite[p.170-173]{Dimca} and 
\cite[Chapter 2]{Sch}. 

\begin{corollary}\label{cpr:2-99}
In the situation of Theorem \ref{prp:2-99}, 
for a point $x \in X_0=f^{-1}(0)$ of $X_0$ let 
$\rho^{-1}(x)= \sqcup_{\alpha} Z_{\alpha}$ be 
a stratification of $Z:= \rho^{-1}(x) \subset 
Y_0=g^{-1}(0)$ such that for any $\alpha$ 
the $\CC(t)^*$-valued constructible function 
$\zeta_{g}( \G ) \in \CF_{\CC(t)^*}(Y_0)$ is 
constant on $Z_{\alpha}$. Then we have 
\begin{equation}
\zeta_{f,x}( {\rm R} \rho_* \G )(t)= 
\dprod_{\alpha} \Bigl\{
\zeta_{g, y_{\alpha}}( \G )(t) 
\Bigr\}^{\chi ( Z_{\alpha} )}, 
\end{equation}
where $y_{\alpha}$ is a reference point of $Z_{\alpha}$. 
\end{corollary}

In the special but important case where $\G = \CC_Y$ 
and the restriction $Y \setminus Y_0 \longrightarrow 
X \setminus X_0$ of $\rho$ is an isomorphism, 
this result was first obtained by 
Gusein-Zade-Luengo-Melle-Hern\'andez 
\cite[Corollary 1]{GZ-L-MH-1}. 
The following two lemmas are useful to calculate monodromy 
zeta functions of constructible sheaves by 
Corollary \ref{cpr:2-99}. Indeed, 
by resolutions of singularities, we can reduce the 
problem to the situations treated in them. 
For their applications to 
the monodromies of $A$-hypergeometric functions, 
see \cite{Takeuchi-2} and \cite{A-E-T}.

\begin{lemma}\label{lem:2-99}
For $1 \leq k \leq n$ let $\L$ be a $\CC$-local system 
of rank $r>0$ on $( \CC^*)^k \times \CC^{n-k} \subset \CC^n$. 
Let $j: ( \CC^*)^k \times \CC^{n-k} \hookrightarrow \CC^n$
be the inclusion map and $h: \CC^n \longrightarrow \CC$ 
the function on $\CC^n$ defined by $h(z)=z_1^m$ 
($m \in \ZZ_{>0}$) for $z=(z_1,z_2, \ldots, z_n) \in \CC^n$. 
\begin{enumerate}
\item[\rm{(i)}] Assume moreover that $k>1$. Then we have 
$\psi_h(j_! \L )_0 \simeq 0$ and hence $\zeta_{h,0}(j_! \L )(t)=1$.  
\item[\rm{(ii)}]  Assume moreover that $k=1$ and 
let $A \in {\rm GL}_r( \CC )$ be the monodromy matrix of 
$\L$ along the loop $C:= \{ (e^{\sqrt{-1} \theta}, 0, \ldots, 0) 
\ | \ 0 \leq \theta \leq 2 \pi \}$ in 
$( \CC^*)^k \times \CC^{n-k}= 
\CC^* \times \CC^{n-1}$ (defined up to conjugacy). 
Then we have 
\begin{equation}
\zeta_{h,0}(j_! \L )(t)= 
{\rm det} \bigl( {\rm id}- t^m A \bigr) \quad \in \CC (t)^*. 
\end{equation}
\end{enumerate}
\end{lemma}

\begin{proof}
By Theorem \ref{thm.7.2.004} the assertion of (i) 
immediately follows from \cite[Proposition 5.2 (ii)]{M-T-new2} 
and the K\"unneth formula.  
We also obtain the one in (ii) by 
\cite[Proposition 5.2 (iii)]{M-T-new2}. 
\qed 
\end{proof}

The following result, which is a 
generalization of A'Campo's lemma i.e. 
Lemma \ref{lem:2-ac-21} 
(see also Oka \cite[Example (3.7)]{Oka}) 
to constructible sheaves, was proved in 
\cite[Proposition 5.3]{M-T-new2}.  

\begin{lemma}\label{lem:2-999}
For $0 \leq k \leq n$ let $\L$ be a $\CC$-local system 
of rank $r>0$ on $( \CC^*)^k \times \CC^{n-k} \subset \CC^n$. 
Let $j: ( \CC^*)^k \times \CC^{n-k} \hookrightarrow \CC^n$
be the inclusion map and $h: \CC^n \longrightarrow \CC$ 
the function on $\CC^n$ defined by $h(z)=z_1^{m_1}z_2^{m_2} 
\cdot z_n^{m_n}$ 
($m_i \in \ZZ_{\geq 0}$) for $z=(z_1,z_2, \ldots, z_n) \in \CC^n$. 
Assume that $\sharp \{ 1 \leq i \leq n \ | \ m_i>0 \} 
\geq 2$. Then we have $\zeta_{h,0}(j_! \L )(t)=1$.  
\end{lemma}

\section{Theorems of Kouchnirenko, Varchenko and Oka}\label{section 3}

In this section, we introduce the beautiful theorems of 
Kouchnirenko, Varchenko and Oka which express 
the Milnor numbers and the monodromy 
zeta functions of holomorphic functions in terms of their 
Newton polyhedra. Here, for simplicity, we first consider 
non-constant polynomials 
$f(x) \in \CC [x]= \CC [x_1, \ldots, x_n]$ on $\CC^n$ 
such that $f(0)=0$. We set $\ZZ_+^n:=( \ZZ_{\geq 0})^n \subset \ZZ^n$ and 
$\RR_+^n:=( \RR_{\geq 0})^n \subset \RR^n$. 
We shall freely use standard notions in toric geomety, 
for which we refer to Fulton \cite{Fulton} and 
Oda \cite{Oda}. First of all, we recall 
the simplest case of the 
Bernstein-Khovanskii-Kouchnirenko  
theorem \cite{Khovanskii}. 

\begin{definition}
Let $g(x)=\sum_{v \in \ZZ^n} c_vx^v \in 
\CC [x_1^{\pm}, \ldots, x_n^{\pm}]$ ($c_v\in \CC$) be a 
Laurent polynomial on the algebraic torus 
$T=(\CC^*)^n$. 
\begin{enumerate}
\item[\rm{(i)}] We call the convex hull of 
$\supp(g):=\{v\in \ZZ^n \ | \ c_v\neq 0\} 
\subset \ZZ^n \subset \RR^n$ in $\RR^n$ the 
Newton polytope of $g$ and denote it by $NP(g)$.
\item[\rm{(ii)}] For a face $\gamma \prec NP(g)$ of $NP(g)$, 
we define the $\gamma$-part 
$g^{\gamma}$ of $g$ by 
$g^{\gamma}(x):=\sum_{v \in \gamma \cap \ZZ^n} c_vx^v$. 
\end{enumerate}
\end{definition}

\begin{definition}\label{non-deg-simple} 
(see \cite{Oka}) 
Let $g(x) \in \CC [x_1^{\pm}, \ldots, x_n^{\pm}]$ be a 
Laurent polynomial on  
$T=(\CC^*)^n$ and set $\Delta :=NP(g) \subset \RR^n$. 
Then we say that $g$ is 
non-degenerate 
if for any face $\gamma \prec \Delta$ of 
$\Delta$ the $1$-form $dg^{\gamma}$ does not vanish 
on the complex hypersurface 
$\{ x\in T=(\CC^*)^n  \ | \ g^{\gamma}(x)=0 \}$.
\end{definition}

\begin{theorem}\label{BKK-simple} 
{\rm (Bernstein-Khovanskii-Kouchnirenko 
theorem \cite{Khovanskii}, see also 
Oka \cite[Chapter IV, Theorem (3.1)]{Oka})} 
Let $g(x) \in \CC [x_1^{\pm}, \ldots, x_n^{\pm}]$ be a 
Laurent polynomial on  
$T=(\CC^*)^n$ and set $\Delta :=NP(g) \subset \RR^n$. 
Assume that $g$ is 
non-degenerate. Then for the complex hypersurface 
$Z^* =\{ x\in T=(\CC^*)^n \ | \ g(x)
=0 \}$ of $T=(\CC^*)^n$ we have
\begin{equation}
\chi(Z^*)=(-1)^{n-1} \Vol_{\ZZ}( \Delta ),  
\end{equation}
where $\Vol_{\ZZ}( \Delta ) \in \ZZ_+$ 
is the normalized $n$-dimensional volume 
(i.e. the $n!$ times the usual volume) of 
$\Delta$ with respect to the lattice $\ZZ^n 
\subset \RR^n$.
\end{theorem}

\begin{definition}
Let $f(x)=\sum_{v \in \ZZ_+^n} a_v x^v \in \CC [x]$ ($a_v\in \CC$) be a 
non-constant polynomial on $\CC^n$ such that $f(0)=0$. 
\begin{enumerate}
\item[\rm{(i)}] We define the support $\supp(f) \subset \ZZ_+^n$ 
of $f$ by 
\begin{equation}
\supp(f):=\{v\in \ZZ^n_+ \ | \ a_v\neq 0\} 
\ \subset \ZZ_+^n. 
\end{equation}
\item[\rm{(ii)}] We define the Newton polyhedron 
$\Gamma_+(f) \subset \RR_+^n$ of $f$ at 
the origin $0 \in \CC^n$ of $\CC^n$ to be 
the convex hull of $\cup_{v \in \supp(f)} 
(v + \RR_+^n)$ in $\RR_+^n$.  
\item[\rm{(iii)}] For a compact 
face $\gamma \prec \Gamma_+(f)$ of $\Gamma_+(f)$, 
we define the $\gamma$-part 
$f^{\gamma} \in \CC [x]$ of $f$ by 
$f^{\gamma}(x):=\sum_{v \in \gamma \cap \ZZ_+^n} a_vx^v$. 
\item[\rm{(iv)}] We say that $f$ is convenient if 
$\Gamma_+(f)$ intersects the positive part of each 
coordinate axis of $\RR^n$. 
\end{enumerate}
\end{definition}

\begin{definition} {\rm (Kouchnirenko \cite{Kushnirenko})} 
Let $f(x)=\sum_{v \in \ZZ_+^n} a_v x^v \in \CC [x]$ ($a_v\in \CC$) be a 
non-constant polynomial on $\CC^n$ such that $f(0)=0$. 
Then we say that $f$ is non-degenate (at the origin $0 \in \CC$) if 
for any compact face $\gamma \prec \Gamma_+(f)$ of $\Gamma_+(f)$ 
the complex hypersurface 
\begin{equation}
\{ x=(x_1, \ldots, x_n)  \in (\CC^*)^n \ |\ f^{\gamma}(x)=0 \} 
\ \subset (\CC^*)^n
\end{equation} 
of $(\CC^*)^n$ is smooth.
\end{definition}

Note that among the polynomials having the same 
Newton polyhedron this non-degeneracy condition 
is satisfied by generic ones. 
For a subset $\J\subset \{1,2,\ldots,n\}$, let us set
\begin{equation}
\RR^{\J}:=\{ v=(v_1,v_2,\ldots,v_n)\in \RR^n\ |\ 
\text{$v_i=0$ for $i \notin \J$}\} \simeq \RR^{\sharp S}.
\end{equation}
Denote by $( \RR^{\J} )^* \simeq \RR^{\sharp S}$ 
the dual vector space of $\RR^{\J} \simeq \RR^{\sharp S}$. 
We set also $\Gamma_{+}^{\J}(f)=\Gamma_{+}(f) \cap \RR^{\J}$. 
Note that if $f$ is convenient 
we have $\dim \Gamma_{+}^{\J}(f)=\sharp {\J}$ 
for any non-empty subset ${\J} \subset \{1,2,\ldots,n\}$. 
Now let $f(x)=\sum_{v \in \ZZ_+^n} a_v x^v \in \CC [x]$ ($a_v\in \CC$) be a 
non-constant polynomial on $\CC^n$ such that $f(0)=0$. 
Denote by $( \RR^n)^*_+ 
\simeq \RR^n_+$ the dual cone of $\RR_+^n \subset \RR^n$ 
in $( \RR^n)^*$. For $u \in ( \RR^n)^*_+$ we define 
the supporting face $\gamma^u \prec \Gamma_+(f)$ of 
$u$ in $\Gamma_+(f)$ by 
\begin{equation}
\gamma^u := \left\{ v \in \Gamma_+(f) \ | \ 
\langle u, v \rangle = 
\min_{w \in \Gamma_{+}(f)} \langle u , w \rangle \right\}. 
\end{equation}
Then for a face $\gamma \prec \Gamma_+(f)$ of 
$\Gamma_+(f)$ the (closed) polyhedral cone 
\begin{equation}
\sigma ( \gamma ):= 
\overline{ \{ u \in ( \RR^n)^*_+ \ | \ \gamma^u = \gamma \} } 
\quad \subset ( \RR^n)^*_+
\end{equation}
in $( \RR^n)^*_+$ is rational and strongly convex. 
Moreover, we can easily see that the family 
$\Sigma_1 := \{ \sigma ( \gamma ) \}_{\gamma \prec \Gamma_+(f)}$ 
of cones in $( \RR^n)^*_+$ thus obtained 
satisfies the axioms of fans. We call it the dual fan of 
$\Gamma_+(f)$ in $( \RR^n)^*_+$. 
Conversely, for a rational and strongly convex 
polyhedral cone $\sigma \subset ( \RR^n)^*_+$ we set 
\begin{equation}
\gamma ( \sigma ):= 
\bigcap_{u \in \sigma} \gamma^u \quad \prec \Gamma_+(f). 
\end{equation}
We call it the supporting face of $\sigma$ in 
$\Gamma_+(f)$. Then it is easy to see that 
for any $u \in ( \RR^n)^*_+$ in the relative 
interior ${\rm rel.int} ( \sigma )$ of $\sigma$ 
we have $\gamma ( \sigma )= \gamma^u$. 
For each non-empty subset ${\J} \subset \{1,2,\ldots, n\}$, let 
$n({\J}) \in \ZZ_+$ be the number of the 
$(\sharp {\J}-1)$-dimensional compact faces 
i.e. the compact facets 
of $\Gamma_{+}^{\J}(f)$ and 
$\gamma_1^{\J},\gamma_2^{\J}, \ldots,\gamma_{n({\J})}^{\J}$ 
be such faces. 
For $1 \leq i \leq n(\J)$, let $u_i^{\J} \in (\RR^{\J})^* \cap \ZZ_+^{\J}$ be 
the unique non-zero primitive vector which takes its 
minimum in $\Gamma_{+}^{\J}(f)$ on the whole $\gamma_i^{\J}$ and set
\begin{equation}
d_i^{\J}: = \min_{v\in \Gamma_{+}^{\J}(f)} \langle u_i^{\J} ,v \rangle \ \in \ZZ_{>0}.
\end{equation}
We call $d_i^{\J}$ the lattice distance of $\gamma_i^{\J}$ from the origin $0 \in \RR^{\J}$. 
For each compact facet $\gamma_i^{\J} \prec \Gamma_{+}^{\J}(f)$, 
let $\LL(\gamma_i^{\J})$ be the smallest affine linear subspace of $\RR^n$ 
containing $\gamma_i^{\J}$ and $\Vol_{\ZZ}(\gamma_i^{\J}) \in \ZZ_{>0}$ the 
normalized $(\sharp \J -1)$-dimensional volume (i.e. the 
$(\sharp S-1)!$ times the usual volume) of $\gamma_i^{\J}$ 
with respect to the lattice $\ZZ^n \cap \LL(\gamma_i^{\J})$. 
Then we have the following celebrated theorem of 
Varchenko \cite{Varchenko}. 

\begin{theorem}\label{thm:3-5-ad}{\rm 
(Varchenko \cite[Theorem 4.1]{Varchenko})} 
Assume that $f$ is non-degenerate at the origin $0 \in \CC^n$. Then we have 
\begin{equation}
\zeta_{f,0}(t)=\prod_{{\J} \neq \emptyset } \ \zeta_{f, 0}^S(t),
\end{equation}
where for each non-empty subset $\J \subset \{1,2,\ldots,n\}$ we set
\begin{equation}
\zeta_{f, 0}^S(t):=
\prod_{i=1}^{n({\J})}
\ (1-t^{d_i^{\J}})^{(-1)^{\sharp {\J}-1}\Vol_{\ZZ}(\gamma_i^{\J})}.
\end{equation}
\end{theorem}

We shall give a short proof to Theorem \ref{thm:3-5-ad} along the 
same line as in \cite[Section 4]{M-T-new2}. 

\begin{proof}
For a subset $\J\subset \{1,2,\ldots,n\}$, let us set
\begin{equation}
T_{\J}:=\{ x=(x_1,x_2,\ldots,x_n) \in \CC^n \ | \ 
x_i=0 \ (i \notin \J ), \ 
x_i \not= 0 \ (i \in \J ) \} \subset \CC^n.
\end{equation}
Then we have $T_S \simeq ( \CC^*)^{\sharp S}$. Moreover 
for the closure $\overline{T_S}$ of $T_S$ in $\CC^n$ 
we have $\overline{T_S} \simeq \CC^{\sharp S}$. 
Let $i_S: \overline{T_S} \hookrightarrow \CC^n$ 
be the inclusion map. Then by Lemma \ref{lemma:2-8} and 
the decomposition $\CC^n \setminus \{ 0 \} 
 = \sqcup_{\J \not= \emptyset} T_S$ 
of $\CC^n \setminus \{ 0 \}$ we obtain 
\begin{align}
\zeta_{f,0}(t) & = \zeta_{f,0}( \CC_{\CC^n} )(t) = 
\zeta_{f,0}( \CC_{\CC^n \setminus \{ 0 \}} )(t) 
\\
& = \prod_{S \not= \emptyset} \ \zeta_{f,0}( \CC_{T_S} )(t)
\\
& = \prod_{S \not= \emptyset} \ \zeta_{f,0}( i_{S *} \CC_{T_S} )(t)
\\
& = \prod_{S \not= \emptyset} \ \zeta_{f \circ i_S,0}( \CC_{T_S} )(t), 
\end{align} 
where we used also Corollary \ref{new-corol}  
(resp. Theorem \ref{prp:2-99}) 
in the second (resp. last) equality. Hence it suffices to 
prove that for any non-empty subset $\J\subset \{1,2,\ldots,n\}$ 
we have 
\begin{equation}
\zeta_{f \circ i_S,0}( \CC_{T_S} )(t)= \zeta_{f, 0}^S(t)
=\prod_{i=1}^{n({\J})}
\ (1-t^{d_i^{\J}})^{(-1)^{\sharp {\J}-1}\Vol_{\ZZ}(\gamma_i^{\J})}.
\end{equation}
Here we shall prove it only for $\J = \{1,2,\ldots,n\}$. 
The proof for the other cases is similar. 
For $\J = \{1,2,\ldots,n\}$ we set 
\begin{equation}
T:=T_S=T_{\{1,2,\ldots,n\}} =( \CC^*)^n \subset \CC^n. 
\end{equation}
Let us consider $\CC^n$ as a toric variety associated to 
the fan $\Sigma_0$ in $( \RR^n)^*$ formed by the all faces of 
the first quadrant $( \RR^n)^*_+ \subset ( \RR^n)^*$. 
Then we regard $T = ( \CC^*)^n$ as the open dense torus in it. 
Let $\Sigma$ be a smooth subdivision of $\Sigma_0$ and 
the dual fan $\Sigma_1$ of $\Gamma^S_{+}(f) =\Gamma_{+}(f)$ in 
$( \RR^n)^*_+$. 
Let $X_{\Sigma}$ be the smooth 
toric variety  associated to $\Sigma$ and $T^{\prime} \subset 
X_{\Sigma}$ the open dense torus in it. Since $\Sigma$ 
is a subdivision of $\Sigma_0$, there exists a proper 
morphism $\pi : X_{\Sigma} \longrightarrow \CC^n$ of 
toric varieties. We can easily see that $\pi$ 
induces an isomorphism $T^{\prime} \simto T$. So in 
what follows, we identify $T^{\prime}$ with $T$ and 
denote it simply by $T$. Recall that $T^{\prime} \simeq T$ 
acts on $X_{\Sigma}$ and the $T$-orbits are parametrized by 
the cones in $\Sigma$. For a cone $\sigma \in \Sigma$ denote by 
$T_{\sigma} \simeq (\CC^*)^{n-\dim \sigma}$ the corresponding $T$-orbit. 
We have also natural affine open subsets $\CC^n(\sigma) 
\simeq \CC^n$ of $X_{\Sigma}$ associated to $n$-dimensional 
cones $\sigma$ in $\Sigma$. Now we apply Theorem \ref{prp:2-99}  
to the proper morphism $\pi : X_{\Sigma} \longrightarrow \CC^n$, 
the function $f \circ i_S=f : \CC^n \longrightarrow \CC$ 
and the constructible sheaf $\CC_T \in \Dbc (X_{\Sigma})$. 
Then for $S= \{1,2,\ldots,n\}$ we obtain an equality 
\begin{equation}
\zeta_{f \circ i_S,0}( \CC_{T_S} )= \zeta_{f}( \CC_{T} )(0)
=  \zeta_{f}( {\rm R} \pi_* \CC_{T} )(0)
= \int_{\pi^{-1}(0)}  \zeta_{f \circ \pi}( \CC_{T} ) 
\end{equation}
in $\CC (t)^*$. Hence for the calculation of 
$\zeta_{f \circ i_S,0}( \CC_{T_S} ) \in \CC (t)^*$ it 
suffices to calculate the value of the constructible 
function $\zeta_{f \circ \pi}( \CC_{T} )$ at 
each point of $\pi^{-1}(0) \subset X_{\Sigma}$. Note 
that $\pi^{-1}(0) \subset X_{\Sigma}$ is the union of 
the $T$-orbits $T_{\tau}$ in $X_{\Sigma}$ associated to the cones 
$\tau \in \Sigma$ such that 
${\rm rel.int} ( \tau ) \subset {\rm Int} ( \RR^n)^*_+
\simeq ( \RR_{>0})^n$.  
Let $\tau \in \Sigma$ be a cone such that 
${\rm rel.int} ( \tau ) \subset {\rm Int} ( \RR^n)^*_+$ 
and $\sigma \in \Sigma$ an $n$-dimensional cone 
such that $\tau \prec \sigma$. 
Let $w_1, w_2, \ldots, w_n \in \ZZ_+^n \setminus \{ 0 \}$ be the 
primitive vectors on the edges of 
the simplicial cone $\sigma$. 
Then there exists an affine open subset $\CC^n(\sigma)$ 
of $X_{\Sigma}$ such that $\CC^n(\sigma) \simeq \CC^n_y$ 
and $f \circ \pi$ has the following form on it:
\begin{equation} 
(f \circ \pi )(y)=\sum_{v \in \ZZ_+^n}a_{v}
y_1^{\langle w_1,v \rangle}\cdots y_n^{\langle w_n, v \rangle} =
y_1^{b_1} \cdots y_n^{b_n} \cdot f_{\sigma}(y), 
\end{equation}
where we set 
\begin{equation}
b_i=\min_{v\in \Gamma_{+}(f)} \langle w_i,v 
\rangle \geq 0 \qquad (i=1,2,\ldots,n)
\end{equation}
and $f_{\sigma}(y)$ is a polynomial on $\CC^n(\sigma) \simeq \CC^n_y$. 
Set $d:= {\rm dim} \tau$ and for simplicity, 
assume that $w_1,\ldots, w_d$ generate the simplicial cone $\tau$. 
Then in the affine chart $\CC^n(\sigma) \simeq \CC^n_y$ the $T$-orbit 
$T_{\tau}$ associated to $\tau$ is explicitly defined by
\begin{equation}
T_{\tau}=\{(y_1,\ldots,y_n)\in \CC^n(\sigma)\ 
 |\ y_1=\cdots =y_d=0,\ y_{d+1},\ldots, y_n\neq 0\}\simeq (\CC^*)^{n-d}.
\end{equation}
By the condition ${\rm rel.int} ( \tau ) \subset {\rm Int} ( \RR^n)^*_+$ 
the supporting face $\gamma ( \tau ) \prec \Gamma_+(f)$ of 
$\tau$ in $\Gamma_+(f)$ being compact, 
it follows form the non-degeneracy of $f$ that 
the complex hypersurface $\{ x \in (\CC^*)^n \ |\ 
f^{\gamma ( \tau )}(x)=0 \} 
\subset (\CC^*)^n$ is smooth. As 
the restriction of $f_{\sigma}$ to $T_{\tau} \subset 
\CC^n(\sigma)$ is naturally identified with 
$f^{\gamma ( \tau )}$, this implies that the hypersurface 
$f_{\sigma}^{-1}(0) \subset \CC^n(\sigma) \simeq \CC^n_y$ 
intersects $T_{\tau}$ transversally. 
Since we have $T_{\tau} \subset \pi^{-1}(0) \subset 
(f \circ \pi )^{-1}(0)$ and hence the function 
$f \circ \pi$ is zero on $T_{\tau}$, 
we have $b_i>0$ for some $1 \leq i \leq d$. Then by 
Lemmas \ref{lem:2-ac-1} and \ref{lem:2-ac-21}, 
if $d \geq 2$ for any point 
$y \in T_{\tau}$ we have 
$\zeta_{f \circ \pi}( \CC_{T} )(y)= \zeta_{f \circ \pi, y}( \CC_{T} )
(t)=1$. By Lemma \ref{lem:2-ac-21}, if $d=1$ for any point 
$y \in T_{\tau}$ we have 
\begin{equation}
\zeta_{f \circ \pi}( \CC_{T} )(y)= \zeta_{f \circ \pi, y}( \CC_{T} )
(t)
=\begin{cases}
1-t^{b_1} & (y \in T_{\tau} \setminus f_{\sigma}^{-1}(0) ),\\
 & \\
1 & (y \in T_{\tau} \cap f_{\sigma}^{-1}(0)).
\end{cases}
\end{equation}
Moreover, if $d=1$ and the dimension of the supporting face 
$\gamma ( \tau ) \prec \Gamma_+(f)$ of $\tau$ in $\Gamma_+(f)$ 
is less than 
$n- {\rm dim} \tau =n-1$, then we have 
$T_{\tau} \cap f_{\sigma}^{-1}(0) \simeq 
A \times ( \CC^*)^{n-1- {\rm dim} \gamma ( \tau )}$ 
for some set $A$ and hence we get 
\begin{equation}
\chi ( T_{\tau} \setminus f_{\sigma}^{-1}(0) ) = 
\chi ( T_{\tau} ) - \chi ( T_{\tau} \cap f_{\sigma}^{-1}(0) ) 
=0-0=0. 
\end{equation}
This implies that $\int_{T_{\tau}}\zeta_{f \circ \pi}( \CC_{T} )$ 
is not trivial only in the case where 
$\gamma ( \tau ) \prec \Gamma_+(f)$ is equal to 
the facet $\gamma_i^S \prec \Gamma_+(f)$ for 
some $1 \leq i \leq n(S)$. In such a case, we have 
$d=1$, $b_1=d_i^S>0$ and moreover 
\begin{equation}
\chi ( T_{\tau} \setminus f_{\sigma}^{-1}(0) ) 
=0-(-1)^{n-2} \ \Vol_{\ZZ}(\gamma_i^{\J}) 
=(-1)^{\sharp S-1} \ \Vol_{\ZZ}(\gamma_i^{\J}) 
\end{equation}
by Theorem \ref{BKK-simple}. Then we obtain the desired equality 
\begin{equation}
 \int_{\pi^{-1}(0)}  \zeta_{f \circ \pi}( \CC_{T} ) =
\prod_{i=1}^{n({\J})} \ (1-t^{d_i^{\J}})^{(-1)^{\sharp {\J}-1}
\Vol_{\ZZ}(\gamma_i^{\J})}.
\end{equation}
This completes the proof. 
\qed 
\end{proof}

The following lemma is well-known to the specialists. 
For the convenience of the readers, here we give a 
proof to it. 

\begin{lemma}\label{leme:3-5-ad}
Assume that $f$ is non-degenerate at the origin $0 \in \CC^n$ 
and convenient. Then the complex hypersurface 
$f^{-1}(0) \subset \CC^n$ is smooth or has an isolated 
singular point at the origin $0 \in \CC^n$. 
\end{lemma}

\begin{proof}
Let us consider $\CC^n$ as the toric variety associated to 
the fan $\Sigma_0$ formed by the all faces of the first 
quadrant $(\RR^n)^*_+ \subset (\RR^n)^*$. Recall that 
we denote by $\Sigma_1$ the dual fan of 
$\Gamma_+(f)$ in $(\RR^n)^*_+$. 
Then by our assumptions, we can easily see that 
any cone $\sigma \in \Sigma_0$ in $\Sigma_0$ such 
that $\sigma \subset \partial (\RR^n)^*_+ 
\simeq ( \RR^n_+ \setminus {\rm Int} (\RR^n_+) ) $ is 
an element of $\Sigma_1$. Hence we can construct a 
smooth subdivision $\Sigma$ of $\Sigma_1$ without 
subdividing the cones $\sigma \in \Sigma_1$ such 
that $\sigma \subset \partial (\RR^n)^*_+$ 
(see e.g. \cite[Chapter II, Lemma (2.6)]{Oka}). 
Let $X_{\Sigma}$ be the smooth toric variety associated 
to $\Sigma$. Then there exists a morphism 
$\pi_0: X_{\Sigma} \rightarrow \CC^n$
of toric varieties 
inducing an isomorphism 
$X_{\Sigma} \setminus \pi_0^{-1}( \{ 0 \} ) 
\simto \CC^n \setminus \{ 0 \}$. Let 
$\rho_1, \rho_2, \ldots, \rho_m \in \Sigma$ be 
the $1$-dimensional cones in $\Sigma$ 
not contained in $\partial (\RR^n)^*_+$. 
For each $\rho_i$ denote by $T_i \simeq ( \CC^*)^{n-1}$ 
the $T$-orbit in $X_{\Sigma}$ and 
set $D_i:= \overline{T_i} \subset X_{\Sigma}$. 
Then $D_i$ are smooth and we have 
$\pi_0^{-1}( \{ 0 \} ) =D_1 \cup D_2 \cup 
\cdots \cup D_m$. Moreover by the non-degeneracy of 
$f$ at $0 \in \CC^n$, for any non-empty subset 
$I \subset \{ 1,2, \ldots, m \}$ the 
smooth subvariety $D_I:= \cap_{i \in I} D_i$ of 
$X_{\Sigma}$ intersects the proper transform of 
$f^{-1}(0) \subset \CC^n$ in $X_{\Sigma}$ 
transversally. This in particular implies that 
the proper transform itself is smooth. 
Then the assertion is very clear. 
\qed 
\end{proof}

For a non-empty subset ${\J} \subset \{1,2,\ldots, n\}$ 
we set 
\begin{equation}
E_S:= \dsum_{i=1}^{n(S)} d_i^S \cdot \Vol_{\ZZ}(\gamma_i^{\J}) 
\quad \geq 0. 
\end{equation}
Moreover for the empty set $S= \emptyset$ we set 
$E_{\emptyset} := 1$. 
Then by Theorem \ref{thm:3-5-ad} we obtain the following 
fundamental theorem of Kouchnirenko \cite{Kushnirenko}.  

\begin{theorem} {\rm (Kouchnirenko \cite[Theorem 1.10]{Kushnirenko})} 
Assume that $f$ is non-degenerate at the origin $0 \in \CC^n$. Then we have 
\begin{equation}
\chi (F_0)= \dsum_{{\J} \neq \emptyset} 
(-1)^{\sharp {\J}-1} E_S
\end{equation}
In particular, if moreover $f$ is convenient and 
$f^{-1}(0)$ is singular at $0 \in \CC^n$, then 
$f^{-1}(0)$ has an isolated 
singular point at $0 \in \CC^n$ and the Milnor 
number $\mu >0$ there is given by 
\begin{equation}
\mu = \dsum_{{\J} \subset \{1,2,\ldots, n\}} 
(-1)^{n- \sharp {\J}} E_S \quad 
=E_{\{1,2, \ldots, n \}} + \cdots \cdots +(-1)^n 
>0. 
\end{equation}
\end{theorem}

From now on, we shall introduce Oka's theorem proved 
in \cite{O-1} and \cite{Oka} (see also Kirillov \cite{Kirillov}). 
As we will see below, it is a wide generalization of Theorem \ref{thm:3-5-ad}. 
Since some technical assumptions imposed in Oka's theorem in 
\cite{O-1} and \cite{Oka} were 
removed later in Matsui-Takeuchi \cite[Theorem 3.12]{M-T-new2}, 
here we present it in its final form in 
\cite[Theorem 3.12]{M-T-new2}. Note that 
in \cite[Theorem 3.12]{M-T-new2} more general 
results on the Milnor monodromies of regular 
functions on complete intersection subvarieties 
of (possibly non-normal)  
affine toric varieties were obtained. 
In \cite{M-T-new1}, they were used to obtain 
formulas for the degrees and the dimensions of 
the $A$-discriminant varieties introduced by 
Gelfand-Kapranov-Zelevinsky \cite{GKZ}. However, 
for the sake of simplicity, here we present only the 
results in the simplest case of regular functions on 
complete intersection subvarieties of $\CC^n$. First, 
let us recall 
the most general form of the 
Bernstein-Khovanskii-Kouchnirenko 
theorem \cite{Khovanskii}. 
Let $( \RR^n)^* \simeq \RR^n$ be the dual vector space 
of $\RR^n$ and 
\begin{equation}
\langle \cdot , \cdot  \rangle : 
 ( \RR^n)^* \times \RR^n  \longrightarrow \RR, \quad 
((u,v) \longmapsto \langle u, v \rangle := 
\sum_{i=1}^n u_iv_i )
\end{equation}
the canonical pairing. Let $\Delta \subset \RR^n$ be a 
lattice polytope 
in $\RR^n$. For $u \in ( \RR^n)^* \simeq \RR^n$ we define the 
supporting face $\gamma^u \prec  \Delta$ 
of $u$ in $ \Delta$ by 
\begin{equation}
\gamma^u = \left\{ v \in \Delta \ | \ 
\langle u , v \rangle 
= 
\min_{w \in \Delta } 
\langle u ,w \rangle \right\}. 
\end{equation} 
For a face $\gamma$ of $\Delta$ set 
\begin{equation}
\sigma (\gamma) := \overline{ \{ u \in ( \RR^n)^* \ | \ 
\gamma^u = \gamma   \} } \subset ( \RR^n)^*. 
\end{equation}
\noindent Then  $\sigma (\gamma )$ 
is an $(n- \dim \gamma )$-dimensional 
rational convex polyhedral 
cone in $( \RR^n)^*$. Moreover 
the family $\{ \sigma (\gamma ) \ | \ 
\gamma \prec  \Delta \}$ of cones in $( \RR^n)^*$   
thus obtained is a subdivision of $( \RR^n)^*$. 
We call it the dual subdivision of $( \RR^n)^*$ by 
$\Delta$. If $\dim \Delta =n$, the cones 
$\sigma ( \gamma )$ are strongly convex and it 
satisfies the axiom 
of fans. Then we call it the dual fan of 
$\Delta$. More generally, let $\Delta_1, \ldots, \Delta_p  
\subset \RR^n$ be lattice polytopes 
in $\RR^n$ and $\Delta = 
\Delta_1 + \cdots + \Delta_p  
\subset \RR^n$ their Minkowski sum. 
Then for a face $\gamma \prec \Delta$ of 
$\Delta$, by taking a point $u \in ( \RR^n)^*$ 
in the relative interior of its dual cone $\sigma (\gamma)$ 
we define the supporting face 
$\gamma_i \prec \Delta_i$ of $u$ in $\Delta_i$ 
so that we have 
$\gamma = \gamma_1 + \cdots + \gamma_p$. 

\begin{definition}\label{non-deg} 
(see \cite{Oka}) 
Let $g_1, g_2, \ldots , g_p$ be 
Laurent polynomials on $T=(\CC^*)^n$. 
Set $\Delta_i=NP(g_i)$ $(i=1,\ldots, p)$ and 
$\Delta = \Delta_1 + \cdots + \Delta_p$. 
Then we say that the subvariety 
$Z^* =\{ x\in T=(\CC^*)^n \ | \ g_1(x)=g_2(x)= 
\cdots =g_p(x)=0 \}$ of $T=(\CC^*)^n$ is a 
non-degenerate complete intersection 
if for any face $\gamma \prec \Delta$ of 
$\Delta$ the $p$-form $dg_1^{\gamma_1} \wedge 
dg_2^{\gamma_2} \wedge 
\cdots \wedge dg_p^{\gamma_p}$ does not vanish 
on $\{ x\in T=(\CC^*)^n  \ | \ 
g_1^{\gamma_1}(x)= \cdots =
g_p^{\gamma_p}(x)=0 \}$.
\end{definition}

\begin{definition}\label{def-1}
Let $\Delta_1,\ldots,\Delta_n$ 
be lattice 
polytopes in $\RR^n$. Then 
their normalized $n$-dimensional 
mixed volume 
$\Vol_{\ZZ}( \Delta_1,\ldots,\Delta_n) 
\in \ZZ$ is defined by the formula 
\begin{equation}
\Vol_{\ZZ}( \Delta_1, \ldots , \Delta_n)=
\frac{1}{n!} 
\dsum_{k=1}^n (-1)^{n-k} 
\sum_{\begin{subarray}{c}I\subset 
\{1,\ldots,n\}\\ |I| =k\end{subarray}}
\Vol_{\ZZ}\left(
\dsum_{i\in I} \Delta_i \right)
\end{equation}
where $\Vol_{\ZZ}( \cdot )
= n! \cdot \Vol ( \cdot ) \in \ZZ_+$ is 
the normalized $n$-dimensional volume 
with respect to the lattice $\ZZ^n 
\subset \RR^n$.
\end{definition}

\begin{theorem}\label{BKK}\label{thm:14}  
{\rm (Bernstein-Khovanskii-Kouchnirenko 
theorem \cite{Khovanskii}, see also 
Oka \cite[Chapter IV, Theorem (3.1)]{Oka})} 
Let $g_1, g_2, \ldots , g_p$ be 
Laurent polynomials on $T=(\CC^*)^n$. 
Assume that the subvariety $Z^* 
=\{ x\in T=(\CC^*)^n \ | \ g_1(x)=g_2(x)= 
\cdots =g_p(x)=0 \}$ of $T=(\CC^*)^n$ is a 
non-degenerate complete intersection. 
Set $\Delta_i=NP(g_i)$ $(i=1,\ldots, p)$. Then we have
\begin{equation}
\chi(Z^* )=(-1)^{n-p} 
\dsum_{\begin{subarray}{c} 
m_1,\ldots,m_p \geq 1\\ m_1+\cdots +m_p=n 
\end{subarray}}\Vol_{\ZZ}(
\underbrace{\Delta_1,\ldots,\Delta_1}_{\text{
$m_1$-times}},\ldots, 
\underbrace{\Delta_p,
\ldots,\Delta_p}_{\text{$m_p$-times}}),
\end{equation}
where $\Vol_{\ZZ}(\underbrace{\Delta_1,
\ldots,\Delta_1}_{\text{$m_1$-times}},
\ldots,\underbrace{\Delta_p,\ldots,
\Delta_p}_{\text{$m_p$-times}})\in \ZZ_+$ 
is the normalized $n$-dimensional mixed volume 
with respect to the lattice $\ZZ^n 
\subset \RR^n$.
\end{theorem}

Now let $f_1,f_2, \ldots, f_k \in \CC [x]= \CC 
[x_1,x_2, \ldots, x_n]$ be polynomials on 
$\CC^n$ and consider the following 
two subvarieties of $\CC^n$: 
\begin{equation}
V:= \{ f_1=f_2= \cdots =f_k=0 \} \ \subset 
W:= \{ f_1=f_2= \cdots =f_{k-1}=0 \}.  
\end{equation}
Assume that $V \not= W$ and $0 \in V$. Then for the non-constant 
regular function $g:=f_k|_W: W \longrightarrow \CC$ on $W \subset \CC^n$,  
let us consider the Milnor fiber $F_0 \subset (W \setminus V)$ and 
the Milnor monodromies 
\begin{equation}
\Phi_{j,0} \colon H^j(F_0; \CC ) \simto H^j(F_0; \CC ) \qquad (j\in\ZZ) 
\end{equation}
of $g$ at the origin $0 \in \CC^n$. Our objective here 
is to describe the monodromy zeta function $\zeta_{g,0}(t) \in \CC (t)^*$ 
of $g$ in terms of the Newton polyhedra $\Gamma_+(f_j) \subset 
\RR^n_+$ of $f_j$ ($1 \leq j \leq k$). 
For a subset ${\J} \subset \{1,2,\ldots, n\}$ such that 
$\Gamma_+^S(f_k)= 
\Gamma_+(f_k) \cap \RR^S \not= \emptyset$ we set 
\begin{equation}
I(S):= \{ 1 \leq j \leq k-1 \ | \ 
\Gamma_+^S(f_j) \not= \emptyset \} \quad 
\subset \{ 1,2, \ldots, k-1 \} 
\end{equation}
and $m(S):= \sharp I(S)+1$. Moreover, for such $S$, 
$u \in {\rm Int} ( \RR^S)^*_+ \subset ( \RR^S)^*$ 
and $j \in I(S) \sqcup \{ k \}$ we define 
the supporting face $\gamma_j^u \prec \Gamma_+^S(f_j) 
\subset \RR_+^S$ of $u$ in $\Gamma_+^S(f_j)$ as before 
and set 
\begin{equation}
f_j^u(x):= \dsum_{v \in \gamma_j^u \cap \ZZ^n_+} a_{j,v} x^v \quad 
\in \CC [x], 
\end{equation}
where $f_j(x)= \sum_{v \in \ZZ^n_+} a_{j,v} x^v$ 
($a_{j,v} \in \CC$). 

\begin{definition}\label{de-no-deg} {\rm (Oka \cite{O-1}, \cite{Oka}, 
see also \cite[Definition 3.9]{M-T-new2})} 
We say that $f=(f_1,f_2, \ldots, f_k)$ 
is non-degenerate (at the origin $0 \in \CC^n$) if 
for any subset ${\J} \subset \{1,2,\ldots, n\}$ such that 
$\Gamma_+^S(f_k) \not= \emptyset$ and 
any $u \in {\rm Int} ( \RR^S)^*_+$ the following 
two subvarieties of $( \CC^*)^n$ are 
non-degenerate complete intersections: 
\begin{align}
& \{ x=(x_1, \ldots, x_n)  \in (\CC^*)^n \ | \ f_j^u(x)=0 
 \ (j \in I(S) \sqcup \{ k \} ) \},  
\\
& \{ x=(x_1, \ldots, x_n)  \in (\CC^*)^n \ | \ f_j^u(x)=0 
 \ (j \in I(S) ) \}. 
\end{align} 
\end{definition}
For each subset ${\J} \subset \{1,2,\ldots, n\}$ such that 
$\Gamma_+^S(f_k) \not= \emptyset$, let us 
consider the Minkowski sum 
\begin{equation}
\Gamma_+^S(f):= \dsum_{j \in I(S) \sqcup \{ k \}} \ \Gamma_+^S(f_j) \ 
\subset \RR^S_+. 
\end{equation}
Then, let 
$\gamma_1^{\J},\gamma_2^{\J}, \ldots,\gamma_{n({\J})}^{\J}$ be the 
compact facets 
of $\Gamma_{+}^{\J}(f)$. 
For $1 \leq i \leq n(\J)$, let $u_i^{\J} \in (\RR^{\J})^* \cap \ZZ_+^{\J}$ be 
the unique non-zero primitive vector which takes its 
minimum in $\Gamma_{+}^{\J}(f)$ on the whole $\gamma_i^{\J}$. 
Then for any $1 \leq i \leq n(\J)$ the supporting faces 
\begin{equation}
\gamma_{i,j}^S:= \gamma_j^{u_i^S} \prec  \Gamma_+^S(f_j) \qquad  
(j \in I(S) \sqcup \{ k \} ) 
\end{equation}
of $u_i^S$ in $\Gamma_+^S(f_j)$ satisfy the condition 
\begin{equation}
\gamma_i^S= \dsum_{j \in I(S) \sqcup \{ k \}} \ \gamma_{i,j}^S. 
\end{equation}
For $1 \leq i \leq n(\J)$ we set also 
\begin{equation}
d_i^{\J}: = \min_{v\in \Gamma_{+}^{\J}(f_k)} \langle u_i^{\J} ,v 
\rangle \ \in \ZZ_{>0}
\end{equation} 
and 
\begin{equation}
K_i^S:= 
\dsum_{\begin{subarray}{c} 
\alpha_1 + \cdots + \alpha_{m(S)}= \sharp S-1, 
\\ 
\alpha_q \geq 1 \ (q \leq m(S)-1), \ \alpha_{m(S)} \geq 0 
\end{subarray}} 
\ 
\Vol_{\ZZ}(
\underbrace{\gamma_{i,j_1}^S,\ldots, \gamma_{i,j_1}^S}_{\text{
$\alpha_1$-times}},\ldots, 
\underbrace{\gamma_{i,j_{m(S)}}^S,
\ldots, \gamma_{i,j_{m(S)}}^S}_{\text{$\alpha_{m(S)}$-times}}), 
\end{equation}
where we set $I(S) \sqcup \{ k \} = 
\{ j_1, j_2, \ldots, j_{m(S)-1}, k=j_{m(S)} \}$ and 
\begin{equation}
\Vol_{\ZZ}(
\underbrace{\gamma_{i,j_1}^S,\ldots, \gamma_{i,j_1}^S}_{\text{
$\alpha_1$-times}},\ldots, 
\underbrace{\gamma_{i,j_{m(S)}}^S,
\ldots, \gamma_{i,j_{m(S)}}^S}_{\text{$\alpha_{m(S)}$-times}}) 
\ \in \ZZ_+
\end{equation}
is the normalized $( \sharp S-1)$-dimensional mixed volume 
with respect to the lattice 
$\ZZ^n \cap \LL(\gamma_i^{\J}) \simeq \ZZ^{\sharp S-1}$ in 
$\LL(\gamma_i^{\J}) \simeq \RR^{\sharp S-1}$. 

\begin{remark}
If $\sharp S-1=0$ and $\sharp I(S)=m(S)-1=0$, we set 
\begin{equation}
K_i^S= 
\Vol_{\ZZ}(
\underbrace{\gamma_{i,k}^S,\ldots, \gamma_{i,k}^S}_{\text{
$0$-times}}) =1 
\end{equation}
(in this case $\gamma_{i,k}^S$ is a point). 
\end{remark}

\begin{theorem}\label{thm:3-5-CI}{\rm (Oka \cite{O-1}, 
\cite[Chapter IV]{Oka}, 
see also \cite[Theorem 3.12]{M-T-new2})} 
Assume that $f=(f_1,f_2, \ldots, f_k)$  
is non-degenerate at the origin $0 \in \CC^n$. Then 
the monodromy zeta function $\zeta_{g,0}(t) \in \CC (t)^*$ 
of $g=f_k|_W: W \longrightarrow \CC$ at $0 \in \CC^n$ is given by 
\begin{equation}
\zeta_{g,0}(t)=
\dprod_{S : \ \Gamma_+^S(f_k) \not= \emptyset, \  m(S) \leq \sharp S}
 \ \zeta_{g, 0}^S(t),
\end{equation}
where for each subset $\J \subset \{1,2,\ldots,n\}$ 
such that $\Gamma_+^S(f_k) \not= \emptyset$ and 
$m(S) \leq \sharp S$ we set
\begin{equation}
\zeta_{g, 0}^S(t):=
\prod_{i=1}^{n({\J})}
\ (1-t^{d_i^{\J}})^{(-1)^{\sharp {\J}-m(S)} K_i^S}.
\end{equation}
\end{theorem}

\begin{proof}
The proof is similar to that of 
Theorem \ref{thm:3-5-ad}. Let $T_S, i_S$ 
etc. be as in the proof of Theorem \ref{thm:3-5-ad}. 
Then by Corollary \ref{new-corol},  
Lemma \ref{lemma:2-8}, Theorem \ref{prp:2-99} and  
the decomposition $W \setminus \{ 0 \} 
 = \sqcup_{\J \not= \emptyset}(W \cap T_S)$ 
of $W \setminus \{ 0 \}$ we obtain 
\begin{equation}
\zeta_{g,0}(t) = \zeta_{f_k,0}( \CC_W) (t)
= \prod_{S \not= \emptyset} \ 
\zeta_{f_k \circ i_S,0}( \CC_{W \cap T_S} )(t). 
\end{equation} 
If $\Gamma_+^S(f_k) = \emptyset$ we have 
$f_k \circ i_S=0$ and hence 
$\zeta_{f_k \circ i_S,0}( \CC_{W \cap T_S} )(t)=1$.  
Then it suffices to take the product of 
$\zeta_{f_k \circ i_S,0}( \CC_{W \cap T_S} )(t)$ 
over the subsets $S \subset \{ 1,2, \ldots, n \}$ such 
that $\Gamma_+^S(f_k)  \not= \emptyset$. 
Moreover, for such $S$ we have 
\begin{equation}
W \cap T_S= \{ x \in T_S \ | \ 
f_j(x)=0 \ (j \in I(S)) \}. 
\end{equation}
Fix a subset $\J = \{1,2,\ldots,n\}$ 
such that $\Gamma_+^S(f_k) \not= \emptyset$, 
set $l:= \sharp S >0$ and 
consider $\overline{T_S} \simeq \CC^l$ as the toric variety associated to 
the fan $\Sigma_0$ in $( \RR^S)^* \simeq \RR^l$ formed by the all faces of 
the first quadrant $( \RR^S)^*_+ \subset ( \RR^S)^*$. 
Then we regard $T_S \simeq ( \CC^*)^l$ as the open dense torus in it. 
Let $\Sigma$ be a smooth subdivision of $\Sigma_0$ and 
the dual fan of $\Gamma^S_{+}(f) \subset \RR^S_+$ in $( \RR^S)^*_+$. 
Let $X_{\Sigma}$ be the smooth 
toric variety  associated to $\Sigma$. 
Then there exists a proper 
morphism $\pi : X_{\Sigma} \longrightarrow 
\overline{T_S}$ of 
toric varieties. For a cone $\sigma \in \Sigma$ denote by 
$T_{\sigma} \simeq (\CC^*)^{l-\dim \sigma}$ the $T_S$-orbit 
in $X_{\Sigma}$ associated to $\sigma$. 
We have also natural affine open subsets $\CC^l(\sigma) 
\simeq \CC^l$ of $X_{\Sigma}$ associated to $l$-dimensional 
cones $\sigma$ in $\Sigma$. Now apply Theorem \ref{prp:2-99}  
to the proper morphism $\pi : X_{\Sigma} \longrightarrow 
 \overline{T_S}$, 
the function $f_k \circ i_S : \overline{T_S} 
 \longrightarrow \CC$ 
and the constructible sheaf $\CC_{W \cap T_S} \in \Dbc (X_{\Sigma})$. 
Then we obtain an equality 
\begin{equation}
\zeta_{f_k \circ i_S,0}( \CC_{W \cap T_S} )
=  \zeta_{f_k \circ i_S}( {\rm R} \pi_* \CC_{W \cap T_S} )(0)
= \int_{\pi^{-1}(0)}  \zeta_{f_k \circ i_S \circ \pi}( \CC_{W \cap T_S} ) 
\end{equation}
in $\CC (t)^*$. Hence it 
suffices to calculate the value of the constructible 
function $\zeta_{f_k \circ i_S \circ \pi}( \CC_{W \cap T_S} )$ at 
each point of $\pi^{-1}(0) \subset X_{\Sigma}$. Note 
that $\pi^{-1}(0)$ is the union of 
the $T_S$-orbits $T_{\tau}$ in $X_{\Sigma}$ associated to the cones 
$\tau \in \Sigma$ such that 
${\rm rel.int} ( \tau ) \subset {\rm Int} ( \RR^S)^*_+
\simeq ( \RR_{>0})^l$. 
Let $\tau \in \Sigma$ be such a cone 
and $\sigma \in \Sigma$ an $l$-dimensional cone 
such that $\tau \prec \sigma$. 
Let $w_1, w_2, \ldots, w_l \in \ZZ_+^S \setminus \{ 0 \}$ be the 
primitive vectors on the edges of 
the simplicial cone $\sigma$. 
Then there exists an affine open subset $\CC^l(\sigma)$ 
of $X_{\Sigma}$ such that $\CC^l(\sigma) \simeq \CC^l_y$ 
and the functions $f_j \circ i_S  \circ \pi$ 
($j \in I(S) \sqcup \{ k \}$) 
have the following form on it:
\begin{equation} 
(f_j \circ i_S \circ \pi )(y)=
y_1^{b_{j,1}} \cdots y_l^{b_{j,l}} \cdot f_{j,\sigma}(y), 
\end{equation}
where we set 
\begin{equation}
b_{j,i}=\min_{v \in \Gamma_{+}^S(f_j)} \langle w_i, v 
\rangle \geq 0 \qquad (i=1,2,\ldots, l)
\end{equation}
and $f_{j, \sigma}(y)$ is a polynomial on $\CC^l(\sigma) \simeq \CC^l_y$. 
Set $d:= {\rm dim} \tau$ and for simplicity, 
assume that $w_1,\ldots, w_d$ generate the simplicial cone $\tau$. 
Then in the affine chart $\CC^l(\sigma) \simeq \CC^l_y$ the $T_S$-orbit 
$T_{\tau}$ associated to $\tau$ is explicitly defined by
\begin{equation*}
T_{\tau}=\{(y_1,\ldots,y_l)\in \CC^l(\sigma)\ 
 |\ y_1=\cdots =y_d=0,\ y_{d+1},\ldots, y_l \neq 0\}\simeq (\CC^*)^{l-d}.
\end{equation*}
For $j \in I(S) \sqcup \{ k \}$ let $h_j: T_{\tau} \longrightarrow 
\CC$ be the restriction of $f_{j, \sigma}: \CC^l( \sigma ) \longrightarrow 
\CC$ to $T_{\tau} \subset \CC^l( \sigma )$. Then 
by the non-degeneracy of $f=(f_1,f_2, \ldots, f_k)$, 
the following two subvarieties of $T_{\tau}$ are 
non-degenerate complete intersections: 
\begin{align}
V_{\tau}:= & \{ x  \in T_{\tau} \ | \ h_j (x)=0 
 \ (j \in I(S) \sqcup \{ k \} ) \},  
\\
W_{\tau}:= & \{ x  \in T_{\tau} \ | \ h_j (x)=0 
 \ (j \in I(S) ) \}. 
\end{align} 
This implies that the subvariety 
$\cap_{j \in I(S) \sqcup \{ k \}} f_{j, \sigma}^{-1}(0)$ 
(resp. $\cap_{j \in I(S)} f_{j, \sigma}^{-1}(0)$) 
of $\CC^l( \sigma )$ is smooth and of codimension 
$m(S)$ (resp. $m(S)-1= \sharp I(S)$) in $\CC^l( \sigma )$ 
on a neighborhood of $T_{\tau} \subset \CC^l( \sigma )$. 
Note also that for the closure $\overline{W \cap T_S}$ 
of $W \cap T_S$ in $\CC^l( \sigma )$ we have 
\begin{equation}
\overline{W \cap T_S} = 
\bigcap_{j \in I(S)} f_{j, \sigma}^{-1}(0)
\end{equation}
on a neighborhood of $T_{\tau}$. 
Then as in the proof of Theorem \ref{thm:3-5-ad}, 
it suffices to consider only the $1$-dimensional 
cones $\tau \in \Sigma$ such that the supporting 
face $\gamma ( \tau ) \prec \Gamma_+^S(f)$ of 
$\tau$ in $\Gamma_+^S(f)$ is equal to 
the facet $\gamma_i^S \prec \Gamma_+(f)$ for 
some $1 \leq i \leq n(S)$. In such a case, we have 
$d=1$, $b_{k,1}=d_i^S>0$. But, if $m(S)> \sharp S$, 
then $\sharp I(S)=m(S)-1> \sharp S-1= {\rm dim } 
T_{\tau}$ and hence $\cap_{j \in I(S)} f_{j, \sigma}^{-1}(0)$ 
does not intersect $T_{\tau}$. So we do not have to 
consider the subsets $S \subset \{ 1,2, \ldots, n \}$ 
such that $m(S)> \sharp S$. 
Moreover, for the $1$-dimensional cone $\tau \in \Sigma$ 
such that $\gamma ( \tau )= \gamma_i^S$, by Theorem \ref{BKK} 
we have 
\begin{equation}
\chi ( W_{\tau} \setminus V_{\tau} )= 
(-1)^{\sharp S- m(S)} K_i^S. 
\end{equation}
Then the assertion immediately follows. 
This completes the proof. 
\qed 
\end{proof}

As in the proof of Lemma \ref{leme:3-5-ad} we can 
easily prove the following result.  

\begin{lemma}\label{leme:31-5-ad}
Assume that $f=(f_1,f_2, \ldots, f_k)$ is non-degenerate at 
the origin $0 \in \CC^n$ 
and $f_1,f_2, \ldots, f_k$ are convenient. Then the complex 
subvarieties 
$V, W \subset \CC^n$ are smooth or have an isolated 
singular point at the origin $0 \in \CC^n$ and 
satisfy the conditions ${\rm dim} V=n-k$ and 
${\rm dim} W=n-k+1$. 
\end{lemma}

Together with the celebrated generalization of 
Theorem \ref{thm.7.2.003} by Hamm \cite{Hamm}, 
in the situation of Lemma \ref{leme:31-5-ad} 
we obtain the characteristic polynomial of 
the $(n-k)$-th monodromy operator 
\begin{equation}
\Phi_{n-k,0}: H^{n-k}(F_0; \CC ) \simto 
H^{n-k}(F_0; \CC )
\end{equation}
by Theorem \ref{thm:3-5-CI}.

\section{Singularities and 
monodromies at infinity of polynomial maps}\label{section 4}

In this section, we introduce some basic definitions 
and results on 
singularities and monodromies at infinity of polynomial maps 
$f \colon \CC^n \longrightarrow \CC$. In particular, here we focus on 
N{\'e}methi-Zaharia's theorem in \cite{N-Z} and 
Libgober-Sperber's one in  \cite{L-S}. 
After the two fundamental papers Broughton 
\cite{Broughton} and Siersma-Tib{\u a}r \cite{S-T-1}, many authors 
studied the global behavior of polynomial maps 
$f \colon \CC^n \longrightarrow \CC$. 
First of all, for a 
polynomial map $f \colon \CC^n \longrightarrow \CC$, 
it is well-known that there 
exists a finite subset $B \subset \CC$ such that the restriction
\begin{equation}
\CC^n \setminus f^{-1}(B) \longrightarrow \CC \setminus B
\end{equation}
of $f$ is a $C^{\infty}$ locally trivial fibration. 
We denote by $B_f$ the smallest 
subset $B \subset \CC$ satisfying this 
condition. Let ${\rm Sing} f \subset \CC^n$ 
be the set of the critical points of 
$f \colon \CC^n \longrightarrow \CC$. Then by 
the definition of $B_f$, obviously we have 
$f( {\rm Sing} f) \subset B_f$. 
We call the elements of $B_f$ the 
bifurcation values of $f$. 
The determination of the bifurcation set 
$B_f \subset \CC$ is a fundamental problem and has been  
studied by a lot of mathematicians such as \cite{Broughton}, 
\cite{H-L}, \cite{Ishikawa}, \cite{I-N-P}, 
\cite{N-Z}, \cite{N-Z-2}, \cite{Paru}, \cite{S-T-1}, \cite{Suzuki}, 
\cite{Tak-1}, \cite{Tanabe-G}, \cite{Zaharia} and from 
various points of view. 
For the generalizations of some of 
these results to more general polynomial maps 
$f \colon \CC^n \longrightarrow \CC^k$ 
for $1 \leq k \leq n$, 
see also \cite{C-D-T-T}, \cite{H-N}, \cite{KOS}, \cite{Nguyen} 
and \cite{Ra}. 
The essential difficulty in determining $B_f \subset \CC$ 
lies in the fact that $f \colon \CC^n \longrightarrow \CC$ 
is not proper and hence 
we may have $f( {\rm Sing} f) \not= B_f$ i.e. 
$f$ has some extra singularities at infinity. 

\begin{definition} {\rm (Broughton \cite{Broughton})} \label{dfn:tame}
Let $\partial f\colon \CC^n \longrightarrow \CC^n$ 
be the map defined by $\partial f(x)=(\partial_1f(x), \ldots, \partial_n f(x))$ 
($x \in \CC^n$). 
Then we say that $f$ is tame at infinity if the restriction $(\partial f)^{-1}(B(0;\e )) 
\longrightarrow B(0;\e )$ of $\partial f$ to a sufficiently 
small ball $B(0;\e ) \subset \CC^n$ centered at the origin $0 \in \CC^n$ is proper. 
\end{definition}

In \cite{Broughton} Broughton showed that if 
$f \colon \CC^n \longrightarrow \CC$ is tame at infinity 
then we have $f( {\rm Sing} f) = B_f$. 
In the tame case, the following result is also fundamental. 

\begin{theorem} {\rm (Broughton \cite{Broughton})} \label{tame}
Assume that $f$ is tame at infinity. 
Then the generic fiber $f^{-1}(c)$ ($c \in \CC \setminus B_f$) 
has the homotopy type of the bouquet of some $(n-1)$-spheres. In particular, we have
\begin{equation}
H^j(f^{-1}(c);\CC)=0 \quad (j \neq 0, n-1).
\end{equation}
\end{theorem}
By Siersma and Tib\u ar \cite{S-T-1} 
this concentration result 
was later extended to polynomial maps 
$f \colon \CC^n \longrightarrow \CC$ 
with isolated singularities with respect 
to some fiber-compactifying extensions of $f$. From now on, 
we shall introduce the theorem of 
N\'emethi and Zaharia \cite{N-Z} which 
gives an upper bound for $B_f \subset \CC$ described in terms of 
a Newton polyhedron of $f$. 

\begin{definition} {\rm (Kouchnirenko \cite{Kushnirenko})} \label{dfn:3-1}
We call the convex hull of $\{0\}\cup NP(f)$ in 
$\RR^n$ the Newton polyhedron at infinity of $f$ 
and denote it by $\Gamma_{\infty}(f)$.
\end{definition}

For a face $\gamma \prec \Gamma_{\infty}(f)$ of $\Gamma_{\infty}(f)$ 
we say that $\gamma$ is at infinity if $0 \notin \gamma$. 

\begin{definition} {\rm (Kouchnirenko \cite{Kushnirenko})} \label{dfn:3-3}
We say that $f(x)=\sum_{v\in \ZZ_+^n} a_vx^v$ ($a_v\in \CC$) 
is non-degenerate at infinity if for any 
face at inifinity $\gamma$ of $\Gamma_{\infty}(f)$  
the complex hypersurface $\{x \in (\CC^*)^n\ |\ f^{\gamma}(x)=0\}$ of 
$T=(\CC^*)^n$ 
is smooth and reduced, where we set 
$f^{\gamma}(x)=\sum_{v \in \gamma \cap \ZZ_+^n} a_vx^v$.
\end{definition}

Kouchnirenko \cite{Kushnirenko} proved that if 
$f$ is convenient and non-degenerate 
at infinity then we have $B_f = f( {\rm Sing} f)$. 
Indeed in such a case, by a result of Broughton \cite{Broughton} 
$f$ is tame at infinity. However, 
in the non-convenient case, 
N\'emethi and Zaharia \cite{N-Z} showed 
that more bifurcation values may occur 
by the presence of the so-called 
``bad faces'' of $\Gamma_{\infty}(f)$. Here 
we introduce their result in the new formulation of 
\cite{Tak-1}.  

\begin{definition} {\rm (Takeuchi-Tib\u ar \cite{T-T})} 
We say that a face $\gamma \prec \Gamma_{\infty}(f)$ 
is atypical if $0 \in \gamma$, 
$\dim \gamma \geq 1$ and 
the cone $\sigma ( \gamma ) \subset (\RR^n)^*$ 
which corresponds it in the dual 
subdivision of $\Gamma_{\infty}(f)$ 
is not contained in the 
first quadrant $(\RR^n)^*_+  \simeq \RR^n_+$ 
of $(\RR^n)^*$. 
\end{definition}
This definition is related to that of 
the bad faces of $NP(f-f(0))$ 
in N{\'e}methi-Zaharia \cite{N-Z} as follows. 
If a face $\Delta \prec NP(f-f(0))$ of $NP(f-f(0))$ 
is bad in the sense of \cite{N-Z}, then 
the convex hull $\gamma$ of $\{ 0 \} \cup 
\Delta$ in $\RR^n$ is an atypical face of $\Gamma_{\infty}(f)$. 
See also Remark \ref{bad} below. 

\begin{example}\label{EXP} 
Let $n=3$ and consider a non-convenient polynomial 
$f(x,y,z)$ on $\CC^3$ whose Newton polyhedron at 
infinity $\Gamma_{\infty}(f)$ is the convex hull of 
the points $(2,0,0), (2,2,0), (2,2,3) \in \RR^3_+$ 
and the origin $0=(0,0,0) \in \RR^3$. Then the line 
segment connecting the point $(2,2,0)$ 
(resp. $(2,0,0)$) 
and the origin $0 \in \RR^3$ is an atypical 
face of $\Gamma_{\infty}(f)$.  However the 
triangle whose vertices are the points 
$(2,0,0)$, $(2,2,0)$ 
and the origin $0 \in \RR^3$ is not so. 
Note that for the line segment $\gamma$ 
connecting $(2,0,0)$ and the origin 
we have $\dim \sigma ( \gamma ) 
\cap ( \RR^3)^*_+ =2$. 
\end{example}

Let $\gamma_1, \ldots, \gamma_m$ 
be the atypical faces of 
$\Gamma_{\infty}(f)$. 
For $1 \leq i \leq m$ let 
$K_i= f^{\gamma_i}( 
{\rm Sing} f^{\gamma_i}) \subset \CC$ 
be the set of the critical values of the $\gamma_i$-part 
\begin{equation}
f^{\gamma_i}: T= (\CC^*)^n \longrightarrow \CC
\end{equation}
of $f$. Let us set 
\begin{equation}
K_f = f( {\rm Sing} f) \cup \{ f(0) \} \cup 
( \cup_{i=1}^m K_i). 
\end{equation}
Then N{\'e}methi-Zaharia \cite{N-Z} 
proved the following very useful result. 

\begin{theorem}\label{BESS}
{\rm (N{\'e}methi-Zaharia \cite[Theorem 2]{N-Z})}  
Assume that $f$ is non-degenerate 
at infinity. Then we have 
$B_f \subset K_f$. 
\end{theorem}

\begin{remark}\label{bad} 
If for an atypical face $\gamma_i$ 
of $\Gamma_{\infty}(f)$ the face 
$\Delta = \gamma_i \cap NP(f-f(0)) 
\prec NP(f-f(0))$ of $NP(f-f(0))$ 
is not bad in the sense 
of N{\'e}methi-Zaharia \cite{N-Z}, 
then $\dim NP(f_{\gamma_i} -f(0))= \dim \Delta 
< \dim \gamma_i$, $f_{\gamma_i} - f(0)$ is 
a positively homogeneous Laurent polynomial 
on $T= (\CC^*)^n$ and hence 
we have $K_i = \{ f(0) \}$. 
Therefore the above inclusion $B_f \subset K_f$ 
coincides with the one in \cite{N-Z}. 
\end{remark} 

In \cite{N-Z} N{\'e}methi and Zaharia 
also proved the equality $B_f=K_f$ for $n=2$ 
and conjectured its validity 
in higher dimensions. So the remaining important 
problem is to prove the inverse inclusion 
$K_f \subset B_f$. Later Zaharia \cite{Zaharia} 
proved $K_f \setminus \{ f(0) \} \subset B_f$ 
for $n \geq 2$ 
assuming some conditions and 
that $f$ has isolated singularities 
at infinity on a fixed smooth toric compactification 
of $\CC^n$. More recently  
in \cite{Tak-1}, the author relaxed 
the conditions in \cite{Zaharia} with the 
help of some basic properties of 
perverse sheaves. See also \cite{Tanabe-G} 
for some recent progress in the problem. 
Also for more general polynomial maps 
$f=(f_1,f_2, \ldots, f_k)  
\colon \CC^n \longrightarrow \CC^k$ 
($1 \leq k \leq n$), one can define 
their bifurcation sets $B_f \subset \CC^k$ 
similarly. In \cite{C-D-T-T} and \cite{Nguyen}, 
as in Theorem \ref{BESS}  
the authors of them gave upper bounds 
for $B_f \subset \CC^k$ described in terms of 
Newton polyhedra of $f_j$ 
($1 \leq j \leq k$). For two polynomials 
$P(x), Q(x) \in \CC [x]= \CC [x_1,x_2, \ldots, x_n]$ on 
$\CC^n$ such that $Q(x) \not= 0$ we define 
a rational function $f(x)$ on $\CC^n$ by 
\begin{equation}
f(x)= \frac{P(x)}{Q(x)}  \qquad (x \in \CC^n \setminus Q^{-1}(0)) 
\end{equation}
and consider the holomorphic map $f: \CC^n \setminus Q^{-1}(0) 
\longrightarrow \CC$ associated to it. Also 
in this case, we can define 
the bifurcation set $B_f \subset \CC$ of $f$. 
Then it would be also an important problem to determine 
$B_f$ for such rational functions $f$. 
In \cite{Thang} and \cite{N-S-T} only some 
partial answers were given to it.

 Let $C_R=\{x\in \CC\ |\ |x|=R\}$ ($R\gg 0$) be a sufficiently 
large circle in $\CC$ such that $B_f\subset \{x \in \CC\ |\ |x|<R\}$. 
Then by restricting the locally trivial fibration $\CC^n \setminus f^{-1}(B_f) 
\longrightarrow \CC \setminus B_f$ to $C_R$ we obtain a 
geometric monodromy automorphism $\Phi_f^{\infty} \colon 
f^{-1}(R) \simto f^{-1}(R)$ and the linear maps
\begin{equation}
\Phi_j^{\infty} \colon H^j(f^{-1}(R) ;\CC) 
\overset{\sim}{\longrightarrow} H^j(f^{-1}(R) ;\CC) \ \ (j=0,1,\ldots)
\end{equation}
associated to it, where the orientation of $C_R$ 
is taken to be counter-clockwise as usual. We call $\Phi_j^{\infty}$'s 
the (cohomological) monodromies at infinity of $f$. Various formulas for 
their eigenvalues (i.e. the semisimple parts) 
were obtained by Libgober-Sperber \cite{L-S} 
and Gusein-Zade-Luengo-Melle-Hern\'andez 
\cite{GZ-L-MH-new}. 
Also, several results on their nilpotent parts were 
obtained by Garc{\'i}a-L{\'o}pez-N{\'e}methi \cite{L-N-2}, Dimca-Saito \cite{D-S-1} 
and Matsui-Takeuchi \cite{M-T-2}. 
For example, Dimca-Saito \cite{D-S-1} and 
Matsui-Takeuchi \cite{M-T-2} obtained upper 
bounds for the sizes of Jordan blocks in $\Phi_j^{\infty}$. 
For the special case $n=2$, see also \cite{Dimca2}. 
In Sections \ref{section 7} and \ref{section 9},  
we will introduce the combinatorial expressions of 
the Jordan normal form of $\Phi_{n-1}^{\infty}$ 
obtained in Matsui-Takeuchi 
\cite{M-T-3} and Stapledon \cite{Stapledon}. 
The monodromies at infinity $\Phi_j^{\infty}$ are important, because 
after a basic result \cite{N-N} of Neumann-Norbury, Dimca-N{\'e}methi 
\cite{D-N} proved that for any $j \in \ZZ$ 
the monodromy representation 
\begin{equation}
\pi_1(\CC \setminus B_f, c) \longrightarrow {\rm Aut}
 ( H^j(f^{-1}(c); \CC) ) \qquad (c \in \CC \setminus B_f)
\end{equation}
is described by the matrix of 
$\Phi_j^{\infty}$ for a basis in the decomposition of 
$H^j(f^{-1}(c); \CC)$ with respect to the vanishing 
cycles at the points in $B_f$. 
Recall that by Theorem \ref{tame} if $f$ is tame at infinity 
 $\Phi_{n-1}^{\infty}$ is 
the only non-trivial monodromy at infinity and its characteristic 
polynomial is calculated by the following zeta function 
$\zeta_f^{\infty}(t) \in \CC(t)^*$. 

\begin{definition}\label{dfn:3-4}
We define the monodromy zeta function at infinity $\zeta_f^{\infty}(t)$ of $f$ by
\begin{equation}
\zeta_{f}^{\infty}(t):=\prod_{j=0}^{\infty} 
\det(\id -t\Phi_j^{\infty})^{(-1)^j} \quad \in \CC(t)^*.
\end{equation}
\end{definition}

For a subset $\J\subset \{1,2,\ldots,n\}$, let 
$\RR^{\J} \simeq \RR^{\sharp S}$ be as in Section \ref{section 3} 
and set $\Gamma_{\infty}^{\J}(f):=\Gamma_{\infty}(f) \cap \RR^{\J}$. 
Then $f$ is convenient if and only if 
we have $\dim \Gamma_{\infty}^{\J}(f)=\sharp {\J}$ for any 
${\J} \subset \{1,2,\ldots,n\}$. For each non-empty 
subset ${\J} \subset \{1,2,\ldots, n\}$, 
let $\gamma_1^{\J},\gamma_2^{\J}, \ldots,
\gamma_{n({\J})}^{\J}$ be the $(\sharp {\J}-1)$-dimensional 
faces at infinity of $\Gamma_{\infty}^{\J}(f)$. 
Note that for $\J\subset \{1,2,\ldots,n\}$ such 
that ${\rm dim} \Gamma_{\infty}^{\J}(f) < \sharp S$ 
there is not such a face and hence we set $n(S)=0$. 
For $1 \leq i \leq n(\J)$, let 
$u_i^{\J} \in (\RR^{\J})^* \cap \ZZ^{\J}$ be the 
unique non-zero primitive vector which takes its minimum in 
$\Gamma_{\infty}^{\J}(f)$ on the whole $\gamma_i^{\J}$ and set
\begin{equation}
d_i^{\J}: = - \min_{v\in \Gamma_{\infty}^{\J}(f)} 
\langle u_i^{\J} ,v \rangle \in \ZZ_{>0}.
\end{equation}
We call $d_i^{\J}$ the lattice distance of $\gamma_i^{\J}$ from 
the origin $0 \in \RR^{\J}$. For each face $\gamma_i^{\J} 
\prec \Gamma_{\infty}^{\J}(f)$, let $\LL(\gamma_i^{\J})$ be 
the smallest affine linear subspace of $\RR^n$ 
containing $\gamma_i^{\J}$ and $\Vol_{\ZZ}(\gamma_i^{\J}) 
\in \ZZ_{>0}$ the normalized $(\sharp \J -1)$-dimensional 
volume of $\gamma_i^{\J}$ with 
respect to the lattice $\ZZ^n \cap \LL(\gamma_i^{\J})$.

\begin{theorem}\label{thm:3-5}{\rm (Libgober-Sperber \cite[Theorem 1]{L-S}, 
see also \cite[Theorem 3.1]{M-T-new3} for a slight generalization)} 
Assume that $f$ is non-degenerate at infinity. Then we have 
\begin{equation}
\zeta_f^{\infty}(t)=\prod_{{\J} \neq \emptyset }\zeta^{\infty}_{f, {\J}}(t),
\end{equation}
where for each non-empty subset $\J \subset 
\{1,2,\ldots,n\}$ we set
\begin{equation}
\zeta^{\infty}_{f,{\J}}(t):=\prod_{i=1}^{n({\J})}
(1-t^{d_i^{\J}})^{(-1)^{\sharp {\J}-1}\Vol_{\ZZ}(\gamma_i^{\J})}.
\end{equation}
\end{theorem}

This theorem was first proved by Libgober-Sperber \cite{L-S}. Here 
we introduce the new proof given later in \cite{M-T-new3}.

\begin{proof}
Let $j \colon \CC \longhookrightarrow \PP^1=\CC\sqcup \{\infty\}$ be the 
compactification of $\CC$ and take a local coordinate 
$h$ of $\PP^1$ on a neighborhood of the point 
$\infty \in \PP^1$ such that $\infty=\{h=0\}$. 
Then by the (generalized) Poincar\'e duality isomorphisms  
\begin{equation}
H_j(f^{-1}(R);\CC) \simeq H_c^{2n-2-j}(f^{-1}(R);\CC) \qquad 
(R \gg 0, j \in \ZZ_+) 
\end{equation}
we find that
\begin{equation}
\zeta_f^{\infty}(t)=\zeta_{h, \infty}(j_! {\rm R} f_!\CC_{\CC^n})(t) 
\qquad \in \CC(t)^*.
\end{equation}
Let $T_S, i_S$ 
etc. be as in the proof of Theorem \ref{thm:3-5-ad}. 
Then by Corollary \ref{new-corol},  
Lemma \ref{lemma:2-8} and  
the decomposition $\CC^n \setminus \{ 0 \} 
 = \sqcup_{\J \not= \emptyset}T_S$ 
of $\CC^n \setminus \{ 0 \}$ we obtain 
\begin{equation}
\zeta_{h, \infty}(j_! {\rm R} f_!\CC_{\CC^n})(t)
= \prod_{S \not= \emptyset} \ 
\zeta_{h, \infty}(j_! {\rm R} (f \circ i_S)_! \CC_{T_S})(t). 
\end{equation} 
Note that we have $ {\rm R} (f \circ i_S)_! \CC_{T_S} 
\simeq  {\rm R} (f|_{T_S})_! \CC_{T_S}$. 
Hence it suffices to 
prove that for any non-empty subset $\J\subset \{1,2,\ldots,n\}$ 
we have 
\begin{equation}
\zeta_{h, \infty}(j_! {\rm R} (f|_{T_S})_! \CC_{T_S})(t)
= \zeta^{\infty}_{f,{\J}}(t)
=\prod_{i=1}^{n({\J})}
\ (1-t^{d_i^{\J}})^{(-1)^{\sharp {\J}-1}\Vol_{\ZZ}(\gamma_i^{\J})}.
\end{equation}
Here we shall prove it only for $\J = \{1,2,\ldots,n\}$. 
The proof for the other cases are similar. 
For $\J = \{1,2,\ldots,n\}$ we set 
\begin{equation}
T:=T_S=T_{\{1,2,\ldots,n\}} =( \CC^*)^n \subset \CC^n. 
\end{equation}
Let $\Sigma$ be a smooth fan in $( \RR^n)^* \simeq \RR^n$ obtained by 
subdividing the dual subdivision of $\Gamma^S_{\infty}(f) =
\Gamma_{\infty}(f)$ in $( \RR^n)^*$. 
Let $X_{\Sigma}$ be the smooth 
toric variety  associated to $\Sigma$. 
Recall that $T=( \CC^*)^n$ 
acts on $X_{\Sigma}$ and the $T$-orbits are parametrized by 
the cones in $\Sigma$. For a cone $\sigma \in \Sigma$ denote by 
$T_{\sigma} \simeq (\CC^*)^{n-\dim \sigma}$ the corresponding $T$-orbit. 
We have also natural affine open subsets $\CC^n(\sigma) 
\simeq \CC^n$ of $X_{\Sigma}$ associated to $n$-dimensional 
cones $\sigma$ in $\Sigma$. 
Let $\sigma$ be an $n$-dimensional cone in $\Sigma$ and 
$w_1,\ldots, w_n \in \ZZ^n$ the primitive vectors 
on the edges of $\sigma$. Then there exists an 
affine open subset $\CC^n(\sigma)$ of $X_{\Sigma}$ such 
that $\CC^n(\sigma) \simeq \CC^n_y$ and 
the meromorphic function $\tl{f}$ on $X_{\Sigma}$ 
obtained by extending $f|_T: T \longrightarrow \CC$ 
to $X_{\Sigma}$ 
has the following form on it:
\begin{equation} 
\tl{f} (y)=\sum_{v \in \ZZ_+^n}a_{v}y_1^{\langle w_1,v \rangle}\cdots 
y_n^{\langle w_n, v \rangle} =y_1^{b_1} \cdots y_n^{b_n} \cdot f_{\sigma}(y), 
\end{equation}
where we set $f=\sum_{v \in \ZZ_+^n}a_{v}x^{v}$,
\begin{equation}
b_i=\min_{v\in \Gamma_{\infty}(f)} \langle w_i,v \rangle \leq 0 \qquad (i=1,2,\ldots,n)
\end{equation}
and $f_{\sigma}(y)$ is a polynomial on $\CC^n(\sigma) \simeq \CC^n_y$. 
Note that on the affine open subset 
$\CC^n( \sigma ) \simeq \CC^n$ of $X_{\Sigma}$ 
the hypersurface 
\begin{equation}
Z:= \overline{(f|_T)^{-1}(0)} \quad \subset X_{\Sigma} 
\end{equation}
is explicitly written 
as $f_{\sigma}^{-1}(0) \subset \CC^n( \sigma )$. 
So far we have extended $f|_T: T \longrightarrow \CC$ 
to the meromorphic function $\tl{f}$ on 
the smooth complete toric variety $X_{\Sigma}$. 
But $\tl{f}$ still has points of indeterminacy. 
This prevents us from using Theorem \ref{prp:2-99} 
to calculate 
$\zeta_{h, \infty}(j_! {\rm R} (f|_T)_! \CC_{T})(t) 
\in \CC (t)^*$. 
From now on, 
we shall eliminate the points of indeterminacy of 
$\tl{f}$ by blowing up $X_{\Sigma}$. 
Following \cite{L-S}, we say that a $T$-orbit $T_{\sigma}$ in $X_{\Sigma}$ 
is at infinity if the supporting face $\gamma(\sigma) \prec 
\Gamma_{\infty}(f)$ of $\sigma$ in $\Gamma_{\infty}(f)$ 
is at infinity i.e. $0 \notin \gamma(\sigma)$. 
Let $\rho_1, \rho_2, \ldots, \rho_m \in \Sigma$ be the 
$1$-dimensional cones 
in $\Sigma$ such that the $T$-orbits $T_i:=T_{\rho_i} 
\simeq ( \CC^*)^{n-1}$ are at infinity. 
Then for any $1 \leq i \leq m$ the toric divisor $D_i:=\overline{T_i}$ 
is a smooth hypersurface in $X_{\Sigma}$ 
and the poles of $\tl{f}$ are contained in 
their union $D_1 \cup \cdots \cup D_m$. 
Moreover by the non-degeneracy at infinity of $f$, the hypersurface 
$Z= \overline{(f|_T)^{-1}(0)}$ in $X_{\Sigma}$ 
intersects $D_I:= \bigcap_{i \in I}D_i$ transversally 
for any non-empty subset $I \subset \{ 1,2, \ldots, m \}$. 
At such intersection points, $\tl{f}$ has indeterminacy. Furthermore 
we denote the (unique non-zero) primitive vector in 
$\rho_i \cap \ZZ^n$ by $u_i$. Then the order 
$a_i>0$ of the pole of $\tl{f}$ along $D_i$ is given by
\begin{equation}
a_i=-\min_{v\in \Gamma_{\infty}(f)} \langle u_i,v \rangle.
\end{equation} 
Now, in order to eliminate the indeterminacy 
of the meromorphic function $\tl{f}$ on $X_{\Sigma}$, 
we first consider the blow-up $\pi_1 \colon X_{\Sigma}^{(1)} 
\longrightarrow X_{\Sigma}$ of $X_{\Sigma}$ along the $(n-2)$-dimensional 
smooth subvariety $D_1\cap Z$. Then the indeterminacy of the pull-back 
$\tl{f}\circ \pi_1$ of $\tl{f}$ to $X_{\Sigma}^{(1)}$ is improved. If 
$\tl{f}\circ \pi_1$ still has points of indeterminacy on the intersection 
of the exceptional divisor $E_1$ of $\pi_1$ and the proper transform 
$Z^{(1)}$ of $Z$, we construct the blow-up $\pi_2 \colon X_{\Sigma}^{(2)} 
\longrightarrow X_{\Sigma}^{(1)}$ of $X_{\Sigma}^{(1)}$ along $E_1 \cap Z^{(1)}$. 
By repeating this procedure $a_1$ times, we obtain a tower of blow-ups
\begin{equation}
X_{\Sigma}^{(a_1)} \underset{\pi_{a_1}}{\longrightarrow} 
\cdots \cdots
\underset{\pi_2}{\longrightarrow} X_{\Sigma}^{(1)} 
\underset{\pi_1}{\longrightarrow} X_{\Sigma}.
\end{equation}
Then the pull-back of $\tl{f}$ to $X_{\Sigma}^{(a_1)}$ has no indeterminacy over $T_1$. 
See \cite[Figures 1 and 2]{M-T-new3} 
and \cite[Figures 1,2 and 3]{M-T-3} for the details. 
Next we apply this construction to the proper transforms of $D_2$ and $Z$ 
in $X_{\Sigma}^{(a_1)}$. Then we obtain also a tower of blow-ups
\begin{equation}
X_{\Sigma}^{(a_1)(a_2)} \longrightarrow \cdots \cdots 
\longrightarrow X_{\Sigma}^{(a_1)(1)} \longrightarrow X_{\Sigma}^{(a_1)}
\end{equation}
and the indeterminacy of the pull-back of $\tl{f}$ to $X_{\Sigma}^{(a_1)(a_2)}$ 
is eliminated over $T_1 \sqcup T_2$. By applying the same 
construction to (the proper transforms of) $D_3, D_4,\ldots, D_m$, 
we finally obtain a birational morphism 
$\pi \colon \tl{X_{\Sigma}} \longrightarrow X_{\Sigma}$ 
such that $g:=\tl{f} \circ \pi$ has no point of 
indeterminacy on the whole $\tl{X_{\Sigma}}$. 
Note that the smooth compactification $\tl{X_{\Sigma}}$ of $T=( \CC^*)^n$ 
thus obtained is not a toric variety any more. 
Nevertheless, the union of the exceptional divisors of 
$\pi \colon \tl{X_{\Sigma}} \longrightarrow X_{\Sigma}$ 
and the proper transforms of $D_1, \ldots, D_m$ in $\tl{X_{\Sigma}}$ 
is normal crossing. 
We thus obtain a commutative diagram of holomorphic maps
\begin{equation}
\xymatrix{
T \ar@{^{(}->}[r]^{\iota} \ar[d]_{f|_T} & \tl{X_{\Sigma}} \ar[d]^g\\
\CC \ar@{^{(}->}[r]^j & \PP^1,}
\end{equation}
where $\iota$ and $j$ are the inclusion maps and 
$g$ is proper. 
Then there exists an isomorphism 
\begin{equation}
j_! {\rm R} (f|_T)_! \CC_T \simeq 
{\rm R} g_* ( \iota_! \CC_T)
\end{equation}
in $\Dbc(\PP^1)$. 
Now we apply Theorem \ref{prp:2-99}  
to the proper morphism $g \colon \tl{X_{\Sigma}} \longrightarrow \PP^1$, 
the function $h$ 
and the constructible sheaf $\CC_T
= \iota_! \CC_T \in \Dbc ( \tl{X_{\Sigma}} )$. 
Then for $S= \{1,2,\ldots,n\}$ we obtain an equality 
\begin{equation}
\zeta_{h, \infty}( {\rm R} g_* \CC_{T} )= 
\zeta_{h}( {\rm R} g_* \CC_{T} )( \infty )
= \int_{g^{-1}( \infty )}  \zeta_{h \circ g}( \CC_{T} ) 
\end{equation}
in $\CC (t)^*$. Note that $h\circ g$ is 
the meromorphic extension of 
$1/(f|_T): T \setminus f^{-1}(0) \longrightarrow \CC$ 
to $\tl{X_{\Sigma}}$ and we have 
$(h\circ g)^{-1}(0) =g^{-1}(\infty ) \subset \tl{X_{\Sigma}}$. 
Then as in the proof of Theorem \ref{thm:3-5-ad}, 
by calculating the constructible function 
$\zeta_{h \circ g}( \CC_{T} ) \in {\rm CF}_{\CC (t)^*}
(g^{-1}(\infty ))$ at each point of $g^{-1}(\infty )$ 
with the help of 
Lemmas \ref{lem:2-ac-1} and \ref{lem:2-ac-21} and 
Theorem \ref{BKK-simple}, for $S= \{ 1,2, \ldots, n \}$ 
we obtain the desired equality 
\begin{equation}
\int_{g^{-1}( \infty )}  \zeta_{h \circ g}( \CC_{T} )  =
\prod_{i=1}^{n({\J})} \ (1-t^{d_i^{\J}})^{(-1)^{\sharp {\J}-1}
\Vol_{\ZZ}(\gamma_i^{\J})}.
\end{equation}
Note that for $S= \{ 1,2, \ldots, n \}$ 
if ${\rm dim} \Gamma_{\infty}(f) < \sharp S=n$ 
then by Theorem \ref{BKK-simple} we have 
$\int_{g^{-1}( \infty )}  \zeta_{h \circ g}( \CC_{T} )=1$ 
and $n(S)=0$. 
This completes the proof.
\qed
\end{proof}

By this new proof of Theorem \ref{thm:3-5} we can easily obtain 
the following globalization of Oka's theorem. 
Let $f=(f_1,f_2, \ldots, f_k)$, $V \subset W \subset \CC^n$ 
and $g=f_k|_W: W \longrightarrow \CC$ be as in (the last 
half of) Section \ref{section 3}. 
Then there exists a finite subset $B \subset \CC$ such that the restriction
\begin{equation}
W \setminus g^{-1}(B) \longrightarrow \CC \setminus B
\end{equation}
of $g$ is a locally trivial fibration. 
Let $C_R=\{x\in \CC\ |\ |x|=R\}$ ($R\gg 0$) be a sufficiently 
large circle in $\CC$ such that $B \subset \{x \in \CC\ |\ |x|<R\}$ 
and consider the monodromies at infinity 
\begin{equation}
\Phi_j^{\infty} \colon H^j(g^{-1}(R) ;\CC) 
\overset{\sim}{\longrightarrow} H^j(g^{-1}(R) ;\CC) \ \ (j=0,1,\ldots)
\end{equation}
of $g=f_k|_W: W \longrightarrow \CC$ along it. Let 
$\zeta^{\infty}_g(t) \in \CC (t)^*$ be the 
monodromy zeta function associated to them. 
In order to describe it in terms of the Newton 
polyhedra $\Gamma_{\infty}(f_j)$ 
($1 \leq j \leq k$), for simplicity here we assume 
that $f_1, f_2, \ldots, f_k$ are convenient. Then, 
for each non-empty subset ${\J} \subset \{1,2,\ldots, n\}$ 
we set $I(S):= \{ 1,2, \ldots, k-1 \}$ and $m(S):=k$ and 
consider the Minkowski sum 
\begin{equation}
\Gamma_{\infty}^S(f):= \dsum_{j \in I(S) \sqcup \{ k \}} \ 
\Gamma_{\infty}^S(f_j) 
= \dsum_{j =1}^k \ \Gamma_{\infty}^S(f_j) \ 
\subset \RR^S_+. 
\end{equation}
Moreover we define the non-degeneracy at infinity of 
$f=(f_1,f_2, \ldots, f_k)$ by replacing the 
condition $u \in {\rm Int} ( \RR^S)^*_+$ for 
$u \in ( \RR^S)^*$ in Definition \ref{de-no-deg} 
by the one $u \in ( \RR^S)^* \setminus ( \RR^S)^*_+$. 
Then by defining $n(S)$, 
$d_i^S$ and $K_i^S$ as in Section \ref{section 3} 
we obtain the following result. 

\begin{theorem}\label{thm:3-5-CI-new}{\rm 
(Matsui-Takeuchi \cite[Theorem 5.1]{M-T-new2})} 
Assume that $f_1, f_2, \ldots, f_k$ are convenient and 
$f=(f_1,f_2, \ldots, f_k)$  
is non-degenerate at infinity. Then 
the monodromy zeta function $\zeta_{g}^{\infty}(t) \in \CC (t)^*$ 
at infinity 
of $g=f_k|_W: W \longrightarrow \CC$ is given by 
\begin{equation}
\zeta_{g}^{\infty}(t)=
\dprod_{S : \ k \leq \sharp S}
 \ \zeta_{g, S}^{\infty}(t),
\end{equation}
where for each subset $\J \subset \{1,2,\ldots,n\}$ 
such that  
$k \leq \sharp S$ we set
\begin{equation}
\zeta_{g, S}^{\infty}(t):=
\prod_{i=1}^{n({\J})}
\ (1-t^{d_i^{\J}})^{(-1)^{\sharp {\J}-m(S)} K_i^S}.
\end{equation}
\end{theorem}

\section{Combinatorial formulas for equivariant 
Hodge-Deligne numbers of toric 
hypersurfaces}\label{section 5}

In this section, we introduce 
the refinements obtained in \cite[Section 2]{M-T-3}
of the results of Danilov-Khovanskii \cite{D-K}. 
We shall freely use standard notions of mixed Hodge structures, 
for which we refer to El Zein \cite{E-Z} and 
Voisin \cite{Voisin}. 
In the sequel, let us fix an element $\tau =(\tau_1,\ldots, \tau_n) \in 
T:=(\CC^*)^n$ and let $g$ be a Laurent polynomial on $(\CC^*)^n$ 
such that the complex hypersurface $Z^*=\{ x\in (\CC^*)^n \ |\ g(x)=0\}$ is non-degenerate 
and invariant by the automorphism $l_{\tau} \colon (\CC^*)^n 
\underset{\tau \times}{\simto}(\CC^*)^n$ induced by the multiplication by $\tau$. 
Set $\Delta =NP(g)$ and for simplicity assume that $\d \Delta=n$. Then there exists 
$\beta \in \CC$ such that $l_{\tau}^*g= g \circ l_{\tau}=\beta g$. This implies that 
for any vertex $v$ of $\Delta =NP(g)$ we have ${\tau}^v={\tau}_1^{v_1} \cdots 
{\tau}_n^{v_n}=\beta$. Moreover by the condition $\d \Delta=n$ we see that $\tau_1, \tau_2, \ldots , \tau_n$ 
are roots of unity. For $p,q \geq 0$ and $k \geq 0$, let $h^{p,q}(H_c^k(Z^*;\CC))$ be the mixed Hodge number of $H_c^k(Z^*;\CC)$ and set
\begin{equation}
e^{p,q}(Z^*)=\dsum_k (-1)^k h^{p,q}(H_c^k(Z^*;\CC))
\end{equation}
as in \cite{D-K}. The above automorphism of $(\CC^*)^n$ induces a morphism of mixed 
Hodge structures $l_{\tau}^* \colon H_c^k(Z^*;\CC) \simto H_c^k(Z^*;\CC)$ and 
hence $\CC$-linear automorphisms of the $(p,q)$-parts $H_c^k(Z^*;\CC)^{p,q}$ of $H_c^k(Z^*;\CC)$. 
For $\alpha \in \CC$, let $h^{p,q}(H_c^k(Z^*;\CC))_{\alpha}$ be the dimension of the 
$\alpha$-eigenspace $H_c^k(Z^*;\CC)_{\alpha}^{p,q}$ of this automorphism of $H_c^k(Z^*;\CC)^{p,q}$ and set
\begin{equation}
e^{p,q}(Z^*)_{\alpha}=\dsum_k (-1)^k h^{p,q}(H_c^k(Z^*;\CC))_{\alpha}.
\end{equation}
Since we have $l_{\tau}^r =\id_{Z^*}$ for some $r \gg 0$, these numbers are zero 
unless $\alpha$ is a root of unity. Obviously we have
\begin{equation}
e^{p,q}(Z^*)=\dsum_{\alpha \in \CC} e^{p,q}(Z^*)_{\alpha}, \qquad e^{p,q}(Z^*)_{\alpha}=e^{q,p}(Z^*)_{\overline{\alpha}}.
\end{equation}
In this setting, along the lines of Danilov-Khovanskii \cite{D-K} we can give an 
algorithm for computing these numbers $e^{p,q}(Z^*)_{\alpha}$ as follows. First of all, 
as in \cite[Section 3]{D-K} we can easily obtain the following result.

\begin{proposition}\label{prp:2-15}
For $p,q \geq 0$ such that $p+q >n-1$, we have
\begin{equation}
e^{p,q}(Z^*)_{\alpha}
=\begin{cases}
(-1)^{n+p+1}\binom{n}{p+1} & (\text{$\alpha=1$ and $p=q$}),\\
 & \\
\ 0 & (\text{otherwise}), 
\end{cases}
\end{equation}
where we used the convention $\binom{a}{b}=0$ ($0 \leq a <b$) for binomial coefficients.
\end{proposition}

For a vertex $w$ of $\Delta$, consider the translated polytope $\Delta^w:=\Delta -w$ 
such that $0 \prec \Delta^w$ and ${\tau}^v=1$ for any vertex $v$ of $\Delta^w$. Then for $\alpha \in \CC$ and $k \geq 0$ set
\begin{equation}
l^*(k\Delta)_{\alpha}=\sharp \{ v \in \Int (k\Delta^w) \cap \ZZ^n \ |\ {\tau}^v =\alpha\} \quad 
\in \ZZ_+
\end{equation}
and
\begin{equation}
l(k\Delta)_{\alpha}=\sharp \{ v \in (k\Delta^w) \cap \ZZ^n \ |\ {\tau}^v =\alpha\} \quad 
\in \ZZ_+.
\end{equation}
We can easily see that these numbers $l^*(k\Delta)_{\alpha}$ and $l(k\Delta)_{\alpha}$ 
do not depend on the choice of the vertex $w$ of $\Delta$. Next, define two formal power 
series $P_{\alpha}(\Delta;t)=\sum_{i \geq 0}\varphi_{\alpha, i}(\Delta)t^i$ and 
$Q_{\alpha}(\Delta;t)=\sum_{i \geq 0}\psi_{\alpha,i}(\Delta)t^i$ by
\begin{equation}
P_{\alpha}(\Delta;t)=(1-t)^{n+1} \left\{ \dsum_{k \geq 0} l^*(k\Delta)_{\alpha}t^k\right\}
\end{equation}
and
\begin{equation}
Q_{\alpha}(\Delta;t)=(1-t)^{n+1} \left\{ \dsum_{k \geq 0} l(k\Delta)_{\alpha}t^k\right\}
\end{equation}
respectively. Then we can easily show that $P_{\alpha}(\Delta;t)$ is actually a polynomial 
as in \cite[Section 4.4]{D-K}. Moreover as in Macdonald \cite{Macdonald}, we can easily prove 
that for any $\alpha \in \CC^*$ the function $h_{\Delta,\alpha}(k):=l(k\Delta)_{\alpha^{-1}}$ of $k \geq 0$ 
is a polynomial of degree $n$ with coefficients in $\QQ$. By a straightforward generalization of the Ehrhart 
reciprocity proved by \cite{Macdonald}, we obtain also an equality
\begin{equation}
h_{\Delta,\alpha}(-k)=(-1)^n l^*(k\Delta)_{\alpha}
\end{equation}
for $k> 0$. By an elementary computation (see \cite[Remark 4.6]{D-K}), this implies that we have
\begin{equation}\label{E:sym}
\varphi_{\alpha , i}(\Delta)= \psi_{\alpha^{-1}, n+1-i}(\Delta ) \qquad (i\in\ZZ ).
\end{equation}
In particular, $Q_{\alpha}(\Delta;t)=\sum_{i \geq 0}\psi_{\alpha, i}(\Delta)t^i$ is a polynomial for any $\alpha \in \CC^*$.

\begin{theorem}\label{thm:2-14}
In the situation as above, we have
\begin{equation}
\dsum_q e^{p,q}(Z^*)_{\alpha}
=\begin{cases}
(-1)^{p+n+1}\binom{n}{p+1} +(-1)^{n+1} \varphi_{\alpha, n-p}(\Delta) & 
(\alpha=1), \\
 & \\
(-1)^{n+1} \varphi_{\alpha, n-p}(\Delta) & (\alpha \not= 1). 
\end{cases}
\end{equation}
\end{theorem}

This result can be proved as in the proof for the formula in \cite[Section 4.4]{D-K}. 
With Proposition \ref{prp:2-15} and Theorem \ref{thm:2-14} at hands, we can now easily calculate the 
numbers $e^{p,q}(Z^*)_{\alpha}$ on the non-degenerate hypersurface $Z^* \subset (\CC^*)^n$ for any $\alpha \in \CC$ as 
in \cite[Section 5.2]{D-K}. Indeed for a projective toric compactification $X$ of $(\CC^*)^n$ such that the 
closure $\overline{Z^*}$ of $Z^*$ in $X$ is smooth, the variety $\overline{Z^*}$ is smooth projective and hence there exists a perfect pairing
\begin{equation}
H^{p,q}(\overline{Z^*};\CC)_{\alpha} \times H^{n-1-p, n-1-q}(\overline{Z^*};\CC)_{\alpha^{-1}} \longrightarrow \CC
\end{equation}
for any $p,q \geq 0$ and $\alpha \in \CC^*$ (see for example \cite[Section 5.3.2]{Voisin}). Therefore, we obtain equalities 
$e^{p,q}(\overline{Z^*})_{\alpha}=e^{n-1-p,n-1-q}(\overline{Z^*})_{\alpha^{-1}}$ which are necessary to 
proceed the algorithm in \cite[Section 5.2]{D-K}. We have also the following analogue of \cite[Proposition 5.8]{D-K}.

\begin{proposition}\label{prp:new}
For any $\alpha \in \CC$ and $p> 0$ we have
\begin{equation}
e^{p,0}(Z^*)_{\alpha}=e^{0,p}(Z^*)_{\overline{\alpha}}= (-1)^{n-1} \sum_{\begin{subarray}{c} \Gamma 
\prec \Delta\\ \d \Gamma =p+1\end{subarray}}l^*(\Gamma)_{\alpha}.
\end{equation}
\end{proposition}

The following result is an analogue of \cite[Corollary 5.10]{D-K}. For $\alpha \in \CC$, denote by 
$\Pi(\Delta)_{\alpha}$ the number of the lattice points $v=(v_1,\ldots, v_n)$ on the $1$-skeleton of 
$\Delta^w=\Delta-w$ such that ${\tau}^v=\alpha$, where $w$ is a vertex of $\Delta$.

\begin{proposition}\label{prp:2-19}
In the situation as above, for any $\alpha \in \CC^*$ we have
\begin{equation}
e^{0,0}(Z^*)_{\alpha}
=\begin{cases}
(-1)^{n-1} \left(\Pi(\Delta)_{1}-1\right)  & (\alpha=1), \\
 & \\
(-1)^{n-1}  \Pi(\Delta)_{\alpha^{-1}} & (\alpha \not= 1). 
\end{cases}
\end{equation}
\end{proposition}

\begin{proof}
By Theorem \ref{thm:2-14}, Proposition \ref{prp:new} and the equality \eqref{E:sym}, 
the assertion can be proved as in the proof \cite[Corollary 5.10]{D-K}. \qed
\end{proof}

For a vertex $w$ of $\Delta$, we define a closed convex cone $\Con(\Delta, w)$ by 
$\Con(\Delta,w)=\{ r \cdot (v -w) \ |\ r \in \RR_+, \ v \in \Delta\} \subset \RR^n$.

\begin{definition}\label{dfn:2-16}
Let $\Delta$ be an $n$-dimensional integral polytope in $(\RR^n, \ZZ^n)$.
\begin{enumerate}
\item (see \cite[Section 2.3]{D-K}) We say that $\Delta$ 
is prime if for any vertex $w$ of $\Delta$ the cone $\Con(\Delta,w)$ is generated by a basis of $\RR^n$.
\item We say that $\Delta$ is pseudo-prime if for any 
$1$-dimensional face $\gamma \prec \Delta$ the number of the 
$2$-dimensional faces $\gamma^{\prime} \prec \Delta$ such that $\gamma \prec \gamma^{\prime}$ is $n-1$.
\end{enumerate}
\end{definition}

By definition, prime polytopes are pseudo-prime. Moreover any face of a pseudo-prime polytope is again pseudo-prime. 
From now on, we assume that $\Delta=NP(g)$ is pseudo-prime. Let $\Sigma$ be the dual fan of $\Delta$ and $X_{\Sigma}$ 
the toric variety associated to it. Then except finite points $X_{\Sigma}$ is an orbifold and the 
closure $\overline{Z^*}$ of $Z^*$ in $X_{\Sigma}$ does not intersect such 
points by the non-degeneracy of $g$. Hence $\overline{Z^*}$ is an orbifold i.e. quasi-smooth 
in the sense of \cite[Proposition 2.4]{D-K}. In particular, there exists a Poincar{\'e} duality isomorphism
\begin{equation}
[H^{p,q}(\overline{Z^*};\CC)_{\alpha}]^* \simeq H^{n-1-p,n-1-q}(\overline{Z^*};\CC)_{\alpha^{-1}}
\end{equation}
for any $\alpha \in \CC^*$ (see for example \cite{Danilov-2} and \cite[Corollary 8.2.22]{H-T-T}). 
Then by slightly generalizing the arguments in \cite{D-K} we 
obtain the following analogue of \cite[Section 5.5 and Theorem 5.6]{D-K}.

\begin{proposition}\label{prp:2-17}
In the situation as above, for any $\alpha \in \CC \setminus \{1\}$ and $p,q \geq 0$, we have
\begin{equation}
e^{p,q}(\overline{Z^*})_{\alpha}
=\begin{cases}
 -\dsum_{\Gamma \prec \Delta}(-1)^{\d \Gamma} 
\varphi_{\alpha, \d \Gamma -p}(\Gamma) & (p+q=n-1),\\
 & \\
\ 0 & (\text{otherwise}), 
\end{cases}
\end{equation}
and 
\begin{equation}
e^{p,q}(Z^*)_{\alpha}
= (-1)^{n+p+q} \sum_{\begin{subarray}{c} \Gamma \prec \Delta\\ \d \Gamma =p+q+1\end{subarray}} 
\left\{ \sum_{\Gamma^{\prime} \prec \Gamma} (-1)^{\d 
\Gamma^{\prime}} \varphi_{\alpha, \d \Gamma^{\prime}-p}(\Gamma^{\prime})\right\}.
\end{equation}
\end{proposition}

For $\alpha \in \CC \setminus \{1\}$ and a face $\Gamma \prec \Delta$, set $\tl{\varphi}_{\alpha}(\Gamma)=
\sum_{i=0}^{\d \Gamma} \varphi_{\alpha, i}(\Gamma)$. Then Proposition \ref{prp:2-17} can be rewritten as follows.

\begin{corollary}\label{cor:2-18}
For any $\alpha \in \CC \setminus \{1\}$ and $r \geq 0$, we have
\begin{equation}
\sum_{p+q=r}e^{p,q}(Z^*)_{\alpha}=(-1)^{n+r} \sum_{\begin{subarray}{c} \Gamma \prec \Delta\\ \d 
\Gamma =r+1\end{subarray}} \left\{ \sum_{\Gamma^{\prime} \prec \Gamma} 
(-1)^{\d \Gamma^{\prime}}\tl{\varphi}_{\alpha}(\Gamma^{\prime})\right\}.
\end{equation}
\end{corollary}

\section{Motivic Milnor fibers and their Hodge realizations}\label{section 6}

In this section, we introduce some important results in 
Denef-Loeser \cite{D-L-1} which will be used in 
subsequent sections. Since their original proofs in \cite{D-L-1} 
are written in terms of Chow motives and it is not so 
easy to compare them with the assertions in \cite{D-L-2}, 
here we rewrite them in the new style of \cite{D-L-2}. 
We shall freely use standard notions of mixed Hodge modules, 
for which we refer to  \cite[Section 8.3]{H-T-T}, \cite{Saito-1}, 
\cite{Saito-2} and \cite{Schnell}. 
First, let $\mathrm{Var}_{\CC}$ be the category of 
algebraic varieties defined over $\CC$. We denote 
by $K_{0}(\mathrm{Var}_{\CC})$ the abelian group 
generated by the isomorphism classes $[X]$ of the elements 
$X\in \mathrm{Var}_{\CC}$ of $\mathrm{Var}_{\CC}$ and defined by the 
relation $[X]=[X\setminus Y]+[Y]$ ($Y$ is a 
Zariski closed subset of $X$). By the products 
of varieties $K_{0}(\mathrm{Var}_{\CC})$ is a 
ring. We call it the motivic Grothendieck ring of 
varieties. For a positive integer $d \in \ZZ_{>0}$, let $\mu_d 
:= \{t\in\CC\ | \ t^{d}=1\} 
\simeq \ZZ/\ZZ d$ be the multiplicative 
group consisting of the $d$-roots in $\CC$ 
and $\zeta_{d}:=\exp(2\pi \sqrt{-1}/d)\in\mu_{d}$ 
its standard generator. Then they form a projective system 
by the homomorphisms 
\begin{equation}
\mu_{rd} \longrightarrow  \mu_{d}\qquad (t\longmapsto t^{r})
\end{equation}
and we obtain the projective limit
\begin{equation}
\hat{\mu}:= \underset{d}{\varprojlim} \ \mu_d
\end{equation}
associated to it. 
Following Denef-Loeser \cite[Section 2.4]{D-L-2}, we say 
that an action of the group $\mu_d$ on an 
algebraic variety $X$ is good if any oribit of it is contained 
in an affine open subset. Moreover we say that 
an action of the group $\hat{\mu}$ on $X$ is good if it is induced by 
a good one of $\mu_d$ for some $d \in \ZZ_{>0}$. 
We thus obtain the Grothendieck ring 
$K^{\hat{\mu}}_{0}(\mathrm{Var}_{\CC})$ of 
varieties over $\CC$ with good $\hat{\mu}$-action. 
In fact, it is noting but the inductive limit 
defined by the ring homomorphisms 
$K_{0}^{\mu_{d}}(\mathrm{Var}_{\CC})
\longrightarrow K_{0}^{\mu_{rd}}(\mathrm{Var}_{\CC})$. 
We call the element 
$\mathbb{L}:=[\CC] \in 
K^{\hat{\mu}}_{0}(\mathrm{Var}_{\CC})$ of 
$K^{\hat{\mu}}_{0}(\mathrm{Var}_{\CC})$ defined by 
the affine line $\CC$ endowed with the trivial 
$\hat{\mu}$-action the Lefschetz motive. Then 
we denote by $\M_{\CC}^{\hat{\mu}}$ 
the ring obtained from the Grothendieck ring $\KK_0^{\hat{\mu}}(\Var_{\CC})$ 
by inverting the Lefschetz motive 
$\LL \in \KK_0^{\hat{\mu}}(\Var_{\CC})$.

In what folows, let $Z$ be a (not necessarily 
complete) smooth algebraic variety over 
$\CC$ and $g\colon Z \longrightarrow \CC$ 
a regular function on it such that 
$E:=g^{-1}(0)\subset Z$ is a normal 
crossing divisor. Assume also that 
the irreducible components $E_{1},E_{2},\dots, E_{k}$ 
of $E$ are smooth. We denote the multiplicity 
of $g$ along $E_i$ by $m_i>0$. For each subset 
$I\subset \{1,2,\dots,k\}$ we set 
\begin{equation}
E_{I}=\bigcap_{i\in I}E_{i},
\quad E^{\circ}_{i}=E_{I}\setminus \bigcup_{i\notin I}E_{i}
\end{equation}
and $m_{I}={\rm gcd}(m_{i})_{i\in I}>0$. 
Then we can construct an unramified Galois covering 
$\tl{E_{I}^{\circ}}\longrightarrow E_{I}^{\circ}$ of 
$E^{\circ}_{I}$ as follows. 
First, for a point $p \in E_I^{\circ}$ we take an affine open neighborhood 
$W \subset Z \setminus ( \cup_{i \notin I} E_i)$ of $p$ on which there 
exist regular functions $\xi_i$ ($i\in I$) 
such that $E_i \cap W=\{ \xi_i=0 \}$ for any $i \in I$ and set 
\begin{equation}
g_{1,W}:=g \prod_{i \in I}\xi_i^{-m_i}, \qquad 
g_{2,W}:=\prod_{i \in I} \xi_i^{\frac{m_i}{m_I}}. 
\end{equation}
Note that $g_{1,W}$ is a unit i.e. invertible 
on $W$ and $g_{2,W} \colon W \longrightarrow \CC$ is a regular function. 
Then on $W$ we obtain a decomposition 
\begin{equation}
g=g_{1,W} \cdot (g_{2,W})^{m_{I}}
\end{equation}
of $g$. 
It is easy to see that $E_I^{\circ}$ is covered by 
such affine open subsets $W$ of $Z \setminus ( \cup_{i \notin I} E_i)$. 
Then as in \cite[Section 3.3]{D-L-2} by gluing the varieties
\begin{equation}\label{eq:6-26}
\tl{E_{I,W}^{\circ}}=\{(t,z) \in \CC^* \times 
(E_I^{\circ} \cap W) \ |\ t^{m_I} =(g_{1,W})^{-1}(z)\}
\end{equation}
together in the following way, we obtain the covering 
$\tl{E_I^{\circ}}$ over $E_I^{\circ}$ of degree $m_{I}$. 
If $W^{\prime}$ is another such open subset and 
$g=g_{1,W^{\prime}} \cdot (g_{2,W^{\prime}})^{m_I}$ is the 
decomposition of $g$ on it, we patch $\tl{E_{I,W}^{\circ}}$ and $\tl{E_{I,W^{\prime}}^{\circ}}$ 
by the morphism $(t,z) \longmapsto 
(g_{2,W^{\prime}}(z)( g_{2,W})^{-1}(z) \cdot t, z)$ defined 
over $W \cap W^{\prime}$.
Let $\mu_{m_{I}}=\{t\in\CC\ |\ t^{m_{I}}=1\}\simeq \ZZ/\ZZ m_{I}$ 
be the cyclic group of order $m_I$. Then by 
assigning the automorphism 
$(t,z)\mapsto (\zeta_{m_{I}}\cdot t, z)$ of $\tl{E_{I}^{\circ}}$
to its standard generator 
$\zeta_{m_{I}}=\exp(2\pi\sqrt{-1}/m_{I})\in \mu_{m_{I}}$ 
we obtain a good action of $\mu_{m_{I}}$ on $\tl{E_{I}^{\circ}}$. 
We denote the element of 
the motivic Grothendieck ring $\M^{\hat{\mu}}_{\CC}$ 
with good action of $\hat{\mu}$ thus obtained 
by $[\tl{E_{I}^{\circ}}]$. Moreover we set 
\begin{equation}
\MCS_{g}=\sum_{I\not= \emptyset}(1-
\mathbb{L})^{|I|-1}\cdot {[\tl{E_{I}^{\circ}}]} 
\quad \in\M^{\hat{\mu}}_{\CC} 
\end{equation}
and call it the motivic Milnor fiber of $g$. 
From now, let us explain the Hodge realizations 
of the elements in $\M^{\hat{\mu}}_{\CC}$ 
defined in Denef-Loeser \cite[Section 3.1]{D-L-2}. 
In their terminology, 
a (mixed) Hodge structure is a vector space $V$ 
over $\QQ$ for which we have a direct sum 
decomposition 
$\CC\otimes_{\QQ}V\simeq \oplus_{p,q\in \ZZ}V^{p,q}$ 
satisfying the following two conditions:  
\begin{equation}
(i)\ V^{p,q}=\overline{V^{q,p}}, \quad (ii)\ 
\bigoplus_{p+q=m}V^{p,q}\mbox{ is defined over $\QQ$ 
for any $m \in \ZZ$}. 
\end{equation}
By defining the morphisms of Hodge structures naturally, 
we obtain an abelian category $\mathrm{HS}$ of 
Hodge structures. Similarly, we can define an 
abelian category $\mathrm{HS}^{\mathrm{mon}}$ 
of Hodge structures with a quasi-unipotent 
automorphism (``quasi-unipotent"  
means that some powers of it is unipotent 
as in \cite[Section 3.1.3]{D-L-2}) 
and denote by $K_{0}(\mathrm{HS}^{\mathrm{mon}})$ its 
Grothendieck group. 
By the tensor product of 
Hodge structures, we see that 
$K_{0}(\mathrm{HS}^{\mathrm{mon}})$ has a 
ring structure. Then we can define a ring 
homomorphism 
\begin{equation}
\chi_{h}\colon\M^{\hat{\mu}}_{\CC}\longrightarrow
 K_{0}(\mathrm{HS}^{\mathrm{mon}})
\end{equation}
in the following way. For a variety $Y$ with 
good action of the cyclic group $\mu_{d}$ 
we set 
\begin{equation}
\chi_{h}([Y])=\sum_{j\in\ZZ}(-1)^{j}
{[H^{j}_{c}(Y;\QQ)]} \qquad \in K_{0}(\mathrm{HS}^{\mathrm{mon}}), 
\end{equation}
where $H^{j}_{c}(Y;\QQ)$ stands for the Hodge structure 
endowed with the (quasi-unipotent) automorphism induced by 
the morphism $Y \longrightarrow Y$  
($y\longmapsto \zeta_{d}y$). 
We call the ring homomorphism $\chi_{h}$ 
thus obtained the Hodge characteristic morphism. 
Let us define an element of 
$K_{0}(\mathrm{HS}^{\mathrm{mon}})$ 
associated to $g$. Let 
\begin{equation}
\psi^{H}_{g}\colon \Db ({\rm MHM}(Z))\longrightarrow 
\Db ({\rm MHM}(E))
\end{equation}
be the nearby cycle functor along $g$ for 
mixed Hodge modules. Applying it to the 
constant Hodge module $\CC^{H}_{Z}[n]\in {\rm MHM}(Z)$ 
on $Z$ we obtain a mixed Hodge module 
$\psi^{H}_{g}(\CC^{H}_{Z}[n])\in {\rm MHM}(E)$ on 
$E=g^{-1}(0)$. Here, to be consistent with 
the notations of Denef-Loeser \cite{D-L-1} we 
denote $\QQ_{Z}^{H}[n]$ by $\CC^{H}_{Z}[n]$. 
Note that its underlying $\QQ$-perverse sheaf is 
\begin{equation}
\mathrm{rat}\psi^{H}_{g}(\CC^{H}_{Z}[n])
= 
^p\!\psi_{g} (\QQ_{Z}[n])=\psi_{g}(\QQ_{Z}[n-1]). 
\end{equation}
Let 
\begin{equation}
\psi^{H}_{g}(\CC^{H}_{Z}[n])\simeq 
\bigoplus_{\lambda\in\CC}\psi^{H}_{g,\lambda}(\CC^{H}_{Z}[n])
\end{equation}
be the direct sum decomposition of $\psi^{H}_{g}(\CC^{H}_{Z}[n])$ 
by the generalized eigenspaces of its 
monodromy automorphism. Here $\lambda\in\CC$ 
ranges through roots of unity. Applying the 
proper direct image functor 
\begin{equation}
(a_{E})_{!}\colon \Db ({\rm MHM}(E))\longrightarrow  \Db 
({\rm MHM}({\rm pt}))
\end{equation}
for the map 
$a_{E}\colon E=g^{-1}(0)\longrightarrow {\rm pt}$ 
to the one point space ${\rm pt}$ to 
$\psi^{H}_{g}(\CC^{H}_{Z}[n])$, we obtain 
mixed Hodge structures 
\begin{equation}
H^{j}(a_{E})_{!}\psi^{H}_{g}(\CC^{H}_{Z}[n])
\in {\rm MHM}({\rm pt})\qquad (j\in\ZZ)
\end{equation}
and quasi-unipotent automorphisms of them 
induced by the semisimple parts of their 
monodromies. We thus obtain an element 
\begin{equation}
\chi_{c}(E;\psi^{H}_{g}(\CC^{H}_{Z}[n])):=
\sum_{j\in\ZZ}(-1)^{j}{[H^{j}(a_{E})_{!}
\psi^{H}_{g}(\CC^{H}_{Z}[n])]}
\end{equation}
of $K_{0}(\mathrm{HS}^{\mathrm{mon}})$. 

\begin{theorem}{\rm (Denef-Loeser 
\cite[Theorem 4.2.1]{D-L-1})}\label{Th-DL} 
In the Grothendieck ring $K_{0}(\mathrm{HS}^{\mathrm{mon}})$ 
there exists an equality 
\begin{equation}
\chi_{c}(E;\psi^{H}_{g}(\CC^{H}_{Z}[n]))=
(-1)^{n-1}\chi_{h}(\MCS_{g}).
\end{equation}
\end{theorem}

\begin{proof}
First we recall the results explained in 
e.g. Dimca-Saito \cite[Section 1.4]{D-S}. 
For a root of unity $\lambda\in\CC$ 
the direct sum 
$\psi^{H}_{g,\lambda}(\CC^{H}_{Z}[n])\oplus 
\psi^{H}_{g,\overline{\lambda}}(\CC^{H}_{Z}[n])$ 
corresponds to a $\QQ$-perverse sheaf and 
has a weight filtration $W$ of a mixed 
Hodge module. 
We denote the filtration 
induced by it on $\psi^{H}_{g,\lambda}(\CC^{H}_{Z}[n])$ 
by the same letter $W$. Then for the logarithm 
\begin{equation}
N\colon \psi^{H}_{g,\lambda}(\CC^{H}_{Z}[n])\longrightarrow 
\psi^{H}_{g,\lambda}(\CC^{H}_{Z}[n])
\end{equation}
of the unipotent part of the monodromy on 
$\psi^{H}_{g,\lambda}(\CC^{H}_{Z}[n])$,  
there exists an isomorphism 
\begin{equation}
N^{k}\colon \Gr^{W}_{n-1+k}\psi^{H}_{g,\lambda}(
\CC^{H}_{Z}[n])\simto 
\Gr^{W}_{n-1-k}\psi^{H}_{g,\lambda}(\CC^{H}_{Z}[n])
\end{equation}
for any $k\geq 0$. Namely the filtration $W$ on 
the eigenvalue $\lambda$-part 
$\psi^{H}_{g,\lambda}(\CC^{H}_{Z}[n])$
is the monodromy filtration with center 
$n-1=\dim{Z}-1$. For $k\geq 0$ we define a 
submodule of 
$\Gr^{W}_{n-1+k}\psi^{H}_{g,\lambda}(\CC^{H}_{Z}[n])$ by 
\begin{align*}
&\mathrm{P}\Gr^{W}_{n-1+k}\psi^{H}_{g,\lambda}(\CC^{H}_{Z}[n])\\
:=& {\rm Ker} {\left[N^{k+1}\colon \Gr^{W}_{n-1+k}\psi^{H}_{g,\lambda}(
\CC^{H}_{Z}[n]) \longrightarrow  \Gr^{W}_{n-1-k-2}
\psi^{H}_{g,\lambda}(\CC^{H}_{Z}[n])\right]}. 
\end{align*}
We call it the primitive part of 
$\Gr^{W}_{n-1+k}\psi^{H}_{g,\lambda}(\CC^{H}_{Z}[n])$ 
with respect to $N$. Then there exists an 
isomorphism 
\begin{equation}
\bigoplus_{j}\Gr^{W}_{j}\psi^{H}_{g,\lambda}(\CC^{H}_{Z}[n])\simeq 
\bigoplus_{k\geq 0}\bigoplus_{i=0}^{k}{\left[N^{i}\mathrm{P}
\Gr^{W}_{n-1+k}\psi^{H}_{g,\lambda}(\CC^{H}_{Z}[n])\right]}(i). 
\end{equation}
Here $(i)$ stands for the Tate twist which raises 
weights by $2i$. We can describe the structure of 
the primitive part $\mathrm{P}\Gr^{W}_{n-1+k}  
\psi^{H}_{g,\lambda}(\CC^{H}_{Z}[n])$ precisely 
as follows. For a non-zero complex number $\alpha\in\CC$ 
let $\L_{\alpha}$ be the rank one complex local 
system over $\CC^{*}$ having the 
monodromy $\alpha$ around the origin 
$0 \in \CC$. For the root of unity 
$\lambda$, we take $d \in \ZZ_{>0}, b\in\ZZ_{+}$ which 
are coprime each other such that 
\begin{equation}
\lambda=\exp\left(2\pi\sqrt{-1}\frac{b}{d}\right)
\end{equation}
and consider the covering map 
\begin{equation}
\pi_{d}\colon \CC^{*}\longrightarrow 
 \CC^{*}\qquad (t\longmapsto t^{d})
\end{equation}
of degree $d$. Then there exists 
an isomorphism 
\begin{equation}
{\rm R } 
(\pi_{d})_{*}\CC^{H}_{\CC^{*}}\simeq 
\bigoplus_{\alpha:\alpha^{d}=1}\L_{\alpha}.
\end{equation}
Moreover, by the inclusion map 
$j\colon Z^{*}:=Z\setminus E \hookrightarrow Z$
and the restriction $g|_{Z^{*}}\colon Z^{*}\longrightarrow 
 \CC^{*}$ of $g$ to $Z^{*}$ we define a constructible 
sheaf $\F_{\lambda}$ on $Z$ by 
\begin{equation}
\F_{\lambda}:=j_{*}(g|_{Z^{*}})^{-1}\L_{\lambda^{-1}}
\end{equation}
and set 
\begin{equation}
J(\lambda):=\{1\leq i \leq k \ |\ \lambda^{m_{i}}=1\
 ( \Longleftrightarrow d| m_{i})\}\ 
\subset\ \{1,2,\dots,k\}. 
\end{equation}
Since the monodromy of the local system 
$(g|_{Z^{*}})^{-1}\L_{\lambda^{-1}}$ on $Z^*$ 
around the divisor $E_i \subset Z$ 
is $\lambda^{-m_{i}}\in\CC^{*}$, we see that 
$\F_{\lambda}$ is a rank one complex local system 
on a neighborhood of 
\begin{equation}
U(\lambda):=\bigsqcup_{I\subset J(\lambda)}E^{\circ}_{I}
\quad \subset Z. 
\end{equation}
The following result was proved by Saito 
\cite[Section 3.3]{Saito-2} and 
Denef-Loeser \cite[Lemma 4.2.4]{D-L-1}. 

\begin{lemma}{\rm (Saito \cite[Section 3.3]{Saito-2} and 
Denef-Loeser \cite[Lemma 4.2.4]{D-L-1})} 
For any $k \geq 0$ there exists an isomorphism 
\begin{equation}
{\rm rat} \ 
\mathrm{P}\Gr^{W}_{n-1+k}\psi^{H}_{g,\lambda}(\CC^{H}_{Z}[n])\simeq 
\bigoplus_{\substack{I\subset J(\lambda)\\ \sharp I =k+1}}
{\rm IC}_{E_{I}}(\F_{\lambda}|_{E^{\circ}_{I}})(-k), 
\end{equation}
where ${\rm IC}_{E_{I}}(\F_{\lambda}|_{E^{\circ}_{I}})$ 
stands for the minimal 
extension of the perverse sheaf 
$\F_{\lambda}|_{E^{\circ}_{I}}[n- \sharp I]$ 
on $E^{\circ}_{I}$ to $E_{I}$. 
\end{lemma}

For a non-empty subset $I\subset J(\lambda)$ 
we define an open subset $U_{I}\subset U(\lambda)$ of 
$E_I$ containing $E^{\circ}_{I}$ by 
\begin{equation}
U_{I}:=E_{I}\setminus \bigcup_{i\notin 
J(\lambda)}E_{i}=E_{I}\cap U(\lambda)\subset E_{I}. 
\end{equation}
Then by \cite[Corollary 8.2.6]{H-T-T} the minimal extension 
${\rm IC}_{E_{I}}(\F_{\lambda}|_{E^{\circ}_{I}})$ 
is isomorphic to $\F_{\lambda}|_{U_{I}}[n- \sharp I ]$ 
on $U_I$. Moreover for the inclusion map 
$j_{I}\colon U_{I}\hookrightarrow E_{I}$ we have an 
isomorphism 
\begin{equation}
(j_{I})_{!}(\F_{\lambda}|_{U_{I}}[n- \sharp I ])
\overset{\sim}{\longrightarrow}
{\rm R} (j_{I})_{*}(\F_{\lambda}|_{U_{I}}[n- \sharp I ]). 
\end{equation}
Hence this is isomorphic to the minimal extension 
${\rm IC}_{E_{I}}(\F_{\lambda}|_{E^{\circ}_{I}})$. 
We thus obtain isomorphisms 
\begin{align*}
&\chi_{c}(E;\psi^{H}_{g,\lambda}(\CC^{H}_{Z}[n]))\\
=&\sum_{k\geq 0}\sum_{i=0}^{k}\chi_{c}(E;N^{i}\mathrm{P}
\Gr^{W}_{n-1+k}\psi^{H}_{g,\lambda}(\CC^{H}_{Z}[n])(i))\\
=&\sum_{k\geq 0}\chi_{c}(E;\mathrm{P}\Gr^{W}_{n-1+k}
\psi^{H}_{g,\lambda}(\CC^{H}_{Z}[n]))\cdot(\sum_{i=0}^{k}\CC(i))\\
=&\sum_{k\geq 0}\sum_{\substack{I\subset J(\lambda)\\ \sharp I =k+1}}
\chi_{c}(U_{I};\F_{\lambda}|_{U_{I}}[n- \sharp I ])\cdot (\sum_{i=0}^{k}\CC(-i)).
\end{align*}
Note that for the Lefschetz motive $\mathbb{L}=[\CC]\in 
\M_{\CC}^{\hat{\mu}}$ we have $\chi_{h}(\mathbb{L})=\CC(-1)$. 
Since for any non-empty subset $I\subset J(\lambda)$ we have 
\begin{equation}
U_{I}=\bigsqcup_{I'\subset J(\lambda), I'\supset I}E^{\circ}_{I'}, 
\end{equation}
there exists an equality 
\begin{equation}
\chi_{c}(U_{I};\F_{\lambda}|_{U_{I}}[n- \sharp I ])=(-1)^{n- \sharp I }
\sum_{I'\subset J(\lambda),I'\supset I}\chi_{c}(
E^{\circ}_{I'};\F_{\lambda}|_{E^{\circ}_{I'}}). 
\end{equation}
If we rewrite it replacing $I'\subset J(\lambda)$ 
by $I$, we obtain an equality 
\begin{equation}
\chi_{c}(E;\psi^{H}_{g,\lambda}(\CC^{H}_{Z}[n]))=
\sum_{I\subset J(\lambda)}\alpha_{I}\cdot 
\chi_{c}(E^{\circ}_{I};\F_{\lambda}|_{E^{\circ}_{I}}), 
\end{equation}
where we set 
\begin{equation}
\alpha_{I}=\sum_{k=0}^{ \sharp I -1}(-1)^{n-k-1}
\binom{ \sharp I }{k+1}\cdot \{1+\CC(-1)+\dots +\CC(-k)\}. 
\end{equation}
Moreover, by the binomial theorem, we can easily show 
\begin{equation}
\alpha_{I}=(-1)^{n-1}(1-\CC(-1))^{ \sharp I -1}=
(-1)^{n-1}\chi_{h}((1-\mathbb{L})^{ \sharp I -1}). 
\end{equation}
Therefore it suffices to prove the equality 
\begin{equation}
\chi_{c}(E_{I}^{\circ};\F_{\lambda}|_{E^{\circ}_{I}})
=\chi_{h}({[\tl{E^{\circ}_{I}}]})_{\lambda}
\end{equation}
for any non-empty subset $I\subset J(\lambda)$. Here 
$\chi_{h}([\tl{E^{\circ}_{I}}])_{\lambda}$ stands 
for the eigenvalue $\lambda$-part of 
$\chi_{h}({[\tl{E^{\circ}_{I}}]})$. On an 
affine open neighborhood $W \subset Z$ of a point 
of $E^{\circ}_{I}$ in $Z$ we decompose $g$ as 
\begin{equation}
g=g_{1,W} \cdot (g_{2,W})^{m_{I}}\qquad (\mbox{$g_{1,W}$ 
is invertible on $W$} ). 
\end{equation}
Then we can easily see that $\F_{\lambda}$ is 
isomorphic to $(g_{1,\lambda})^{-1}\L_{\lambda^{-1}}$. 
We define a covering $\widehat{E^{\circ}_{I}\cap W}$ 
of the Zariski open subset $E^{\circ}_{I}\cap W$ of 
$E^{\circ}_{I}$ of degree $m_{I}$ by 
the Cartesian square: 
\begin{equation}
\begin{xy}
\xymatrix{
\widehat{E^{\circ}_{I}\cap W}\ar@{}[rd]|-{\Box}
\ar[d]_{q}\ar[r]^{p} &E^{\circ}_{I}\cap W\ar[d]^{g_{1,W}}\\
\CC^{*}\ar[r]_{\pi_{m_{I}}}&\CC^{*}.
}
\end{xy}
\end{equation}
Namely we set 
\begin{equation}
\widehat{E^{\circ}_{I}\cap W}:=
\{(t,z)\in\CC^{*}\times (E^{\circ}_{I}\cap W)\ |\ t^{m_{I}}=g_{1,W}(z)\}.
\end{equation}
Then there exist isomorphisms 
\begin{align*}
& {\rm R} 
\Gamma_{c}(\widehat{E^{\circ}_{I}\cap W};\CC_{\widehat{E^{\circ}_{I}\cap W}})\\
\simeq & {\rm R} 
\Gamma_{c}(E^{\circ}_{I}\cap W; {\rm R} p_{*}
\CC_{\widehat{E^{\circ}_{I}\cap W}})\\
\simeq & {\rm R}
\Gamma_{c}(E^{\circ}_{I}\cap W;(g_{1,W})^{-1} 
{\rm R} (\pi_{m_{I}})_{*}\CC_{\CC^{*}}).
\end{align*}
Since $\widehat{E^{\circ}_{I}\cap W}$ is isomorphic to 
the part $\tl{E^{\circ}_{I}\cap W}$ of 
the unramified Galois covering $\tl{E^{\circ}_{I}}$ over 
$E^{\circ}_{I}\cap W\subset E^{\circ}_{I}$ by 
the map $(t,z)\mapsto (t^{-1},z)$, we see also 
that the eigenvalue $\lambda$-part of 
${\rm R} \Gamma_{c}(\tl{E^{\circ}_{I}\cap W};
\CC_{\tl{E^{\circ}_{I}\cap W}})$ corresponds to 
${\rm R} \Gamma_{c}(E^{\circ}_{I}\cap W;(g_{1,W})^{-1}\L_{\lambda^{-1}})$ 
in the above isomorphisms. Then by the 
Mayer-Vietoris exact sequence for an open covering 
of $E^{\circ}_{I}$ by such affine open subsets  
$W \subset Z$ we immediately obtain the assertion. 
This completes the proof. 
\qed
\end{proof}

\section{Motivic Milnor fibers at infinity and their applications}\label{section 7}

In this section, we introduce the theory of 
motivic Milnor fibers at infinity developed in Matsui-Takeuchi 
\cite{M-T-3} and Raibaut \cite{Raibaut}, 
\cite{Raibaut-1.5}, \cite{Raibaut-1.8}. 
Combining it with the results in Section \ref{section 5},  
we describe the 
equivariant mixed Hodge numbers 
of motivic Milnor fibers at infinity 
in terms of Newton polyhedra at infinity of polynomials. 
Then we show how the combinatorial 
expressions of the Jordan normal forms of 
monodromies at infinity in \cite{M-T-3} 
are obtained. First, we shall recall the definition of 
motivic Milnor fibers at infinity. For this purpose, 
we consider the situation in Section \ref{section 4}. 
Let $f \colon \CC^n \longrightarrow \CC$ be a polynomial map. 
We take a smooth compactification $X$ of $\CC^n$. 
Then by eliminating the points of indeterminacy of the 
meromorphic extension of $f$ to $X$ we obtain a commutative diagram
\begin{equation}
\xymatrix{
\CC^n \ar@{^{(}->}[r]^{\iota} \ar[d]_f & \tl{X} \ar[d]^g
\\
\CC \ar@{^{(}->}[r]^j & \PP^1}
\end{equation}
of holomorphic maps such that $\tl{X}$ is smooth, 
$g$ is proper and $\tl{X} \setminus \CC^n$, 
$Y:=g^{-1}( \infty )$ are normal crossing divisors in $\tl{X}$. Take a local 
coordinate $h$ of $\PP^1$ on a neighborhood of $\infty\in \PP^1$ 
such that $\infty=\{h=0\}$ and set $\tl{g}=h\circ g$. Note that $\tl{g}$ 
is a holomorphic function defined on a neighborhood of the closed subvariety 
$Y=\tl{g}^{-1}(0)=g^{-1}(\infty) \subset \tl{X} \setminus \CC^n$ of $\tl{X}$. Then 
by Proposition \ref{prp.7.1.002}, for $R \gg 0$ we 
obtain isomorphisms 
\begin{equation}\label{eq:4-2}
H_c^j(f^{-1}(R);\CC) \simeq H^j \psi_h( j_! {\rm R}f_! \CC_{\CC^n} ) \simeq 
H^j(Y; \psi_{\tl{g}}(\iota_! \CC_{\CC^n})).
\end{equation}
Let us define an open subset $\Omega$ of $\tl{X}$ by
\begin{equation}
\Omega=\Int (\iota(\CC^n) \sqcup Y)
\end{equation}
and set $U=\Omega \cap Y \subset Y$. Then $U$ (resp. the complement of $\Omega$ 
in $\tl{X}$) is a normal crossing divisor in $\Omega$ 
(resp. $\tl{X}$). Hence by Corollary \ref{new-corol} 
and Lemma \ref{lem:2-ac-1}, 
we can easily prove the isomorphisms
\begin{equation}
H^j(Y;\psi_{\tl{g}}(\iota_! \CC_{\CC^n}))\simeq 
H^j(Y; \psi_{\tl{g}}(\iota_!^{\prime}\CC_{\Omega}))\simeq H_c^j(U; \psi_{\tl{g}}(\CC_{\tl{X}})),
\label{eq:4-5}
\end{equation}
where $\iota^{\prime} \colon \Omega \longhookrightarrow \tl{X}$ 
is the inclusion map. Now let $E_1, E_2, \ldots, E_k$ be the irreducible components of 
the normal crossing divisor $U=\Omega \cap Y$ in $\Omega \subset \tl{X}$. 
For each $1 \leq i \leq k$, let $b_i>0$ be the order of the zero of 
$\tl{g}$ along $E_i$. For a non-empty subset $I \subset \{1,2,\ldots, k\}$, let us set
\begin{equation}
E_I=\bigcap_{i \in I} E_i,\hspace{10mm}E_I^{\circ}=E_I \setminus \bigcup_{i \not\in I}E_i
\end{equation}
and $d_I=\gcd (b_i)_{i \in I}>0$. 
Then as in Section \ref{section 6}, we 
can construct an unramified Galois covering $\tl{E_I^{\circ}} \longrightarrow E_I^{\circ}$ of $E_I^{\circ}$ 
of degree $d_I$. 
Moreover it is endowed with 
a natural good action of the cyclic group $\mu_{d_I}$. 
Namely the variety $\tl{E_I^{\circ}}$ has a good 
$\hat{\mu}$-action in the sense of 
Denef-Loeser \cite[Section 2.4]{D-L-2}. 

\begin{definition}{\rm (Matsui-Takeuchi \cite{M-T-3} and Raibaut \cite{Raibaut-1.5})}\label{def-MT-R}
We define the motivic Milnor fiber at infinity 
$\SS_f^{\infty} \in \M_{\CC}^{\hat{\mu}}$ of the polynomial map $f \colon \CC^n \longrightarrow \CC$ by
\begin{equation}\label{MMF}
\SS_f^{\infty} =\sum_{I \neq \emptyset} (1-\LL)^{\sharp I -1} [\tl{E_I^{\circ}}] 
\quad \in \M_{\CC}^{\hat{\mu}}.
\end{equation}
\end{definition}

By Guibert-Loeser-Merle \cite[Theorem 3.9]{G-L-M}, 
the motivic Milnor fiber at infinity $\SS_f^{\infty}$ of $f$ does not 
depend on the choice of the smooth compactification $\tl{X}$ of $\CC^n$. 
Recall that $\HSm$ denotes the abelian category of 
Hodge structures with a quasi-unipotent endomorphism. 
Then, to the object $\psi_h(j_! {\rm R} f_!\CC_{\CC^n})\in \Dbc(\{\infty\})$ 
and the semisimple part of the monodromy automorphism acting on it, we can associate an element
\begin{equation}
[H_f^{\infty}] \in \KK_0(\HSm)
\end{equation}
in a standard way. Similarly, to $\psi_h( {\rm R} j_* {\rm R} f_*\CC_{\CC^n})
\in \Dbc(\{\infty\})$ we associate an element
\begin{equation}
[G_f^{\infty}] \in \KK_0(\HSm).
\end{equation}
According to a deep result \cite[Theorem 13.1]{Sabbah-2} of Sabbah, 
if $f$ is tame at infinity then the weights of the element $[G_f^{\infty}]$ are 
defined by the monodromy filtration up to some Tate twists (see also \cite{Saito-1} and \cite{Saito-2}). 
This implies that for the calculation of the monodromy 
at infinity $\Phi_{n-1}^{\infty}\colon H^{n-1}(f^{-1}(R);\CC) \simto 
H^{n-1}(f^{-1}(R);\CC)$ ($R \gg 0$) of $f$ it suffices to 
calculate $[H_f^{\infty}]\in \KK_0(\HSm)$ which is the dual of $[G_f^{\infty}]$. 
To describe the element $[H_f^{\infty}]\in \KK_0(\HSm)$ in 
terms of $\SS_f^{\infty}\in \M_{\CC}^{\hat{\mu}}$, let
\begin{equation}
\chi_h \colon \M_{\CC}^{\hat{\mu}} \longrightarrow \KK_0(\HSm)
\end{equation}
be the Hodge characteristic morphism. Then by applying Theorem \ref{Th-DL} 
to our situation \eqref{eq:4-2} and \eqref{eq:4-5}, we obtain the following result.

\begin{theorem}{\rm (Matsui-Takeuchi \cite[Theorem 4.4]{M-T-3} and Raibaut \cite{Raibaut-1.5})}\label{thm:7-6}
In the Grothendieck group $\KK_0(\HSm)$, we have
\begin{equation}
[H_f^{\infty}]=\chi_h(\SS_f^{\infty}).
\end{equation}
\end{theorem}

On the other hands, the results in \cite{Sabbah} and \cite{Sabbah-2} imply 
the following symmetry of the weights of the element $[H_f^{\infty}] \in \KK_0(\HSm)$ when 
$f$ is tame at infinity. 
See \cite[Appendix]{M-T-3} 
for the details. Later another proof 
was found in Dimca-Saito \cite{D-S-new}. Recall that if $f$ is tame at infinity 
we have $H_c^j(f^{-1}(R); \CC )=0$ ($R \gg 0$) for $j \not= n-1, 2n-2$ and $H_c^{2n-2}(f^{-1}(R); \CC ) 
\simeq [H^0(f^{-1}(R); \CC )]^* \simeq \CC$. For an element 
$[V] \in \KK_0(\HSm)$ defined by $V \in \HSm $ with a quasi-unipotent endomorphism 
$\Theta \colon V \simto V$, $p, q \geq 0$ and $\lambda \in \CC$ denote by 
$e^{p,q}([V])_{\lambda}$ the dimension of the $\lambda$-eigenspace of 
the morphism $V^{p,q} \simto V^{p,q}$ induced by $\Theta$ on the $(p,q)$-part $V^{p,q}$ of $V$.

\begin{theorem}{\rm (Sabbah \cite{Sabbah} and \cite{Sabbah-2})}\label{cor:7-6-2}
Assume that $f$ is tame at infinity. Then
\begin{enumerate}
\item Let $\lambda \in \CC^* \setminus \{1\}$. Then we have 
$e^{p,q}( [H_f^{\infty}])_{\lambda}=0$ for $(p,q) \notin [0,n-1] \times [0,n-1]$. 
Moreover for $(p,q) \in [0,n-1] \times [0,n-1]$ we have
\begin{equation}
e^{p,q}( [H_f^{\infty}])_{\lambda}=e^{n-1-q,n-1-p}( [H_f^{\infty}])_{\lambda}.
\end{equation}
\item We have $e^{p,q}( [H_f^{\infty}])_{1}=0$ for $(p,q) \notin \{(n-1, n-1)\} 
\sqcup ([0,n-2] \times [0,n-2])$ and $e^{n-1,n-1}( [H_f^{\infty}])_{1}=1$. Moreover for $(p,q) \in [0,n-2] \times [0,n-2]$ we have
\begin{equation}
e^{p,q}( [H_f^{\infty}])_{1}=e^{n-2-q,n-2-p}( [H_f^{\infty}])_{1}.
\end{equation}
\end{enumerate}
\end{theorem}

Since the weights of $[G_f^{\infty}] \in \KK_0(\HSm)$ are defined by 
the monodromy filtration and $[G_f^{\infty}]$ is the dual of $[H_f^{\infty}]$ up to some Tate twist, we obtain the following result.

\begin{theorem}{\rm (Matsui-Takeuchi \cite[Theorem 4.6]{M-T-3})}\label{thm:7-6-2}
Assume that $f$ is tame at infinity. Then we have 
\begin{enumerate}
\item Let $\lambda \in \CC^* \setminus \{1\}$ and $k \geq 1$. Then the number of the 
Jordan blocks for the eigenvalue $\lambda$ with sizes $\geq k$ in $\Phi_{n-1}^{\infty} 
\colon H^{n-1}(f^{-1}(R);\CC) \simto H^{n-1}(f^{-1}(R);\CC)$ ($R \gg 0$) is equal to
\begin{equation}
(-1)^{n-1} \sum_{p+q=n-2+k, n-1+k} e^{p,q}( \chi_h(\SS_f^{\infty}))_{\lambda}.
\end{equation}
\item For $k \geq 1$, the number of the Jordan blocks for the 
eigenvalue $1$ with sizes $\geq k$ in $\Phi_{n-1}^{\infty}$ is equal to
\begin{equation}
(-1)^{n-1} \sum_{p+q=n-2-k, n-1-k} e^{p,q}( \chi_h(\SS_f^{\infty}))_{1}.
\end{equation}
\end{enumerate}
\end{theorem}

In terms of the Newton polyhedron at infinity $\Gamma_{\infty}(f)$ of $f$, 
we can rewrite the above results as follows. For this purpose, assume 
that $f\in \CC[x_1,\ldots,x_n]$ is convenient and non-degenerate at infinity. 
Then $f$ is tame at infinity and it suffices to calculate $\Phi_j^{\infty}$ 
only for $j=n-1$. In fact, as was shown in Takeuchi-Tib{\u a}r 
\cite{T-T}, even if we do not assume that $f$ is convenient 
we can obtain also the combinatorial expressions of the Jordan 
normal forms of $\Phi_{n-1}^{\infty}$ for some 
eigenvalues $\lambda \not= 1$. However, in order to 
introduce the results for the eigenvalue $1$ in \cite{M-T-3}, 
here we restrict ourselves to the case where 
 $f$ is convenient. Then, unlike 
in the proof of Theorem \ref{thm:3-5}, we first 
construct a smooth toric compactification of 
not $T= ( \CC^*)^n$ but $\CC^n$ itself as follows. 
Let us consider $\CC^n$ as the toric variety associated to 
the fan $\Sigma_0$ formed by the all faces of the first 
quadrant $(\RR^n)^*_+ \subset (\RR^n)^*$.  
We denote by $\Sigma_1$ the dual fan of 
$\Gamma_{\infty}(f)$ in $(\RR^n)^*$. 
Then by the convenience of $f$, we see that 
any cone $\sigma \in \Sigma_0$ in $\Sigma_0$ is 
an element of $\Sigma_1$. Namely $\Sigma_0$ is 
a subfan of $\Sigma_1$. Hence we can construct a 
smooth subdivision $\Sigma$ of $\Sigma_1$ without 
subdividing the cones in $\Sigma_0$ 
(see e.g. \cite[Chapter II, Lemma (2.6)]{Oka}). 
Let $X_{\Sigma}$ be the smooth toric variety associated 
to $\Sigma$. Since $\Sigma_0$ is still a subfan of 
$\Sigma$ by our construction, $\CC^n$ is an 
affine open subset of $X_{\Sigma}$. 
Let $\rho_1, \ldots, \rho_m$ be the $1$-dimensional cones in 
the smooth fan $\Sigma$ such that $\rho_i \not\subset ( \RR^n)^*_+$. 
We call them the rays at infinity. 
Each ray $\rho_i$ at infinity corresponds to a 
smooth toric divisor $D_i$ in $X_{\Sigma}$ 
and the divisor $D:=D_1\cup \cdots \cup D_m=X_{\Sigma} \setminus \CC^n$ 
in $X_{\Sigma}$ is normal crossing. 
We denote by $a_i>0$ the order of the poles 
of $f$ along $D_i$. Then as in the proof of Theorem \ref{thm:3-5}, 
we can eliminate the points of 
indeterminacy of the meromorphic extension of $f$ to $X_{\Sigma}$ 
and construct a commutative diagram
\begin{equation}
\xymatrix{
\CC^n \ar@{^{(}->}[r]^{\iota} \ar[d]_f & \tl{X_{\Sigma}} \ar[d]^g\\\CC \ar@{^{(}->}[r]^j & \PP^1}
\end{equation}
of holomorphic maps. Take a local coordinate $h$ of $\PP^1$ on a neighborhood of 
 $\infty \in \PP^1$ such that $\infty=\{h=0\}$ and set $\tl{g}=h\circ g$, $Y=\tl{g}^{-1}(0)=g^{-1}(\infty) 
\subset \tl{X_{\Sigma}}$ and $\Omega=\Int(\iota(\CC^n) \sqcup Y)$. For simplicity, let us set 
$\tl{g}=\frac{1}{f}$. Then the divisor $U=Y \cap \Omega$ in $\Omega$ contains 
not only the proper transforms $D_1^{\prime}, \ldots, D_m^{\prime}$ of $D_1, \ldots, D_m$ in $\tl{X_{\Sigma}}$ 
but also the exceptional divisors of the blow-up: $\tl{X_{\Sigma}} \longrightarrow X_{\Sigma}$ 
(see the proof of Theorem \ref{thm:3-5} for the details). 
However these exceptional divisors are not necessary to compute the 
monodromy at infinity of $f \colon \CC^n \longrightarrow \CC$ as we see below. 
For each non-empty subset $I \subset \{1,2,\ldots, m\}$, set $D_I= \bigcap_{i \in I} D_i$,
\begin{equation}
D_I^{\circ}=D_I \setminus \left\{ \( \bigcup_{i \notin I}D_i\) \cup \overline{f^{-1}(0)}\right\} \subset X_{\Sigma}
\end{equation}
and $d_I =\gcd(a_i)_{i \in I} >0$. Then the function 
$\tl{g}=\frac{1}{f}$ is regular on $D_I^{\circ}$ and we can decompose it as $\frac{1}{f}=\tl{g_{1}}(\tl{g_{2}})^{d_I}$ globally 
on a Zariski open neighborhood $W$ of $D_I^{\circ}$ in $X_{\Sigma}$, where $\tl{g_1}$ is a unit on $W$ 
and $\tl{g_2} \colon W \longrightarrow \CC$ is regular. Therefore we can construct an unramified Galois covering 
$\tl{D_I^{\circ}}$ of $D_I^{\circ}$ with a natural $\mu_{d_I}$-action. Let 
$[\tl{D_I^{\circ}}]$ be the element of the ring $\M_{\CC}^{\hat{\mu}}$ defined by $\tl{D_I^{\circ}}$.

\begin{theorem}{\rm (Matsui-Takeuchi \cite[Theorem 4.7]{M-T-3})}\label{thm:7-7}
Assume that $f$ is convenient and non-degenerate at infinity. Then we have an equality
\begin{equation}
\chi_h\(\SS_f^{\infty}\)= \dsum_{I \neq \emptyset}\chi_h\( (1-\LL)^{\sharp I -1} [\tl{D_I^{\circ}}]\)
\end{equation}
in the Grothendieck group $\KK_0(\HSm)$.
\end{theorem}

Recall that a face $\gamma \prec \Gamma_{\infty}(f)$ is 
at infinity if $0 \notin \gamma$. 
For a face at infinity $\gamma \prec \Gamma_{\infty}(f)$ of $\Gamma_{\infty}(f)$, 
let $d_{\gamma}>0$ be the lattice distance of $\gamma$ from the origin $0 \in \RR^n$ and $\Delta_{\gamma}$ 
the convex hull of $\{0\} \sqcup \gamma$ in $\RR^n$. Let $\LL(\Delta_{\gamma})$ be 
the $(\dim \gamma +1)$-dimensional linear subspace of $\RR^n$ spanned by $\Delta_{\gamma}$ 
and consider the lattice $M_{\gamma}=\ZZ^n \cap \LL(\Delta_{\gamma}) \simeq \ZZ^{\dim \gamma+1}$ in it. 
Then we set 
\begin{equation}
T_{\Delta_{\gamma}}:=\Spec (\CC[M_{\gamma}]) \simeq (\CC^*)^{\dim \gamma +1}.
\end{equation}
Moreover let $\LL(\gamma)$ be the smallest affine linear subspace of $\RR^n$ containing $\gamma$ 
and for $v \in M_{\gamma}$ define their lattice heights $\height (v, \gamma) \in \ZZ$ 
from $\LL(\gamma)$ in $\LL(\Delta_{\gamma})$ so that we have $\height (0, \gamma)=d_{\gamma}>0$. 
Then to the group homomorphism $M_{\gamma} \longrightarrow \CC^*$ defined by $v \longmapsto 
\zeta_{d_{\gamma}}^{\height (v, \gamma)}$ we can naturally associate an element 
$\tau_{\gamma} \in T_{\Delta_{\gamma}}$. We define a Laurent polynomial 
$g_{\gamma}=\sum_{v \in M_{\gamma}}b_v x^v$ on $T_{\Delta_{\gamma}}$ by
\begin{equation}
b_v=\begin{cases}
a_v & (v \in \gamma),\\
 & \\
-1 & (v=0),\\
 & \\
\ 0 & (\text{otherwise}),
\end{cases}
\end{equation}
where we set $f=\sum_{v \in \ZZ^n_+} a_v x^v$. Then we have $NP(g_{\gamma}) =\Delta_{\gamma}$, 
$\supp g_{\gamma} \subset \{ 0\} \sqcup \gamma$ and the complex hypersurface $Z_{\Delta_{\gamma}}^*=
\{ x \in T_{\Delta_{\gamma}}\ |\ g_{\gamma}(x)=0\}$ is non-degenerate (see the proof of 
\cite[Proposition 5.3]{M-T-3}). 
Since $Z_{\Delta_{\gamma}}^* \subset T_{\Delta_{\gamma}}$ is invariant by 
the multiplication $l_{\tau_{\gamma}} \colon  T_{\Delta_{\gamma}} \simto T_{\Delta_{\gamma}}$ 
by $\tau_{\gamma}$, $Z_{\Delta_{\gamma}}^*$ admits an action of $\mu_{d_{\gamma}}$. 
We thus obtain an element $[Z_{\Delta_{\gamma}}^*]$ of $\M_{\CC}^{\hat{\mu}}$. 
For a face at infinity $\gamma \prec \Gamma_{\infty}(f)$ let 
$S_{\gamma} \subset \{1,2,\ldots, n\}$ be the minimal subset of $\{1,2,\ldots,n\}$ 
such that $\gamma \subset \RR^{S_{\gamma}}$ and set $m_{\gamma}=\sharp S_{\gamma}-\dim \gamma -1\geq 0$.

\begin{theorem}{\rm (Matsui-Takeuchi \cite[Theorem 5.7]{M-T-3})}\label{thm:7-12}
Assume that $f$ is convenient and non-degenerate at infinity. Then we have the following 
results, where in the sums $\sum_{\gamma}$ below the face $\gamma$ of $\Gamma_{\infty}(f)$ ranges through those at infinity.
\begin{enumerate}
\item In the Grothendieck group $\KK_0(\HSm)$, we have
\begin{equation}
[H_f^{\infty}]=\chi_h(\SS_f^{\infty})=\sum_{\gamma} \chi_h((1-\LL)^{m_{\gamma}} \cdot [Z_{\Delta_{\gamma}}^*]).
\end{equation}
\item Let $\lambda \in \CC^* \setminus \{1\}$ and $k \geq 1$. Then the number of 
the Jordan blocks for the eigenvalue $\lambda$ with sizes $\geq k$ in 
$\Phi_{n-1}^{\infty} \colon H^{n-1}(f^{-1}(R);\CC) \simto H^{n-1}(f^{-1}(R);\CC)$ ($R \gg 0$) is equal to
\begin{equation}
(-1)^{n-1} \sum_{p+q=n-2+k, n-1+k} \left\{ \sum_{\gamma} e^{p,q} 
\left( \chi_h((1-\LL)^{m_{\gamma}}\cdot [Z_{\Delta_{\gamma}}^*])\right)_{\lambda} \right\}.
\end{equation}
\item For $k \geq 1$, the number of the Jordan blocks for the eigenvalue $1$ with sizes $\geq k$ in $\Phi_{n-1}^{\infty}$ is equal to
\begin{equation}
(-1)^{n-1} \sum_{p+q= n-2-k, n-1-k} \left\{\sum_{\gamma} e^{p,q} 
\left( \chi_h((1-\LL)^{m_{\gamma}}\cdot [Z_{\Delta_{\gamma}}^*])\right)_{1} 
\right\}.
\end{equation}
\end{enumerate}
\end{theorem}

The part (i) of Theorem \ref{thm:7-12} was obtained also by Raibaut \cite{Raibaut-1.8} 
independently. 
Note also that by using the results in Section \ref{section 5} we can always 
calculate $e^{p,q}(\chi_h((1-\LL)^{m_{\gamma}}\cdot [Z_{\Delta_{\gamma}}^*]))_{\lambda}$ 
explicitly. From now on, we shall introduce some closed formulas in \cite{M-T-3} 
for the numbers of the Jordan blocks in $\Phi_{n-1}^{\infty}$. 
First let us consider the numbers of the Jordan blocks for the eigenvalues 
$\lambda \in \CC \setminus \{1\}$. Let $q_1,\ldots,q_l$ (resp. $\gamma_1,\ldots, 
\gamma_{l^{\prime}}$) be the $0$-dimensional (resp. $1$-dimensional) faces of $\Gamma_{\infty}(f)$ 
such that $q_i\in \Int (\RR_+^n)$ (resp. the relative interior $\relint(\gamma_i)$ of $\gamma_i$ is 
contained in $\Int(\RR_+^n)$). Obviously these faces are at infinity. For each $q_i$ (resp. 
$\gamma_i$), denote by $d_i >0$ (resp. $e_i>0$) the lattice distance $\dist(q_i, 0)$ (resp. $\dist(\gamma_i,0)$) 
of it from the origin $0\in \RR^n$. For $1\leq i \leq l^{\prime}$, let $\Delta_i$ be the convex hull of 
$\{0\}\sqcup \gamma_i$ in $\RR^n$. Then for $\lambda \in 
\CC \setminus \{1\}$ and $1 \leq i \leq l^{\prime}$ such that $\lambda^{e_i}=1$ we set
\begin{align}
n(\lambda)_i
 := & 
\sharp\{ v\in \ZZ^n \cap \relint(\Delta_i) \ |\ \height (v, \gamma_i)=k\} 
\nonumber 
\\
 & +  \sharp \{ v\in \ZZ^n \cap \relint(\Delta_i) \ |\ \height (v, \gamma_i)=e_i-k\},
\end{align}
where $k$ is the minimal positive integer satisfying $\lambda=\zeta_{e_i}^{k}$ and 
for $v\in \ZZ^n \cap \relint(\Delta_i)$ we denote by 
$\height (v, \gamma_i)$ the lattice height of $v$ from the base $\gamma_i$ of $\Delta_i$. 

\begin{theorem}{\rm (Matsui-Takeuchi \cite[Theorem 5.9]{M-T-3})}\label{thm:7-11}
Assume that $f$ is convenient and non-degenerate at infinity and let $\lambda \in \CC^* \setminus \{1\}$. 
Then we have 
\begin{enumerate}
\item The number of the Jordan blocks for the eigenvalue $\lambda$ with the 
maximal possible size $n$ in $\Phi_{n-1}^{\infty} \colon H^{n-1}(f^{-1}(R);\CC) 
 \simto H^{n-1}(f^{-1}(R);\CC)$ ($R \gg 0$) is equal to $\sharp \{q_i \ |\ \lambda^{d_i}=1\}$. 
\item The number of the Jordan blocks for the eigenvalue $\lambda$ with size 
$n-1$ in $\Phi_{n-1}^{\infty}$ is equal to $\sum_{i \colon \lambda^{e_i}=1} n(\lambda)_i$.
\end{enumerate}
\end{theorem}

\begin{proof}
\noindent (i) By Theorem \ref{thm:7-12} (ii), the number of the Jordan blocks for the 
eigenvalue $\lambda \in \CC^* \setminus \{1\}$ with the maximal possible size $n$ in $\Phi_{n-1}^{\infty}$ is
\begin{align}
(-1)^{n-1} e^{n-1,n-1}(\chi_h(\SS_f^{\infty}))_{\lambda}
&=(-1)^{n-1} \sum_{i=1}^l e^{n-1,n-1}(\chi_h((1-\LL)^{n-1}\cdot [ Z_{\Delta_{q_i}}^* ]))_{\lambda}\\
&= \sum_{i=1}^l e^{0,0}(\chi_h( [Z_{\Delta_{q_i}}^* ]))_{\lambda}.
\end{align}
Note that $Z_{\Delta_{q_i}}^*$ is a finite subset of $\CC^*$ consisting of $d_i$ points. Then (i) follows from 
\begin{equation}
\sum_{i=1}^l e^{0,0}(\chi_h([Z_{\Delta_{q_i}}^* ]))_{\lambda}=\sharp \{q_i \ |\ \lambda^{d_i}=1\}.
\end{equation}
The assertion (ii) can be proved similarly by expressing 
$e^{n-1,n-2}(\chi_h(\SS_f^{\infty}))_{\lambda} +e^{n-2,n-1}(\chi_h(\SS_f^{\infty}))_{\lambda}$ 
in terms of the $1$-dimensional faces at infinity $\gamma_i$ of $\Gamma_{\infty}(f)$. \qed
\end{proof}

\begin{example}
Let $f(x,y)\in \CC[x,y]$ be a convenient polynomial whose Newton 
polyhedron at infinity $\Gamma_{\infty}(f)$ 
is the convex hull of the five points 
$(0,0), (5,0), (5,1), (2,4), (0,4) \in \RR^2_+$ in $\RR^2_+$. 
Assume also that $f$ is non-degenerate at infinity. 
Then by Libgober-Sperber's theorem (Theorem \ref{thm:3-5}) 
the characteristic polynomial $P(\lambda)$ of $\Phi_1^{\infty} 
\colon H^1(f^{-1}(R);\CC) \simto H^1(f^{-1}(R);\CC)$ ($R \gg 0$) is given by
\begin{equation}
P(\lambda)=(\lambda-1)(\lambda^4-1)(\lambda^6-1)^3.
\end{equation}
In particular, the total multiplicity of the roots $-1$ in $P( \lambda )=0$ is $4$. 
For $\lambda \in \CC$, denote by $H^1(f^{-1}(R);\CC)_{\lambda}$ the generalized $\lambda$-eigenspace 
of the monodromy operator $\Phi_1^{\infty}$ at infinity. First, by the monodromy theorem the 
restriction of $\Phi_1^{\infty}$ to $H^1(f^{-1}(R);\CC)_1 \simeq \CC^5$ is semisimple. 
Moreover by Theorem \ref{thm:7-11} (i) the Jordan normal form of the restriction of 
$\Phi_1^{\infty}$ to $H^1(f^{-1}(R);\CC)_{-1} \simeq \CC^4$ is
\begin{equation}
\begin{pmatrix}
-1 &1&0&0\\
0&-1&0&0\\
0&0&-1&0\\
0&0&0&-1
\end{pmatrix}.
\end{equation}
In the same way, we can show that for $\lambda=\zeta_6, 
 \sqrt{-1}, \zeta_3, \zeta_3^2, -\sqrt{-1}, \zeta_6^5$ 
the restriction of $\Phi_1^{\infty}$ to $H^1(f^{-1}(R);\CC)_{\lambda}$ is semisimple.
\end{example}

Next we consider the number of the Jordan blocks for the eigenvalue $1$ in $\Phi_{n-1}^{\infty}$. 
By the results in Section \ref{section 5}, we can rewrite Theorem \ref{thm:7-12} (ii), (iii) as follows. 
Denote by $\Pi_f$ the number of the lattice points on the $1$-skeleton of $\partial \Gamma_{\infty}(f) \cap \Int (\RR^n_+)$. 
For a face at infinity $\gamma \prec \Gamma_{\infty}(f)$, we denote by $l^*(\gamma)$ the number of 
the lattice points on the relative interior $\relint(\gamma)$ of $\gamma$. 

\begin{theorem}{\rm (Matsui-Takeuchi \cite[Theorems 5.11 and 5.12]{M-T-3})}\label{thm:7-16}
Assume that $f$ is convenient and non-degenerate at infinity. Then we have 
\begin{enumerate}
\item The number of the Jordan blocks for the eigenvalue $1$ with 
the maximal possible size $n-1$ in $\Phi_{n-1}^{\infty}$ is $\Pi_f$.
\item The number of the Jordan blocks for the eigenvalue $1$ with size $n-2$ 
in $\Phi_{n-1}^{\infty}$ is equal to $2 \sum_{\gamma} l^*(\gamma)$, where 
$\gamma$ ranges through the faces at infinity of $\Gamma_{\infty}(f)$ such 
that $\d \gamma =2$ and $\relint(\gamma) \subset \Int (\RR^n_+)$. 
In particular, this number is even.
\end{enumerate}
\end{theorem}

\begin{proof}
For a face at infinity $\gamma \prec \Gamma_{\infty}(f)$, denote by $\Pi(\gamma)$ 
the number of the lattice points on the $1$-skeleton of $\gamma$. 
Since for each face at infinity $\gamma \prec \Gamma_{\infty}(f)$ we have $\Pi(\Delta_{\gamma})_1-1=\Pi(\gamma)$ 
(for the definition of $\Pi(\Delta_{\gamma})_1$, see Section \ref{section 5}), 
the assertion (i) follows from Theorem \ref{thm:7-12} (iii) and Proposition \ref{prp:2-19}. 
We can prove also the assertion (ii) by Theorem \ref{thm:7-12} (iii) and Proposition \ref{prp:new}. 
\qed
\end{proof}

From now on, we assume that any face at infinity $\gamma \prec \Gamma_{\infty}(f)$ is prime in 
the sense of Definition \ref{dfn:2-16} (i) 
and rewrite Theorem \ref{thm:7-12} (ii) and (iii) more explicitly. 
First, recall that by Proposition \ref{prp:2-15} for $\lambda \in \CC^* 
\setminus \{1\}$ and a face at infinity $\gamma\prec \Gamma_{\infty}(f)$ 
we have $e^{p,q}(Z_{\Delta_{\gamma}}^*)_{\lambda}=0$ for any $p,q \geq 0$ such that 
$p+q >\d \Delta_{\gamma}-1=\dim \gamma$. So the non-negative integers $r \geq 0$ 
such that $\sum_{p+q=r}e^{p,q}(Z_{\Delta_{\gamma}}^*)_{\lambda}\neq 0$ 
are contained in the closed interval $[0,\d \gamma]\subset \RR$.

\begin{definition}
For a face at infinity $\gamma \prec \Gamma_{\infty}(f)$ and $k \geq 1$, 
we define a finite subset $J_{\gamma,k}\subset [0,\d \gamma] \cap \ZZ$ by
\begin{equation}
J_{\gamma,k}=\{0 \leq r\leq \d \gamma \ |\ n-2+k \equiv r \mod 2\}.
\end{equation}
For each $r\in J_{\gamma,k}$, set
\begin{equation}
d_{k,r}=\dfrac{n-2+k-r}{2}\in \ZZ_+.
\end{equation}
\end{definition}

Since for any face at infinity $\gamma \prec \Gamma_{\infty}(f)$ the polytope 
$\Delta_{\gamma}$ is pseudo-prime in the sense of Definition \ref{dfn:2-16} (ii), 
by Corollary \ref{cor:2-18} for $\lambda \in \CC^* \setminus \{1\}$ and an integer $r \geq 0$ such that $r\in [0,\d \gamma] $ we have
\begin{equation}
\sum_{p+q=r}e^{p,q}(\chi_h([Z_{\Delta_{\gamma}}^*]))_{\lambda}=(-1)^{\d \gamma +r+1} 
\sum_{\begin{subarray}{c} \Gamma\prec \Delta_{\gamma} \\ \d \Gamma=r+1\end{subarray}} 
\left\{ \sum_{\Gamma^{\prime} \prec \Gamma} (-1)^{\d \Gamma^{\prime}} \tl{\varphi}_{\lambda}(\Gamma^{\prime})\right\}.
\end{equation}
For simplicity, we denote this last integer by $e(\gamma,\lambda)_r$. Then by Theorem \ref{thm:7-12} (ii) we obtain the following result.

\begin{theorem}{\rm (Matsui-Takeuchi \cite[Theorem 5.14]{M-T-3})}\label{thm:7-15}
Assume that $f$ is convenient and non-degenerate at infinity and  
any face at infinity $\gamma \prec \Gamma_{\infty}(f)$ is prime. Let  
$\lambda \in \CC^* \setminus\{1\}$ and $k\geq 1$. 
Then the number of the Jordan blocks for the eigenvalue $\lambda$ with sizes $\geq k$ 
in $\Phi_{n-1}^{\infty} \colon H^{n-1}(f^{-1}(R);\CC) \simto H^{n-1}(f^{-1}(R);\CC)$ ($R\gg 0$) is equal to
\begin{equation}
(-1)^{n-1}\sum_{\gamma} \ \left\{ \sum_{r \in J_{\gamma, k}} (-1)^{d_{k,r}} 
\binom{m_{\gamma}}{d_{k,r}} \cdot e(\gamma,\lambda)_r + \sum_{r \in J_{\gamma, k+1}} 
(-1)^{d_{k+1,r}} \binom{m_{\gamma}}{d_{k+1,r}} \cdot e(\gamma,\lambda)_r\right\},
\end{equation}
where in the sum $\sum_{\gamma}$ the face $\gamma$ of $\Gamma_{\infty}(f)$ 
ranges through those at infinity (we used also the convention $\binom{a}{b}=0$ ($0 \leq a <b$) for binomial coefficients).
\end{theorem}

If any face at infinity $\gamma \prec \Gamma_{\infty}(f)$ is prime we can also explicitly 
describe the number of the Jordan blocks for the eigenvalue $1$ in $\Phi_{n-1}^{\infty}$ 
(see Matsui-Takeuchi \cite[Section 5]{M-T-3} for the details). 
The following result is a global analogue of the Steenbrink conjecture 
proved by Varchenko-Khovanskii \cite{K-V} and Saito \cite{Saito-3}. 
Now we return to the general case. 

\begin{definition}{\rm (Sabbah \cite{Sabbah} and Steenbrink-Zucker \cite{S-Z})} 
As a Puiseux series, we define the spectrum at infinity $\sp_f^{\infty}(t)$ of $f$ by
\begin{equation}
\sp_f^{\infty}(t)
= \sum_{\beta \in (0,1] \cap \QQ} \left[ \sum_{i=0}^{n-1} (-1)^{n-1}
\left\{ \sum_{q \geq 0} e^{i,q}([H_f^{\infty}])_{\exp(2\pi \sqrt{-1}\beta)}\right\} t^{i+\beta}\right]+(-1)^{n}t^n.
\end{equation}
\end{definition}

When $f$ is tame at infinity, by Theorem \ref{cor:7-6-2} we can easily 
see that the support of $\sp_f^{\infty}(t)$ is contained in the open interval $(0,n)$ and has the symmetry
\begin{equation}
\sp_f^{\infty}(t)=t^n \sp_f^{\infty}\( \frac{1}{t}\)
\end{equation}
with center at $\frac{n}{2}$. From now on, we assume that $f$ is convenient and non-degenerate 
at infinity. In order to describe $\sp_f^{\infty}(t)$ by $\Gamma_{\infty}(f)$, for each face 
at infinity $\gamma$ of $\Gamma_{\infty}(f)$ let $s_{\gamma}=\sharp S_{\gamma}\in \ZZ_{\geq 1}$ 
be the dimension of the minimal coordinate plane containing $\gamma$ and set $\Cone(\gamma)=\RR_+\gamma$. 
 Next, let $h_f \colon \RR_+^n \longrightarrow \RR$ be the continuous function on $\RR_+^n$ which is 
linear on each cone $\Cone(\gamma)$ and satisfies the condition $h_f|_{\partial \Gamma_{\infty}(f) 
\cap \Int(\RR_+^n)} \equiv 1$. For a face at infinity $\gamma$ of $\Gamma_{\infty}(f)$, let 
$L_{\gamma}$ be the semigroup $\Cone(\gamma) \cap \ZZ_+^n$ and define its Poincar{\'e} series $P_{\gamma}(t)$ by
\begin{equation}
P_{\gamma}(t)=\sum_{\beta \in \QQ_+} \sharp \{ v \in L_{\gamma} \ |\ h_f(v) =\beta\} t^{\beta}.
\end{equation}
Then we have the following result. 

\begin{theorem}{\rm (Matsui-Takeuchi \cite[Theorem 5.16]{M-T-3})}\label{thm:7-19}
Assume that $f$ is convenient and non-degenerate at infinity. Then we have
\begin{equation}
\sp_f^{\infty}(t)=\sum_{\gamma} \ (-1)^{n-1-\d \gamma} (1-t)^{s_{\gamma}}P_{\gamma}(t) +(-1)^n,
\end{equation}
where in the above sum $\gamma$ ranges through the faces at infinity of $\Gamma_{\infty}(f)$.
\end{theorem}

Recently in \cite{Douai} Douai obtained a new proof of Theorem 
\ref{thm:7-19}. For this purpose, he used a global 
Brieskorn lattice instead of motivic Milnor fibers at infinity.

\section{Jordan normal forms of Milnor monodromies}\label{section 8}

 Let $f(x)= \sum_{v \in \ZZ_+^n} a_v x^v \in \CC[x_1,\ldots,x_n]$ be a polynomial on 
$\CC^n$ such that the hypersurface 
$f^{-1}(0)= \{ x \in \CC^n \ |\ f(x)=0 \} \subset \CC^n$ has 
an isolated singular point at the origin $0\in \CC^n$. Then 
by Theorem \ref{thm.7.2.003}, the Milnor fiber $F_0$ of $f$ at 
$0 \in \CC^n$ has the homotopy 
type of bouquet of $(n-1)$-spheres. In particular, we have 
$H^j(F_0;\CC) \simeq 0$ ($j\neq 0, \ n-1$). Let 
\begin{equation}
\Phi_{n-1,0} \colon H^{n-1}(F_0;\CC) \simto H^{n-1}(F_0;\CC)
\end{equation}
\noindent be the $(n-1)$-th Milnor monodromy of $f$ at $0 \in \CC^n$. 
By the theories of A'Campo \cite{A'Campo} and 
Varchenko \cite{Varchenko} etc. explained in Sections \ref{section 2} 
and \ref{section 3}, 
the eigenvalues of $\Phi_{n-1,0}$ were fairly 
well-understood. 
From now, we shall introduce a combinatorial description obtained in 
\cite{M-T-1} of the Jordan normal 
form of $\Phi_{n-1,0}$ by motivic Milnor fibers. 
Let $\pi \colon X \longrightarrow \CC^n$ be an embedded resolution of $f^{-1}(0)$ 
such that $X$ is smooth, the restriction $X \setminus \pi^{-1}( \{ 0 \} ) 
\longrightarrow \CC^n \setminus \{ 0 \}$ of $\pi$ is 
an isomorphism and 
$\pi^{-1}(0)$ and $\pi^{-1}(f^{-1}(0))$ are normal crossing divisors in $X$. 
Let $D_1, D_2, \ldots, D_m$ be the irreducible components of $\pi^{-1}(0)$ 
and denote by $Z$ the proper transform of $f^{-1}(0)$ in $X$. For $1 \leq i \leq m$ 
denote by $a_i>0$ the order of the zero of $g:= f \circ \pi$ along $D_i$. 
For a non-empty subset $I \subset \{ 1,2, \ldots, m\}$ we set $d_I=
{\rm gcd} (a_i)_{i \in I}>0$, $D_I=\bigcap_{i \in I}D_i$ and
\begin{equation}
D_I^{\circ}=D_I \setminus \left\{ \( \bigcup_{i \notin I}D_i\) \cup Z \right\} \subset X.
\end{equation}
Moreover we set
\begin{equation}
Z_I^{\circ}=\left\{ D_I \setminus \left( \bigcup_{i \notin I}D_i\right) \right\} \cap Z \subset X. 
\end{equation}
Then, as in Section \ref{section 6}, 
we can construct an unramified Galois 
covering $\tl{D_I^{\circ}} \longrightarrow D_I^{\circ}$ of $D_I^{\circ}$. 
Moreover the variety $\tl{D_I^{\circ}}$ is endowed with a good $\hat{\mu}$-action 
in the sense of Denef-Loeser \cite[Section 2.4]{D-L-2}. 
Note that also the variety $Z_I^{\circ}$ is equipped with the trivial good 
$\hat{\mu}$-action. 

\begin{definition} {\rm (Denef and Loeser \cite{D-L-1} 
and \cite{D-L-2})}\label{dfn:3-1-1} We define the motivic Milnor fiber 
$\SS_{f,0} \in \M_{\CC}^{\hat{\mu}}$ of $f$ at $0 \in \CC^n$ by
\begin{equation}\label{MMF-new}
\SS_{f,0} =\sum_{I \neq \emptyset}\left\{ (1-\LL)^{\sharp I -1} 
[\tl{D_I^{\circ}}] + (1-\LL)^{\sharp I} [Z_I^{\circ}]\right\} 
\quad  \in \M_{\CC}^{\hat{\mu}}.
\end{equation}
\end{definition}
By using the theory of motivic zeta functions, we can show that this 
definition of $\SS_{f,0} \in \M_{\CC}^{\hat{\mu}}$ does 
not depend on the choice of the  embedded resolution 
$\pi \colon X \longrightarrow \CC^n$ of $f^{-1}(0)$. 
See \cite{D-L-2} for the details.  
On the other hand, as in Section \ref{section 7}, 
to the cohomology groups $H^j(F_0;\CC)$ 
and the semisimple parts of their monodromy automorphisms, 
we can naturally associate an element
\begin{equation}
[H_f] \in \KK_0(\HSm). 
\end{equation}
To describe the element $[H_f]\in \KK_0(\HSm)$ in terms of $\SS_{f,0} \in \M_{\CC}^{\hat{\mu}}$, let
\begin{equation}
\chi_h \colon \M_{\CC}^{\hat{\mu}} \longrightarrow \KK_0(\HSm)
\end{equation}
be the Hodge characteristic morphism. Then we have the following fundamental result. 

\begin{theorem}\label{thm:7-6-new}
{\rm (Denef-Loeser \cite[Theorem 4.2.1]{D-L-1})} In 
the Grothendieck group $\KK_0(\HSm)$, we have
\begin{equation}
[H_f]=\chi_h(\SS_{f,0}).
\end{equation}
\end{theorem}

For $[H_f] \in \KK_0(\HSm)$, the following result due 
to Steenbrink \cite{Steenbrink} and Saito \cite{Saito-1}, 
\cite{Saito-2} is fundamental. Recall that 
the weights of $[H_f] \in \KK_0(\HSm)$ are defined by 
the monodromy filtrations. For an excellent review  
on this subject, see Kulikov \cite{Kulikov}. 

\begin{theorem} 
{\rm (Steenbrink \cite{Steenbrink} and Saito \cite{Saito-1}, \cite{Saito-2})} 
\label{S-S}
In the situation as above, we have
\begin{enumerate}
\item Let $\lambda \in \CC^* \setminus \{1\}$. 
Then we have $e^{p,q}( [H_f])_{\lambda}=0$ for $(p,q) \notin [0,n-1] 
\times [0,n-1]$. Moreover for $(p,q) \in [0,n-1] \times [0,n-1]$ we have
\begin{equation}
e^{p,q}( [H_f])_{\lambda}=e^{n-1-q,n-1-p}( [H_f])_{\lambda}.
\end{equation}
\item We have $e^{p,q}( [H_f])_{1}=0$ for $(p,q) \notin \{(0, 0)\} 
\sqcup ([1,n-1] \times [1,n-1])$ and $e^{0,0}( [H_f])_{1}=1$. 
Moreover for $(p,q) \in [1,n-1] \times [1,n-1]$ we have
\begin{equation}
e^{p,q}( [H_f])_{1}=e^{n-q,n-p}( [H_f])_{1}.
\end{equation}
\end{enumerate}
\end{theorem}

Together with Theorem \ref{thm:7-6-new}, we obtain the following result.

\begin{theorem}\label{MF}
In the situation as above, we have
\begin{enumerate}
\item Let $\lambda \in \CC^* \setminus \{1\}$ and $k \geq 1$. 
Then the number of the Jordan blocks for the eigenvalue $\lambda$ 
with sizes $\geq k$ in $\Phi_{n-1,0}\colon H^{n-1}(F_0;\CC) \simto H^{n-1}
(F_0;\CC)$ is equal to
\begin{equation}
(-1)^{n-1} \sum_{p+q=n-2+k, n-1+k} e^{p,q}( \chi_h(\SS_{f,0} ))_{\lambda}.
\end{equation}
\item For $k \geq 1$, the number of the Jordan blocks for the 
eigenvalue $1$ with sizes $\geq k$ in $\Phi_{n-1, 0}$ is equal to
\begin{equation}
(-1)^{n-1} \sum_{p+q=n-1+k, n+k} e^{p,q}( \chi_h(\SS_{f,0} ))_{1}.
\end{equation}
\end{enumerate}
\end{theorem}

From now, we shall rewrite Theorem \ref{MF} 
in terms of the Newton polyhedron 
$\Gamma_+(f)$ of $f$. As usual, we call the union of 
the compact faces of $\Gamma_+(f)$ the 
Newton boundary of $f$ and denote it by $\Gamma_f$. 
Recall that generic polynomials having a fixed Newton polyhedron 
are non-degenerate at $0\in \CC^n$. From now on, we always assume 
also that $f=\sum_{v \in \ZZ^n_+} a_v x^v\in \CC[x_1,\ldots,x_n]$ 
is convenient and non-degenerate at $0\in \CC^n$. For each face 
$\gamma \prec \Gamma_+(f)$ such that $\gamma \subset \Gamma_f$, 
let $d_{\gamma}>0$ be the lattice distance of $\gamma$ from the 
origin $0 \in \RR^n$ and $\Delta_{\gamma}$ the convex hull of $\{0\} 
\sqcup \gamma$ in $\RR^n$. Let $\LL(\Delta_{\gamma})$ be the $(\dim \gamma +1)$-dimensional 
linear subspace of $\RR^n$ spanned by $\Delta_{\gamma}$ and consider 
the lattice $M_{\gamma}=\ZZ^n \cap \LL(\Delta_{\gamma}) \simeq \ZZ^{\dim \gamma+1}$ 
in it. Then we set $T_{\Delta_{\gamma}}:=\Spec (\CC[M_{\gamma}]) \simeq 
(\CC^*)^{\dim \gamma +1}$. Moreover let $\LL(\gamma)$ be the 
smallest affine linear subspace of $\RR^n$ containing $\gamma$ 
and for $v \in M_{\gamma}$ define their lattice heights $\height (v, \gamma) 
\in \ZZ$ from $\LL(\gamma)$ in $\LL(\Delta_{\gamma})$ so that 
we have $\height (0, \gamma)=d_{\gamma}>0$. 
Then to the group homomorphism $M_{\gamma} \longrightarrow \CC^*$ defined by 
$v \longmapsto \zeta_{d_{\gamma}}^{-\height (v, \gamma)}$ we can naturally associate an 
element $\tau_{\gamma} \in T_{\Delta_{\gamma}}$. We define a Laurent 
polynomial $g_{\gamma}=\sum_{v \in M_{\gamma}}b_v x^v$ on $T_{\Delta_{\gamma}}$ by
\begin{equation}
b_v=\begin{cases}
a_v & (v \in \gamma),\\
 & \\
-1 & (v=0),\\
 & \\
\ 0 & (\text{otherwise}).
\end{cases}
\end{equation}
Then we have $NP(g_{\gamma}) =\Delta_{\gamma}$, $\supp (g_{\gamma}) \subset 
\{ 0\} \sqcup \gamma$ and the hypersurface $Z_{\Delta_{\gamma}}^*=\{ x \in 
T_{\Delta_{\gamma}}\ |\ g_{\gamma}(x)=0\}$ is non-degenerate (see the proof of 
\cite[Proposition 5.3]{M-T-3}). Moreover $Z_{\Delta_{\gamma}}^* \subset T_{\Delta_{\gamma}}$ 
is invariant by the multiplication $l_{\tau_{\gamma}} \colon 
T_{\Delta_{\gamma}} \simto T_{\Delta_{\gamma}}$ by $\tau_{\gamma}$, and hence we 
obtain an element $[Z_{\Delta_{\gamma}}^*]$ of $\M_{\CC}^{\hat{\mu}}$. Let $\LL(\gamma)^{\prime} 
\simeq \RR^{\d \gamma}$ be a linear subspace of $\RR^n$ such that 
$\LL(\gamma)=\LL(\gamma)^{\prime}+w $ for some $w \in \ZZ^n$ 
and set $\gamma^{\prime}=\gamma -w \subset \LL(\gamma)^{\prime}$. 
We define a Laurent polynomial $g_{\gamma}^{\prime} 
=\sum_{v \in \LL(\gamma)^{\prime} \cap \ZZ^n}b_v^{\prime} x^v$ 
on $T(\gamma ):= \Spec (\CC[\LL(\gamma)^{\prime} \cap \ZZ^n])\simeq (\CC^*)^{\d \gamma}$ by
\begin{equation}
b_v^{\prime}=\begin{cases}
a_{v+w} & (v \in \gamma^{\prime} ),\\
 & \\
\ 0 & (\text{otherwise}).
\end{cases}
\end{equation}
Then we have $NP(g_{\gamma}^{\prime}) =\gamma^{\prime}$ and the 
hypersurface $Z_{\gamma}^*=\{ x \in T(\gamma ) \ |\ g_{\gamma}^{\prime}(x)=0\}$ is non-degenerate. 
We define $[Z_{\gamma}^*] \in \M_{\CC}^{\hat{\mu}}$ to be the class of the variety $Z_{\gamma}^*$ with 
the trivial action of $\hat{\mu}$. Finally let $S_{\gamma} \subset \{1,2,\ldots, n\}$ be the 
minimal subset of $\{1,2,\ldots,n\}$ such that 
$\gamma \subset \RR^{S_{\gamma}}$ and set $m_{\gamma}:=
\sharp S_{\gamma}-\dim \gamma -1\geq 0$. Then 
we have the following result. 

\begin{theorem}{\rm (Matsui-Takeuchi \cite[Theorem 4.3]{M-T-1})}\label{thm:8-3}
Assume that $f$ is convenient and non-degenerate at $0\in \CC^n$. Then 
we have
\begin{enumerate}
\item In the Grothendieck group $\KK_0(\HSm)$, we have
\begin{equation}\label{twot}
\chi_h(\SS_{f,0})= \sum_{\gamma \subset \Gamma_f} 
\chi_h\big((1-\LL)^{m_{\gamma}}\cdot[Z_{\Delta_{\gamma}}^*]\big)
+\sum_{\begin{subarray}{c}\gamma \subset \Gamma_f\\ \d \gamma \geq 1\end{subarray}} 
\chi_h\big((1-\LL)^{m_{\gamma}+1}\cdot[Z_{\gamma}^*]\big).
\end{equation}
\item Let $\lambda \in \CC^*\setminus\{1\}$ and $k\geq 1$. Then the number of 
the Jordan blocks for the eigenvalue $\lambda$ with sizes $\geq k$ in 
$\Phi_{n-1,0}\colon H^{n-1}(F_0;\CC) \simto H^{n-1}(F_0;\CC)$ is equal to
\begin{equation}
(-1)^{n-1}\sum_{p+q=n-2+k, n-1+k}\left\{ \sum_{\gamma \subset 
\Gamma_f} e^{p,q}\( \chi_h\((1-\LL)^{m_{\gamma}} \cdot [Z_{\Delta_{\gamma}}^*]\)\)_{\lambda} \right\}.
\end{equation}
\item For $k\geq 1$, the number of the Jordan blocks for the 
eigenvalue $1$ with sizes $\geq k$ in $\Phi_{n-1,0}$ is equal to
\begin{eqnarray}
(-1)^{n-1}\sum_{p+q=n-1+k, n+k}\lefteqn{\Bigg\{ \sum_{\gamma \subset \Gamma_f} 
e^{p,q}\big( \chi_h\big((1-\LL)^{m_{\gamma}} \cdot [Z_{\Delta_{\gamma}}^*]\big)\big)_{1}} \nonumber \\
& & +\sum_{\begin{subarray}{c}\gamma \subset \Gamma_f\\ \d \gamma 
\geq 1\end{subarray}} e^{p,q}\big( \chi_h\big((1-\LL)^{m_{\gamma}+1} \cdot [Z_{\gamma}^*]\big)\big)_{1} \Bigg\}.
\end{eqnarray}
\end{enumerate}
\end{theorem}

Let $q_1,\ldots, q_l$ (resp. $\gamma_1, \ldots, \gamma_{l^{\prime}}$) be the $0$-dimensional 
(resp. $1$-dimensional) faces of $\Gamma_+(f)$ such that $q_i \in \Int (\RR_+^n)$ 
(resp. $\relint(\gamma_i) \subset \Int(\RR_+^n)$). 
For each $q_i$ (resp. $\gamma_i$), denote by $d_i >0$ (resp. $e_i>0$) 
the lattice distance $\dist(q_i, 0)$ (resp. $\dist(\gamma_i,0)$) of it from the origin 
$0\in \RR^n$. For $1\leq i \leq l^{\prime}$, let $\Delta_i$ be the convex hull of 
$\{0\}\sqcup \gamma_i$ in $\RR^n$. Then for $\lambda \in \CC \setminus \{1\}$ and 
$1 \leq i \leq l^{\prime}$ such that $\lambda^{e_i}=1$ we set
\begin{align}
n(\lambda)_i
 := & 
\sharp\{ v\in \ZZ^n \cap \relint(\Delta_i) \ |\ \height (v, \gamma_i)=k\} 
\nonumber 
\\
 & +  \sharp \{ v\in \ZZ^n \cap \relint(\Delta_i) \ |\ \height (v, \gamma_i)=e_i-k\},
\end{align}
where $k$ is the minimal positive integer satisfying $\lambda=\zeta_{e_i}^{k}$ and for 
$v\in \ZZ^n \cap \relint(\Delta_i)$ we denote by $\height (v, \gamma_i)$ 
the lattice height of 
$v$ from the base $\gamma_i$ of $\Delta_i$. 
As in Theorem \ref{thm:7-11}, we obtain the following theorem. 

\begin{theorem}{\rm (Matsui-Takeuchi \cite[Theorem 4.4]{M-T-1})}\label{thm:8-4}
Assume that $f$ is convenient and non-degenerate at $0\in \CC^n$. Then 
for $\lambda \in \CC^* \setminus\{1\}$, we have
\begin{enumerate}
\item The number of the Jordan blocks for the eigenvalue $\lambda$ with 
the maximal possible size $n$ in $\Phi_{n-1,0}$ is equal to $\sharp \{q_i \ |\ \lambda^{d_i}=1\}$.
\item The number of the Jordan blocks for the eigenvalue $\lambda$ with 
size $n-1$ in $\Phi_{n-1,0}$ is equal to $\sum_{i \colon \lambda^{e_i}=1} n(\lambda)_i$.
\end{enumerate}
\end{theorem}

As in Theorem \ref{thm:7-15}, we  
can rewrite Theorem \ref{thm:8-3} (ii) more explicitly in the case where any face 
$\gamma \prec \Gamma_{+}(f)$ such that $\gamma \subset \Gamma_f$ is prime. 
See \cite[Theorem 4.6]{M-T-1} for the details. 
We can also obtain the corresponding results for the eigenvalue $1$ by 
rewriting Theorem \ref{thm:8-3} (iii) more simply as follows.

\begin{theorem}{\rm (Matsui-Takeuchi \cite[Theorem 4.8]{M-T-1})}\label{thm:8}
In the situation of Theorem \ref{thm:8-3}, for $k\geq 1$ the number of the 
Jordan blocks for the eigenvalue $1$ with sizes $\geq k$ in $\Phi_{n-1,0}$ is equal to
\begin{equation}
(-1)^{n-1}\sum_{p+q=n-2-k, n-1-k}\left\{ \sum_{\gamma \subset \Gamma_f} 
e^{p,q}\( \chi_h\((1-\LL)^{m_{\gamma}} \cdot [Z_{\Delta_{\gamma}}^*]\)\)_{1} \right\}.
\end{equation}
\end{theorem}

Indeed Theorem \ref{thm:8} is a direct 
consequence of the following more precise result.

\begin{theorem}{\rm (Matsui-Takeuchi \cite[Theorem 4.10]{M-T-1})}\label{thm:7-6-1}
In the situation of Theorem \ref{thm:8-3}, for any $0 \leq p,q \leq n-2$ we have
\begin{eqnarray}
\lefteqn{\sum_{\gamma \subset \Gamma_f}e^{p,q}\( \chi_h\( (1-\LL)^{m_{\gamma}}[Z_{\Delta_{\gamma}}^*]\)\)_1}\nonumber \\
&=&\sum_{\gamma \subset \Gamma_f}e^{p+1,q+1}
\( \chi_h\( (1-\LL)^{m_{\gamma}}[Z_{\Delta_{\gamma}}^*]+ (1-\LL)^{m_{\gamma}+1}[Z_{\gamma}^*]\)\)_1.
\label{eq:7-6-1}
\end{eqnarray}
\end{theorem}

As in Theorem \ref{thm:7-16}, we obtain the following result. 
Denote by $\Pi_f$ the number of the lattice points on the $1$-skeleton of $\Gamma_f \cap \Int(\RR_+^n)$. 
Also, for a compact face $\gamma \prec \Gamma_+(f)$ we denote by $l^*(\gamma)$ the number of the lattice points on $\relint(\gamma)$.

\begin{theorem}{\rm (Matsui-Takeuchi \cite[Corollary 4.9]{M-T-1})}\label{thm:8-5}
Assume that $f$ is convenient and non-degenerate at $0\in \CC^n$. Then we have
\begin{enumerate}
\item {\rm (van Doorn-Steenbrink \cite{D-St})} The number of the Jordan blocks for the eigenvalue 
$1$ with the maximal possible size $n-1$ in $\Phi_{n-1,0}$ is $\Pi_f$.
\item The number of the Jordan blocks for the eigenvalue $1$ with size $n-2$ in $\Phi_{n-1,0}$ 
is equal to $2\sum_{\gamma} l^*(\gamma)$, where $\gamma$ ranges through the compact faces of 
$\Gamma_+(f)$ such that $\d \gamma=2$ and $\relint(\gamma) \subset \Int(\RR_+^n)$.
\end{enumerate}
\end{theorem}

Note that Theorem \ref{thm:8-5} (i) was previously obtained in van Doorn-Steenbrink \cite{D-St} by different methods. 
These results assert that by replacing $\Gamma_+(f)$ with the Newton 
polyhedron at infinity $\Gamma_{\infty}(f)$ 
the combinatorial description of the local monodromy $\Phi_{n-1,0}$ is the same as 
that of the global one $\Phi_{n-1}^{\infty}$ introduced in Section 
\ref{section 7}. We thus see a beautiful symmetry between local and global. 
In Esterov-Takeuchi \cite{E-T}, 
some of the results in this section were extended 
to the case where $f$ is a regular function on 
a complete intersection subvariety of $\CC^n$ 
as in Theorem \ref{thm:3-5-CI}. For a related result, 
see also Ebeling-Steenbrink \cite{E-S}. 
Moreover in \cite{Saito}, T. Saito extended some of the results in 
this section to the case where $f$ is not necessarily convenient. 
For this purpose, he proved a new vanishing theorem 
and overcame the difficulty arising from the 
non-isolated singular points in $f^{-1}(0) \subset \CC^n$. 
Along the same line as in the proof of Theorem 
\ref{thm:7-19} in \cite{M-T-3}, in \cite{Saito} 
he obtained also a formula for the Hodge spectrum of 
such $f$ at the origin $0 \in \CC^n$ and 
extended the results of \cite{K-V} and \cite{Saito-3} 
to the non-convenient case.  For related results on 
the Hodge spectrum, see also Budur \cite{Budur}, 
Budur-Saito \cite{B-S} and Jung-Kim-Saito-Yoon \cite{J-K-S-Y}.

\section{Theory of Katz and Stapledon}\label{section 9}

In this section, assuming 
that the polynomial $f(x) \in \CC[x_1,\ldots,x_n]$ 
is convenient and non-degenerate at $0\in \CC^n$, we 
introduce the full combinatorial description of the Jordan 
normal form of its Milnor monodromy $\Phi_{n-1, 0}$ for 
the eigenvalues $\lambda \not= 1$ obtained by 
Stapledon \cite{Stapledon}. For its generalization to 
the non-convenient case, see also Saito \cite{Saito}. 
In \cite{Stapledon}, Stapledon also extended our setting in 
Section \ref{section 7} to more general one of sch\"on families of complex 
hypersurfaces in $\CC^n$ over a punctured 
disk $\{ t \in \CC \ | \ 0<|t| < \varepsilon \} \subset \CC$ 
and obtained a formula for their equivariant limit 
Hodge numbers. This is a beautiful application of 
the ideas in intersection cohomology, tropical geometry  
and mirror symmetry. For the further developments of Stapledon's 
theory and its applications to geometric monodromies, 
we refer also to Saito-Takeuchi \cite{S-T}.

\subsection{Equivariant Ehrhart theory of 
Katz-Stapledon}\label{sbbsec:7}

First we recall some 
polynomials in the Equivariant Ehrhart theory 
of Katz-Stapledon \cite{Ka-St-1}, 
\cite{Ka-St-2} and Stapledon~\cite{Stapledon}. 
Here we regard the empty set 
$\emptyset$ as a $(-1)$-dimensional polytope,
and as a face of any polytope.
Let $P$ be a polytope.
If a subset $F\subset P$ is a face of $P$, we write $F\prec P$.
For a pair of faces $F\prec F' \prec P$ of $P$,
we denote by $[F,F']$ the face poset $\{F''\prec 
P\mid F\prec F''\prec F'\}$,
and by $[F,F']^{*}$ a poset which is equal to 
$[F,F']$ as a set with the reversed order.

\begin{definition}
Let $B$ be a poset $[F,F']$ or $[F,F']^{*}$.
We define a polynomial $g(B,t)$ of degree 
$\leq(\dim F' -\dim F)/2$ as follows.
If $F = F'$, we set $g(B;t)=1$.
If $F \neq F'$ and $B=[F,F']$ (resp. $B=[F,F']^{*}$), 
we define $g(B;t)$ inductively by
\begin{align}
t^{\dim{F'}-\dim{F}}g(B;t^{-1})=\sum_{F''\in[F,F']}
(t-1)^{\dim{F'}-\dim{F''}}g([F,F''];t).
\\ ({\rm resp.}~t^{\dim{F'}-\dim{F}}g(B;t^{-1})=
\sum_{F''\in[F,F']^{*}}(t-1)^{\dim{F''}-\dim{F}}g([F'',F']^{*};t).) 
\end{align}
\end{definition}

In what follows, we assume that $P$ is a lattice polytope in $\RR^n$.
Let $S$ be a subset of $P\cap \ZZ^n$ containing the 
vertices of $P$, and $\omega \colon S\to \ZZ$ be a function.
We denote by ${\rm UH}_{\omega}$ the convex hull in $\RR^n\times \RR$ 
of the set $\{(v,s)\in \RR^n\times\RR \mid v\in S, s\geq \omega(v)\}$.
Then, the set of all the projections of the bounded faces of ${\rm UH}_{\omega}$ 
to $\RR^n$ defines a lattice polyhedral subdivision $\mathcal{S}$ of $P$.
Here a lattice polyhedral subdivision $\mathcal{S}$ 
of a polytope $P$ is a set of some polytopes in $P$
such that the intersection of any two polytopes in $\mathcal{S}$ is a 
face of both and all vertices of any polytope 
in $\mathcal{S}$ are in $\ZZ^{n}$. 
Moreover, the set of all the bounded faces of ${\rm UH}_{\omega}$ defines 
a piecewise $\QQ$-affine convex function $\nu\colon P\to \RR$.
For a cell $F\in\mathcal{S}$, we denote by 
$\sigma(F)$ the smallest face of $P$ containing $F$,
and $\LK_{\mathcal{S}}(F)$ the set of all cells 
of $\mathcal{S}$ containing $F$.
We call $\LK_{\mathcal{S}}(F)$ the link of $F$ in $\mathcal{S}$.
Note that $\sigma(\emptyset)=\emptyset$ and 
$\LK_{\mathcal{S}}(\emptyset)=\mathcal{S}$.

\begin{definition}
For a cell $F\in \mathcal{S}$, the $h$-polynomial 
$h(\LK_{\mathcal{S}}(F);t)$ 
of the link $\LK_{\mathcal{S}}(F)$ of $F$ is defined by
\begin{equation}
t^{\dim{P}-\dim{F}}h(\LK_{\mathcal{S}}(F);t^{-1})=
\sum_{F'\in \LK_{\mathcal{S}}(F)}g([F,F'];t)(t-1)^{\dim{P}-\dim{F'}}.
\end{equation}
The local $h$-polynomial $l_{P}(\mathcal{S},F;t)$ 
of $F$ in $\mathcal{S}$ is defined by
\begin{equation}
l_{P}(\mathcal{S},F;t)=\sum_{\sigma(F)
\prec Q\prec P}(-1)^{\dim{P}-\dim Q}h(
\LK_{{\mathcal{S}}|_{Q}}(F);t) \cdot g([Q,P]^{*};t).
\end{equation}
\end{definition}

For $\lambda\in \CC$ and 
$v\in mP\cap{\ZZ^n}$ ($m\in \ZZ_+:=\ZZ_{\geq 0}$) we set 
\begin{align}
w_{\lambda}(v)=
\left\{
\begin{array}{ll}
1& \Bigl( \exp\ \bigl( 2\pi\sqrt{-1}\cdot m\nu(\frac{v}{m})\bigr)=
\lambda \Bigr) \\\\
0&( \text{otherwise} ). \\
\end{array}
\right.
\end{align}
We define the $\lambda$-weighted Ehrhart polynomial 
$\phi_{\lambda}(P,\nu;m)\in\ZZ[m]$ 
of $P$ with respect to $\nu:P\to \RR$ by
\begin{equation}
\phi_{\lambda}(P,\nu;m):=\sum_{v\in {mP}\cap\ZZ^{n}}w_{\lambda}(v).
\end{equation}
Then $\phi_{\lambda}(P,\nu;m)$ is a polynomial in $m$ 
with coefficients $\ZZ$ 
whose degree is $\leq\dim P$ (see \cite{Stapledon}).  

\begin{definition}{\rm (Stapledon \cite{Stapledon})}
\begin{enumerate}
\item[\rm{(i)}] We define the $\lambda$-weighted $h^{*}$-polynomial 
$h^{*}_{\lambda}(P,\nu;u)
\in \ZZ[u]$ by 
\begin{equation}
\sum_{m\geq 0} \phi_{\lambda}(P,\nu;m)u^{m}=
\frac{h^{*}_{\lambda}(P,\nu;u)}{(1-u)^{\dim{P}+1}}.
\end{equation}
If $P$ is the empty polytope, we set $h^{*}_{1}
(P,\nu;u)=1$ and $h^{*}_{\lambda}(P,\nu;u)=0~(\lambda\neq 1)$.
\item[\rm{(ii)}] We define the $\lambda$-local weighted 
$h^{*}$-polynomial $l^{*}_{\lambda}(P,\nu;u)\in\ZZ[u]$ by
\begin{equation}
l^{*}_{\lambda}(P,\nu;u)=\sum_{Q\prec P}
(-1)^{\dim{P}-\dim{Q}}h^{*}_{\lambda}(Q,\nu|_{Q};u)
 \cdot g([Q,P]^{*};u).
\end{equation}
If $P$ is the empty polytope, we set $l^{*}_{1}(P,\nu;u)=1$ 
and $l^{*}_{\lambda}(P,\nu;u)=0~(\lambda\neq 1)$.
\end{enumerate}
\end{definition} 

\begin{definition}{\rm (Stapledon \cite{Stapledon})}\label{def:poly}
\begin{enumerate}
\item[\rm{(i)}] We define the $\lambda$-weighted limit mixed 
$h^{*}$-polynomial $h^{*}_{\lambda}(P,\nu;u,v)\in\ZZ[u,v]$ by
\begin{equation}
h^{*}_{\lambda}(P,\nu;u,v):=\sum_{F\in\MCS}v^{\dim{F}+1}
l^{*}_{\lambda}(F,\nu|_{F};uv^{-1}) \cdot h(\LK_{\MCS}(F);uv).
\end{equation}
\item[\rm{(ii)}] We define the $\lambda$-local weighted limit mixed 
$h^{*}$-polynomial $l^{*}_{\lambda}(P,\nu;u,v)\in\ZZ[u,v]$ by
\begin{equation}
l^{*}_{\lambda}(P,\nu;u,v):=\sum_{F\in\MCS}v^{\dim{F}+1}
l^{*}_{\lambda}(F,\nu|_{F};uv^{-1}) \cdot l_{P}(\MCS,F;uv).
\end{equation}
\end{enumerate}
\end{definition}

\subsection{Stapledon's full combinatorial description of the Jordan 
normal forms of Milnor monodromies}\label{sbbsec:8}

Assume that $f(x) \in \CC[x_1,\ldots,x_n]$ 
is convenient and non-degenerate at $0\in \CC^n$. 
Recall that for a compact face $\gamma \prec 
\Gamma_+(f)$ of its Newton polyhedron $\Gamma_+(f)$ we denote by 
$\Delta_{\gamma}$ the convex hull of $\{ 0 \} \sqcup \gamma$ in $\RR^n$. 
Denote by $K$ the convex hull 
of the closure of 
$\RR^n_+ \setminus \Gamma_+(f)$ in $\RR^n$ and
define a piecewise $\QQ$-affine function 
$\nu$ on $K$ which takes the value $1$ (resp. $0$)
on $\{ 0 \} \subset \RR^n$ 
(resp. on the convex hull of the Newton boundary 
$\Gamma_f \subset \RR^n$ of $f$) such that 
for any compact face $\gamma$ of 
$\Gamma_+(f)$ 
the restriction of $\nu$ to 
$\Delta_{\gamma}$ is an affine function. 
For $\lambda \in \CC$ we define 
the equivariant Hodge-Deligne polynomial for the eigenvalue $\lambda$ 
(of the mixed Hodge structures of the cohomology groups of the Milnor 
fiber $F_{0}$) $E_{\lambda}(F_{0};u,v) \in \ZZ [u,v]$ by 
\begin{equation}
E_{\lambda}(F_{0};u,v)= 
\sum_{p,q\in\ZZ}
\sum_{j\in\ZZ}(-1)^{j}h^{p,q}_{\lambda}(H^{j}(F_{0};\CC))
u^pv^q\in\ZZ[u,v],
\end{equation}
where $h^{p,q}_{\lambda}(H^{j}(F_{0};\CC))$ is the dimension of 
$\GR^{p}_{F}\GR^{W}_{p+q}H^{j}(F_{0};\CC)_{\lambda}$. 
Then for $\lambda \not= 1$, by applying the results of 
\cite{Stapledon} to Theorem \ref{thm:8-3} we can 
calculate the $\lambda$-part of  
the Hodge realization 
of the motivic Milnor fiber 
$\mathcal{S}_{f,0}$ of $f$ and 
obtain the following formula for 
$E_{\lambda}(F_{0};u,v)$. 

\begin{theorem}{\rm (Stapledon \cite[Theorem 6.20]{Stapledon})}\label{the-3214} 
Assume that $f$ is convenient and non-degenerate at $0\in \CC^n$. Then 
for any $\lambda \not= 1$ we have 
\begin{equation}
uv E_{\lambda}(F_{0};u,v)
=(-1)^{n-1}l^{*}_{\lambda}(K,\nu;u,v). 
\end{equation}
\end{theorem}

Let $\mathcal{S}_{\nu}$ be the polyhedral 
subdivision of the polytope $K$ defined by $\nu$. 
By the definition of the $h^*$-polynomial, 
for $\lambda \not= 1$ we have
\begin{equation}
l^{*}_{\lambda}(K,\nu;u,u)=
\sum_{\gamma \prec \Gamma_+(f) :{\rm compact}}
u^{\dim \Delta_{\gamma} +1}
l^{*}_{\lambda}( \Delta_{\gamma} ,\nu;1)
\cdot l_{K}( \mathcal{S}_{\nu}, \Delta_{\gamma} ;u^2),
\end{equation}
where in the sum $\Sigma$ the face $\gamma$ 
ranges through the compact ones of $\Gamma_+(f)$.
The polynomial $l_{K}( \mathcal{S}_{\nu}, \Delta_{\gamma} ;t)$ 
is symmetric and unimodal centered at 
$(n-\dim \gamma -1)/2$, i.e. if $a_{i}\in \ZZ$ 
is the coefficient of $t^i$ in $l_{K}
( \mathcal{S}_{\nu}, \Delta_{\gamma} ;t)$ we have 
$a_{i}=a_{n-\dim \gamma -1-i}$ and $a_{i}\leq a_{j}$ 
for $0\leq i\leq j\leq (n-\dim \gamma -1)/2$.
Therefore, it can be expressed in the form
\begin{equation}
l_{K}( \mathcal{S}_{\nu}, \Delta_{\gamma} ;t)=
\sum_{i=0}^{\lfloor (n-1-\dim \gamma )/2\rfloor}
\tl{l}_{\gamma ,i} \cdot (t^i+t^{i+1}+\dots+t^{n-1-\dim \gamma -i}),
\end{equation}
for some non-negative integers $\tl{l}_{{\gamma},i}\in\ZZ_{\geq 0}$.
We set
\begin{equation}
\tl{l}_{K}( \mathcal{S}_{\nu}, \Delta_{\gamma} ,t):=
\sum_{i=0}^{\lfloor (n-1-\dim \gamma )/2\rfloor}
\tl{l}_{\gamma,i} \cdot t^i.
\end{equation}
For $k\in\ZZ_{> 0}$ and $\lambda \in\CC$ 
we denote by $J_{k,\lambda}$ the number of 
the Jordan blocks in $\Phi_{n-1, 0}$ with size $k$ 
for the eigenvalue $\lambda$.
Then we obtain the following formula for them.

\begin{theorem}{\rm (Stapledon \cite[Corollary 6.22]{Stapledon})}\label{jordan}
Assume that $f$ is convenient and non-degenerate at $0\in \CC^n$. Then 
for any $\lambda \not= 1$ we have
\begin{equation}
\sum_{0\leq k\leq n-1}J_{n-k,\lambda}u^{k+2}
=
\sum_{\gamma \prec \Gamma_+(f) :{\rm compact}}
u^{\dim \Delta_{\gamma} +1}l^{*}_{\lambda}
(\Delta_{\gamma},\nu;1)\cdot\tl{l}_{K}( \mathcal{S}_{\nu},
\Delta_{\gamma} ;u^2),
\end{equation}
where in the sum $\Sigma$ of the right 
hand side the face $\gamma$ ranges through 
the compact ones of $\Gamma_+(f)$.
\end{theorem}

We can check that Theorem \ref{thm:8-4} follows from 
Theorem \ref{jordan}, but we omit the details here.

\section{Milnor fibers and 
monodromies of meromorphic functions}\label{section 10}

In \cite{GZ-L-MH} Gusein-Zade, Luengo and Melle-Hern\'andez 
generalized Milnor's fibration theorem 
to meromorphic functions 
and defined their Milnor fibers. Moreover they obtained 
a formula for their monodromy zeta functions. 
Since then many mathematicians studied Milnor fibers of 
meromorphic functions (see e.g. \cite{B-P}, 
\cite{B-P-S}, \cite{GV-L-M}, \cite{GZ-L-MH-new-new}, 
\cite{L-W}, \cite{Thang}, \cite{N-S-T}, 
\cite{N-T}, \cite{P}, 
\cite{Raibaut-2}, \cite{S-T-2}, 
\cite{Tibar}). However, in contrast to 
Milnor fibers of holomorphic functions, 
the geometric structures of those 
of meromorphic functions are not fully understood yet. 
For example, as we saw in Theorem \ref{thm.7.2.003}, 
if a holomorphic function 
has an isolated singular point then the Milnor 
fiber at it has the homotopy type of a bouquet of 
some spheres. This implies that its reduced 
cohomology groups are concentrated in the 
middle dimension. Unfortunately, so far 
we do not know such a nice structure theorem for 
Milnor fibers of meromorphic functions. 
For this reason, we cannot know any property 
of each Milnor monodromy operator of 
a meromorphic function even if we 
have the formula of \cite{GZ-L-MH} for its monodromy zeta function. 
Recently in \cite{N-T}, 
by developing a theory of 
nearby and vanishing cycle functors for 
meromorphic functions, Nguyen and the author  
overcame this problem partially. In this section, 
we introduce some of the above-mentioned results for Milnor fibers 
and monodromies of meromorphic functions and explain 
some related problems. 

 Let $X$ be a complex manifold and $P(x), Q(x)$ 
holomorphic functions on it. Assume that $Q(x)$ 
is not identically zero on each connected component 
of $X$. Then we define a meromorphic function $f(x)$ on 
$X$ by 
\begin{equation}
f(x)= \frac{P(x)}{Q(x)}  \qquad (x \in X \setminus Q^{-1}(0)). 
\end{equation}
Let us set $I(f)=P^{-1}(0) \cap Q^{-1}(0) \subset X$. 
If $P$ and $Q$ are coprime in the local ring 
$\O_{X,x}$ at a point $x \in X$, then 
$I(f)$ is nothing but the set of the indeterminacy 
points of $f$ on a neighborhood of $x$. Note that 
the set $I(f)$ depends on the pair $(P(x), Q(x))$ of 
holomorphic functions representing $f(x)$. 
For example, if we take a holomorphic function 
$R(x)$ on $X$ (which is not identically zero on 
each connected component of $X$) and set 
\begin{equation}
g(x)= \frac{P(x)R(x)}{Q(x)R(x)}  \qquad (x \in X), 
\end{equation}
then the set $I(g)=I(f) \cup R^{-1}(0)$ might be 
bigger than $I(f)$. In this way, we 
distinguish $f(x)= \frac{P(x)}{Q(x)}$ from 
$g(x)= \frac{P(x)R(x)}{Q(x)R(x)}$ even if 
their values coincide over an open dense 
subset of $X$. 
This is the convention due to 
Gusein-Zade, Luengo and Melle-Hern\'andez 
\cite{GZ-L-MH} etc. 
Then we have the following 
fundamental theorem due to \cite{G-L-M}. 

\begin{theorem}\label{the-fib} 
{\rm (Gusein-Zade, Luengo and Melle-Hern\'andez 
\cite[Section 1]{GZ-L-MH})}  
For any point $x \in P^{-1}(0)$ there exists 
$\e_0> 0$ such that for any $0< \e < \e_0$ 
and the open ball $B(x; \e ) \subset X$ of 
radius $\e >0$ with center at $x$ 
(in a local chart of $X$) the restriction 
\begin{equation}
B(x; \e ) \setminus Q^{-1}(0) 
\longrightarrow \CC
\end{equation}
of $f: X \setminus Q^{-1}(0) 
\longrightarrow \CC$ is a locally trivial 
fibration over a sufficiently small 
punctured disk in $\CC$ with center at 
the origin $0 \in \CC$ 
\end{theorem}

We call the fiber in this theorem the Milnor fiber of 
the meromorphic function $f(x)= \frac{P(x)}{Q(x)}$ 
at $x \in P^{-1}(0)$ and denote it by $F_x$. 
As in the holomorphic case, we obtain also 
its Milnor monodromy operators
\begin{equation}
\Phi_{j,x}: H^j(F_x; \CC ) \simto H^j(F_x; \CC ) \qquad 
(j \in \ZZ ). 
\end{equation}
Then we define 
the monodromy zeta function $\zeta_{f,x}(t) \in \CC (t)$ 
of $f$ at $x \in P^{-1}(0)$ by 
\begin{equation}
\zeta_{f,x}(t) = \prod_{j \in \ZZ} 
\Bigl\{ {\rm det}( \id - t \Phi_{j,x}) \Bigr\}^{(-1)^j} 
\quad \in \CC (t). 
\end{equation}

 From now, we shall introduce the formula of 
Gusein-Zade, Luengo and Melle-Hern\'andez in \cite{GZ-L-MH} 
which expresses $\zeta_{f,x}(t) \in \CC (t)$ 
in terms of the Newton polyhedra of $P$ and $Q$ at $x$.  
The problem being local, we may assume that $X$ is an open 
neighborhood of the origin $0 \in \CC^n$ in 
$\CC^n$ and $x=0 \in I(f)=P^{-1}(0) \cap Q^{-1}(0)$. 
Let $\Gamma_+(P), \Gamma_+(Q) \subset \RR^n_+$ be 
the Newton polyhedra of $P$ and $Q$ at 
the origin $0 \in X \subset \CC^n$ and consider 
their Minkowski sum 
\begin{equation}
\Gamma_+(f):= \Gamma_+(P) + \Gamma_+(Q) \quad 
\subset \RR^n_+. 
\end{equation}
Now, let $\Sigma_f$, 
$\Sigma_P$ and $\Sigma_Q$ 
be the dual fans of $\Gamma_{+}(f)$, 
$\Gamma_{+}(P)$ and $\Gamma_{+}(Q)$ 
in $(\RR^n)_+^* \simeq \RR^n_+$ respectively. Then the 
dual fan $\Sigma_f$ of 
the Minkowski sum 
$\Gamma_+(f)= \Gamma_+(P) + \Gamma_+(Q)$ 
is the coarsest common 
subdivision of $\Sigma_P$ and $\Sigma_Q$. 
This implies that for each face 
$\gamma \prec \Gamma_{+}(f)$ 
we have the corresponding faces 
\begin{equation}
\gamma(P) \prec \Gamma_+(P), \qquad 
\gamma(Q) \prec \Gamma_+(Q)
\end{equation}
such that 
\begin{equation}
\gamma = \gamma(P) + \gamma(Q). 
\end{equation}

\begin{definition}\label{def-3}
{\rm (Gusein-Zade, Luengo and Melle-Hern\'andez \cite{GZ-L-MH})} 
We say that the meromorphic 
function $f(x)= \frac{P(x)}{Q(x)}$ is 
non-degenerate at the origin $0 \in X \subset \CC^n$ if for any 
compact face $\gamma$ of $\Gamma_{+}(f)$ the complex 
hypersurfaces $\{x \in T=(\CC^*)^n\ |\ 
P^{\gamma(P)}(x)=0\}$ and $\{x \in T=(\CC^*)^n\ |\ 
Q^{\gamma(Q)}(x)=0\}$ are smooth and reduced 
and intersect transversally in $T=(\CC^*)^n$. 
\end{definition}
For a subset 
$S \subset \{ 1,2, \ldots, n \}$ let 
$\RR^S \simeq \RR^{\sharp S}$ etc. be 
as in Section \ref{section 3} and set 
\begin{equation}
\Gamma_+^S(f):= \Gamma_+(f) \cap \RR^S 
\quad \subset \RR^S_+. 
\end{equation}
Similarly, we define 
$\Gamma_+^S(P), \Gamma_+^S(Q) \subset \RR^S_+$ 
so that we have 
\begin{equation}
\Gamma_+^S(f)= \Gamma_+^S(P) + \Gamma_+^S(Q). 
\end{equation}
For a non-empty subset 
$S \subset \{ 1,2, \ldots, n \}$
let $\gamma_1^S, \gamma_2^S, \ldots, \gamma_{n(S)}^S$ 
be the compact facets of $\Gamma_+^S(f)$ and for 
each $\gamma_i^S$ ($1 \leq i \leq n(S)$) consider 
the corresponding faces 
\begin{equation}
\gamma_i^S(P) \prec \Gamma_+^S(P), \qquad 
\gamma_i^S(Q) \prec \Gamma_+^S(Q)
\end{equation}
such that 
\begin{equation}
\gamma_i^S= \gamma_i^S(P) + \gamma_i^S(Q). 
\end{equation}
By using the primitive inner conormal vector 
$u_i^S \in \ZZ_+^S \setminus \{ 0 \}$ 
of the facet $\gamma_i^S \prec \Gamma_+^S(f)$ 
we define the lattice distance 
$d_i^S(P)>0$ (resp. $d_i^S(Q)>0$) of 
$\gamma_i^S(P)$ (resp. $\gamma_i^S(Q)$) from 
the origin $0 \in \RR^S$ and set 
\begin{equation}
d_i^S=d_i^S(P)-d_i^S(Q) \in \ZZ. 
\end{equation}
Finally by using 
the normalized ($\sharp S-1$)-dimensional volume 
$\Vol_{\ZZ}(\ \cdot\ )$ we set 
\begin{equation}
K_i^S = 
\dsum_{k=0}^{\sharp S-1} 
\Vol_{\ZZ}(
\underbrace{\gamma_i^S(P),
\ldots,\gamma_i^S(P)}_{\text{
$k$-times}},
\underbrace{\gamma_i^S(Q),
\ldots,\gamma_i^S(Q)}_{\text{($\sharp S-1-k$)-times}})
\in \ZZ.
\end{equation}

\begin{theorem}\label{the-2} 
{\rm (Gusein-Zade, Luengo and Melle-Hern\'andez 
\cite[Theorem 3]{GZ-L-MH})} 
Assume that the meromorphic 
function $f(x)= \frac{P(x)}{Q(x)}$ is 
non-degenerate at the origin $0 \in X \subset \CC^n$. 
Then we have 
\begin{equation}
\zeta_{f,0}(t) = \prod_{S \not= \emptyset} \ \Bigl\{ 
\prod_{i: \ d_i^S >0} 
( 1- t^{d_i^S})^{(-1)^{\sharp S-1} K_i^S} \Bigr\}. 
\end{equation}
\end{theorem}

Now we return to the general case where $X$ is 
a complex manifold and extend the classical 
notion of nearby cycle functors to 
meromorphic functions as follows 
(see also Raibaut \cite{Raibaut-2} for a 
similar but sligthly different approach to them). 
For the function $f= \frac{P}{Q}: 
X \setminus Q^{-1}(0) \longrightarrow \CC$ 
we consider the (not necessarily) closed embedding 
\begin{equation}
i_f: X \setminus Q^{-1}(0) \hookrightarrow 
X \times \CC_t \qquad (x \longmapsto (x,f(x))).  
\end{equation}
Let $t: X \times \CC \rightarrow \CC$ be the 
second projection. 

\begin{definition}
{\rm (Nguyen-Takeuchi \cite{N-T})} For 
$\F \in \Db(X)$ we set 
\begin{equation}
\psi_f^{\mer}( \F ):= \psi_t( {\rm R} i_{f*}
( \F |_{X \setminus Q^{-1}(0)}) ) 
\qquad \in \Db(X). 
\end{equation}
We call $\psi_f^{\mer}( \F )$ the meromorphic 
nearby cycle complex of $\F$ along $f$. 
\end{definition} 
We thus obtain a functor 
\begin{equation}
\psi_f^{\mer}(  \cdot  ) : 
\Db(X) \longrightarrow \Db(X). 
\end{equation}
We call it the meromorphic 
nearby cycle functor along $f$. 
As in the case where $f$ is holomorphic i.e. 
$f=\frac{P}{1}$, we have the following results. 

\begin{lemma}\label{lem-1} 
{\rm (see \cite[Lemma 2.1]{N-T})} 
\begin{enumerate}
\item[\rm{(i)}] The support of $\psi_f^{\mer}( \F )$ is 
contained in $P^{-1}(0)$. 
\item[\rm{(ii)}] There exists an isomorphism 
\begin{equation}
\psi_f^{\mer}( \F ) \simto 
\psi_f^{\mer}( {\rm R} 
\Gamma_{X \setminus (P^{-1}(0) \cup Q^{-1}(0))} 
( \F )). 
\end{equation}
\item[\rm{(iii)}] For any point $x \in P^{-1}(0)$ 
there exist isomorphisms 
\begin{equation}
H^j (F_x; \F ) \simeq H^j \psi_f^{\mer}( \F )_x 
 \qquad (j \in \ZZ ). 
\end{equation}
Moreover these isomorphisms are 
compatible with the monodromy automorphisms 
on the both sides. 
\end{enumerate}
\end{lemma}

\begin{theorem}\label{the-1} 
{\rm (see \cite[Theorem 2.2]{N-T})} 
\begin{enumerate}
\item[\rm{(i)}] If $\F \in \Db(X)$ is constructible, then 
$\psi_f^{\mer}( \F ) \in \Db(X)$ is also constructible. 
\item[\rm{(ii)}] If $\F \in \Db(X)$ is perverse, then 
$\psi_f^{\mer}( \F )[-1] \in \Db(X)$ is also perverse. 
\end{enumerate}
\end{theorem}
By this theorem we obtain a functor 
\begin{equation}
\psi_f^{\mer}(  \cdot ) : 
\Dbc (X) \longrightarrow \Dbc (X). 
\end{equation}

\begin{remark}\label{rem-1} 
Assume that the meromorphic function $f(x)= \frac{P(x)}{Q(x)}$ 
is holomorphic on a neighborhood of a point 
$x \in X$ i.e. there exists a holomorphic function 
$g(x)$ defined on a neighborhood of 
$x \in X$ such that $P(x)=Q(x) \cdot g(x)$ on it. 
Then by identifying $f(x)$ with the 
holomorphic function $g(x)$, 
we have an isomorphism 
\begin{equation}
\psi_f^{\mer}( \F )_x \simeq 
\psi_f ( {\rm R} 
\Gamma_{X \setminus Q^{-1}(0)} ( \F ))_x
\end{equation}
for the classical (holomorphic) nearby cycle 
functor $\psi_f( \cdot )$. This implies that 
even if $f(x)$ is holomorphic on 
a neighborhood of $x \in X$ we do not 
have an isomorphism 
\begin{equation}
\psi_f^{\mer}( \F )_x \simeq 
\psi_f ( \F )_x
\end{equation}
in general. 
\end{remark}

The following result is an analogue 
for $\psi_f^{\mer}( \cdot )$ of the 
classical one for $\psi_f( \cdot )$ 
in Proposition \ref{prp.7.1.002}. 

\begin{proposition}\label{PDI} 
{\rm (see \cite[Proposition 2.4]{N-T})} 
Let $\pi : Y \longrightarrow X$ be a proper 
morphism of complex manifolds and 
$f \circ \pi$ a meromorphic function on $Y$ defined by 
\begin{equation}
f \circ \pi = \frac{P \circ \pi}{Q \circ \pi}. 
\end{equation}
Then for $\G \in \Db(Y)$ 
there exists an isomorphism 
\begin{equation}
\psi_{f}^{\mer} ( {\rm R} \pi_* \G ) \simeq 
{\rm R} \pi_* \psi_{f \circ \pi}^{\mer} ( \G ). 
\end{equation}
If moreover $\pi$ induces an isomorphism 
\begin{equation}
Y \setminus \pi^{-1}(P^{-1}(0) \cup Q^{-1}(0)) 
\simto X \setminus (P^{-1}(0) \cup Q^{-1}(0)), 
\end{equation}
then for $\F \in \Db(X)$ 
there exists an isomorphism 
\begin{equation}
\psi_{f}^{\mer} ( \F ) \simeq 
{\rm R} \pi_* \psi_{f \circ \pi}^{\mer} ( \pi^{-1} \F ). 
\end{equation}
\end{proposition}

With Lemma \ref{lem-1} (iii) and 
Propositon \ref{PDI} at hands, as in 
the proofs of the theorems of Varchenko and Oka in 
Section \ref{section 3} one can now reprove Theorem 
\ref{the-2} using meromorphic nearby 
cycle sheaves. We leave it for an exercise of 
the readers. Moreover, thanks to the perversity in 
Theorem \ref{the-1} (ii), we obtain the following result. 
The problem being local, we may assume that $X$ is an open 
neighborhood of the origin $0 \in \CC^n$ in 
$\CC^n$ and $x=0 \in I(f)=P^{-1}(0) \cap Q^{-1}(0)$. 
For $j \in \ZZ$ and $\lambda \in \CC$ 
we denote by 
\begin{equation}
H^j(F_0; \CC )_{\lambda} \subset H^j(F_0; \CC )
\end{equation}
the generalized eigenspace of $\Phi_{j,0}$ 
for the eigenvalue $\lambda$. 

\begin{theorem}\label{the-33} 
{\rm (Nguyen-Takeuchi \cite[Theorem 1.2]{N-T})} 
Assume that the hypersurfaces $P^{-1}(0)$ and 
$Q^{-1}(0)$ of $X \subset \CC^n$ have an isolated 
singular point at the origin $0 \in X \subset \CC^n$ 
and intersect transversally on $X \setminus \{ 0 \}$. 
Then for any $\lambda \not=1$ we have 
the concentration 
\begin{equation}
H^j(F_0; \CC )_{\lambda} \simeq 0 
\qquad (j \not= n-1). 
\end{equation}
\end{theorem}
It seems that there is some geometric background on 
$F_0$ (like Milnor's bouquet 
decomposition theorem in 
Theorem \ref{thm.7.2.003}) 
for Theorem \ref{the-33} to hold. 
It would be an interesting problem to find it 
and reprove Theorem \ref{the-33} in a purely 
geometric manner. 
Combining the formula for 
$\zeta_{f,0}(t) \in \CC (t)$ in Theorem 
\ref{the-2}  
with Theorem \ref{the-33}, we obtain a 
formula for the multiplicities of the eigenvalues 
$\lambda \not= 1$ in $\Phi_{n-1,0}$ as follows. 

\begin{corollary}\label{cor-1} 
Assume that $f(x)= \frac{P(x)}{Q(x)}$ is 
non-degenerate at the origin $0 \in X= \CC^n$ 
and $P(x), Q(x)$ are convenient. 
Then in the notations of Theorem \ref{the-2}, 
for any $\lambda \not=1$ 
the multiplicity of the eigenvalue 
$\lambda$ in $\Phi_{n-1, 0}$ is 
equal to that of the factor $t- \lambda$ 
in the rational function 
\begin{equation}
\prod_{S \not= \emptyset} \ \Bigl\{ 
\prod_{i: \  d_i^S >0} 
(t^{d_i^S}-1)^{(-1)^{n- \sharp S} K_i^S} \Bigr\}
\qquad \in \CC (t)^*. 
\end{equation}
\end{corollary}

For the meromorphic function $f(x)= \frac{P(x)}{Q(x)}$ 
and $\F \in \Db(X)$ we set also 
\begin{equation}
\psi_f^{\merc}( \F ):= \psi_t( {\rm R}i_{f!}
( \F |_{X \setminus Q^{-1}(0)}) ) 
\qquad \in \Db(X). 
\end{equation}
In fact, this object was originally defined by 
Raibaut in \cite{Raibaut-2}. 
Here we call it the meromorphic 
nearby cycle complex 
with compact support of $\F$ along $f$. 
Then we obtain a functor 
\begin{equation}
\psi_f^{\merc}( \cdot ) : 
\Db(X) \longrightarrow \Db(X)
\end{equation}
which satisfies the properties similar to 
the ones in  Lemma \ref{lem-1} (i) (ii),  
Theorem \ref{the-1} and Proposition 
\ref{PDI} (see e.g. \cite[Section 2]{Tak-3}). 
Moreover if $f$ 
is holomorphic on a neighborhood of a point 
$x \in X$, then we have an isomorphism 
\begin{equation}
\psi_f^{\merc}( \F )_x \simeq 
\psi_f ( \F_{X \setminus Q^{-1}(0)} )_x
\end{equation}
for the classical (holomorphic) nearby cycle 
functor $\psi_f( \cdot )$. 
However the isomorphism in Lemma \ref{lem-1} (iii) 
does not hold for $\psi_f^{\merc}( \F )$. 
This implies that the natural morphism 
\begin{equation}
\psi_f^{\merc}( \F ) \longrightarrow 
\psi_f^{\mer}( \F )
\end{equation}
is not an isomorphism in general. 

From now on, we assume that 
$f(x)= \frac{P(x)}{Q(x)}$ is 
a rational function on $X= \CC^n$ and 
$x=0 \in P^{-1}(0)$. 
To obtain also a formula for 
the Jordan normal form of its monodromy 
$\Phi_{n-1, 0}$ as in 
Theorems \ref{thm:8-4} and \ref{jordan}, 
we have to assume moreover that 
$f$ is polynomial-like in the following sense. 

\begin{definition}\label{def-2 7}
{\rm (Nguyen-Takeuchi \cite[Definition 1.3]{N-T})}
We say that the rational function 
$f(x)= \frac{P(x)}{Q(x)}$ is polynomial-like 
if there exists a 
resolution $\pi_0: \widetilde{X} \rightarrow X= \CC^n$
of singularities of $P^{-1}(0) \cup Q^{-1}(0)$ 
which induces an isomorphism 
$\widetilde{X} \setminus \pi_0^{-1}( \{ 0 \} ) 
\simto X \setminus \{ 0 \}$ such that for 
any irreducible component $D_i$ of the 
(exceptional) normal crossing divisor 
$D= \pi_0^{-1}( \{ 0 \} )$ we have the condition 
\begin{equation}
{\rm ord}_{D_i} (P \circ \pi_0) > 
{\rm ord}_{D_i} (Q \circ \pi_0). 
\end{equation}
\end{definition}
Since we assume here that $0 \in P^{-1}(0)$, 
if $Q=1$ i.e. $f= \frac{P}{1}$ is a polynomial 
on $X= \CC^n$ then it is polynomial-like in 
the sense of Definition \ref{def-2 7}. 
Let us give more examples of polynomial-like 
rational functions 
$f(x)= \frac{P(x)}{Q(x)}$. 

\begin{definition}\label{def-29}
{\rm (Nguyen-Takeuchi \cite[Definition 5.13]{N-T})}
We say that the Newton polyhedron
$\Gamma_+(P)$ is properly contained in the one
$\Gamma_+(Q)$ if for any 
$u \in \Int ( \RR^n)^*_+$ we have 
\begin{equation}
\min_{v \in \Gamma_+(P)} \langle u, v \rangle
>
\min_{v \in \Gamma_+(Q)} \langle u, v \rangle.
\end{equation}
In this case, we write $\Gamma_+(P) \subset
\subset \Gamma_+(Q)$.
\end{definition}
Since we assume here that $0 \in P^{-1}(0)$, obviously 
if $Q=1$ i.e. $f= \frac{P}{1}$ is a polynomial 
on $X= \CC^n$ then we have $\Gamma_+(P) \subset
\subset \Gamma_+(Q)$. 

\begin{lemma}\label{lem-pp-pl}
Assume that the rational function
$f(x)= \frac{P(x)}{Q(x)}$ is 
non-degenerate at the origin $0 \in X= \CC^n$, 
$P(x), Q(x)$ are convenient, and $\Gamma_+(P) \subset
\subset \Gamma_+(Q)$. Then $f(x)= \frac{P(x)}{Q(x)}$ 
is polynomial-like. 
\end{lemma} 

\begin{proof}
First of all, by our assumptions and 
Lemma \ref{leme:3-5-ad} the complex 
hypersurfaces $P^{-1}(0), Q^{-1}(0) \subset 
X= \CC^n$ have (at worst) isolated singular points at 
the origin $0 \in X= \CC^n$. 
Let us consider $X= \CC^n$ as the toric variety associated to 
the fan $\Sigma_0$ formed by the all faces of the first 
quadrant $(\RR^n)^*_+ \subset (\RR^n)^*$. Recall that 
we denote by $\Sigma_f$ the dual fan of 
$\Gamma_+(f)= \Gamma_+(P) + \Gamma_+(Q)$ in $(\RR^n)^*_+$. 
Then by our assumptions, we can easily see that 
any cone $\sigma \in \Sigma_0$ in $\Sigma_0$ such 
that $\sigma \subset \partial (\RR^n)^*_+ 
\simeq ( \RR^n_+ \setminus {\rm Int} (\RR^n_+) ) $ is 
an element of $\Sigma_f$. Hence we can construct a 
smooth subdivision $\Sigma$ of $\Sigma_f$ without 
subdividing the cones $\sigma \in \Sigma_f$ such 
that $\sigma \subset \partial (\RR^n)^*_+$ 
(see e.g. \cite[Chapter II, Lemma (2.6)]{Oka}). 
Let $X_{\Sigma}$ be the smooth toric variety associated 
to $\Sigma$. Then there exists a morphism 
$\pi_0: X_{\Sigma} \rightarrow X= \CC^n$
of toric varieties 
inducing an isomorphism 
$X_{\Sigma} \setminus \pi_0^{-1}( \{ 0 \} ) 
\simto X \setminus \{ 0 \}$. This is a 
resolution of singularities of $P^{-1}(0) \cup Q^{-1}(0)$ 
and by using our assumption $\Gamma_+(P) \subset
\subset \Gamma_+(Q)$ we can easily check the 
conditions in Definition \ref{def-2 7}.  
\qed 
\end{proof}
For a perverse sheaf $\F \in \Dbc (X)$ on $X$, as 
in the holomorphic case there exist direct sum 
decompositions 
\begin{equation}
\psi_f^{\mer}( \F ) \simeq \bigoplus_{\lambda \in \CC} \ 
\psi_{f, \lambda}^{\mer} ( \F ), 
\qquad 
\psi_f^{\merc}( \F ) \simeq \bigoplus_{\lambda \in \CC} \ 
\psi_{f, \lambda}^{\merc}( \F ).
\end{equation} 

\begin{proposition}\label{pro-109}
{\rm (Nguyen-Takeuchi \cite[Proposition 4.2]{N-T})} 
Assume that the rational function 
$f(x)= \frac{P(x)}{Q(x)}$ is polynomial-like 
and satisfies the 
conditions in Theorem \ref{the-33}. 
Then for any $\lambda \not=1$ the natural 
morphism 
\begin{equation}
\psi_{f, \lambda}^{\merc}
( \CC_X [n] )[-1] \longrightarrow 
\psi_{f, \lambda}^{\mer}
( \CC_X [n] )[-1] 
\end{equation}
is an isomorphism of perverse sheaves. Morevoer, the supports 
of the both sides of it 
are contained in 
the origin $\{ 0 \} \subset X= \CC^n$. 
\end{proposition}

Now we shall introduce a filtration on 
$H^{j}(F_{0};\CC)_{\lambda}$ 
($j \in \ZZ$, $\lambda \in \CC$). 
For a variety $Z$ over $\CC$ we 
denote by ${\rm MHM}_{Z}$ the abelian category of 
mixed Hodge modules on $Z$ (see e.g. 
\cite[Section 8.3]{H-T-T}). 
First we regard $\psi_{f}^{\mathrm{mero}}
(\CC_{X}[n])[-1]$ 
as the underlying perverse sheaf of the mixed 
Hodge module $\psi_{t}^{H}({i_{f}}_{*}
(\CC^{H}_{X}[n]|_{X\setminus Q^{-1}(0)}))$ $\in {\rm MHM}_{X}$, 
where $\psi_{t}^{H}$ (resp. ${i_{f}}_{*}$) 
is a functor between the categories 
of mixed Hodge modules corresponding to the functor 
$\psi_{t}[-1]$ (resp. ${\rm R} {i_{f}}_{*}$) 
and $\CC_{X}^{H}[n] \in 
{\rm MHM}_{X}$ is the 
mixed Hodge module whose underlying perverse sheaf 
is $\CC_{X}[n]$.
For the inclusion map $j_{0}\colon \{0\}\hookrightarrow X$, 
we consider the pullback 
\begin{equation}
j_{0}^* 
\psi_{t}^{H}
({i_{f}}_{*}(\CC^{H}_{X}[n]|_{X\setminus Q^{-1}(0)})) \  
\in \Db ( {\rm MHM}_{\{ 0 \}} ) 
\end{equation}
by $j_{0}$, whose underlying constructible sheaf is 
$j_{0}^{-1}( \psi_{f}^{\mathrm{mero}}(\CC_{X}[n])[-1]$. 
Since the $(j-n+1)$-th cohomology group of $j_{0}^{-1}( 
\psi_{f}^{\mathrm{mero}}(\CC_{X}[n])[-1]$ 
is $H^{j}(F_{0};\CC)$, we thus obtain a 
mixed Hodge structure of $H^{j}(F_{0}; 
\CC)$. Restricting its weight filtration to 
$H^{j}(F_{0};\CC)_{\lambda} \subset H^{j}(F_{0};\CC)$ 
we obtain a filtration of $H^{j}(F_{0};\CC)_{\lambda}$. 
Then by Proposition \ref{pro-109} 
we can prove the following result (see 
\cite[Theorem 4.4]{N-T} for the precise proof).  

\begin{theorem}\label{weifilmonofil}
{\rm (Nguyen-Takeuchi \cite[Theorem 4.4]{N-T})} 
In the situation of Proposition \ref{pro-109}, for any 
$\lambda\neq 1$ 
the filtration of $H^{n-1}(F_{0};
\CC)_{\lambda}$ is the monodromy 
filtration of the Milnor monodromy 
$\Phi_{n-1, 0}\colon H^{n-1}(F_{0};
\CC)_{\lambda}\overset{\sim}{\to} H^{n-1}(
F_0 ;\CC)_{\lambda}$ centered at $n-1$.
\end{theorem}

Then in the situation of this theorem, forgetting the eigenvalue $1$ 
parts of the Milnor monodromies $\Phi_{j, 0}$ 
we can define the reduced Hodge spectrum 
$\tl{\rm sp}_{f,0}(t)$ of 
the Milnor fiber $F_0$ of $f$ at the origin 
$0 \in X= \CC^n$ 
(see \cite[Definition 5.11]{N-T} for 
the definition) 
which satisfies the symmetry 
\begin{equation}
\tl{\rm sp}_{f,0}(t) = t^n \cdot
\tl{\rm sp}_{f,0} \Bigl( \frac{1}{t} \Bigr)
\end{equation}
centered at $\frac{n}{2}$. On the other 
hand, in \cite{Raibaut-2} Raibaut defined 
the motivic zeta function 
of the rational function $f$ and the 
motivic Milnor fiber as its limit. Later 
in \cite[Section 5]{N-T}, 
the same objects were defined via different methods. 
Moreover in \cite[Section 5]{N-T}, 
assuming that $f$ is non-degenerate 
at the origin $0 \in X= \CC^n$ and $P(x), Q(x)$ 
are convenient, as in Theorem \ref{thm:8-3} Nguyen and 
the author described the motivic Milnor fiber 
of $f= \frac{P}{Q}$ in terms of 
the Newton polyhedra of $P$ and $Q$. Then 
in the case where $f$ is polynomial-like,  
we can obtain 
combinatorial descriptions of the reduced Hodge spectrum 
$\tl{\rm sp}_{f,0}(t)$ of $F_0$ and the Jordan 
normal forms of $\Phi_{n-1, 0}$ for 
the eigenvalues $\lambda \not= 1$ as in 
Theorems \ref{thm:7-19}, \ref{thm:8-4} and \ref{jordan}. 
See \cite[Section 6]{N-T} for the details. 
We can also globalize 
these results and obtain 
similar formulas for monodromies at infinity of 
the rational function 
$f(x)= \frac{P(x)}{Q(x)}$. They are natural 
generalizations of the results in Sections 
\ref{section 4} and \ref{section 7}.  
See \cite[Section 7]{N-T} for the details. 

By using the meromorphic nearby cycle functors, 
recently in the paper \cite{Tak-2} Bernstein-Sato 
polynomials i.e. b-functions 
for meromorphic functions 
were defined and their basic properties 
were studied. In particular, 
a Kashiwara-Malgrange type 
theorem for the Milnor monodromies was proved 
in \cite[Theorem 1.4]{Tak-2}. Also in 
the recent paper \cite{A-G-L-N}, similar but 
different b-functions for meromorphic functions were 
introduced and studied. Unlike the motivation in 
\cite{Tak-2}, the authors of \cite{A-G-L-N} 
applied their b-functions to the analytic continuations of 
local zeta functions associated to meromorphic functions. 
Moreover in \cite{A-G-L-N} and \cite{Tak-2},  
the authors of them introduced 
multiplier ideal sheaves for meromorphic functions 
and proved that their jumping numbers are 
related to their b-functions. This is an analogue for  
meromorphic functions of the celebrated theorem of 
\cite{E-L-S-V}.


\end{document}